\documentclass[11pt,reqno]{article}
\usepackage{amsmath, amsthm, amssymb, amsfonts}
\usepackage{bbm}
\usepackage{bm}
\usepackage{graphicx}
\usepackage{xcolor}
\usepackage{hyperref}
\usepackage{enumitem}
\usepackage{cleveref}

\usepackage{multicol}
\usepackage[
  top=3cm,
  bottom=3cm,
  left=2cm,
  right=2cm
]{geometry}
\usepackage{lipsum}
\usepackage{mwe}

\usepackage{graphicx}
\usepackage{caption}
\usepackage{subcaption}
\usepackage{float}

\usepackage[linesnumbered,ruled,vlined]{algorithm2e}

\theoremstyle{plain}

\usepackage[
  backend=biber,
  style=numeric,
  giveninits=true,
  maxbibnames=99
]{biblatex}

\numberwithin{equation}{section}

\newtheorem{theorem}{Theorem}[section]

\theoremstyle{definition}

\newtheorem*{remark}{Remark}

\newcommand{\K}{\mathcal{K}}

\newcommand{\R}{\mathbb{R}}

\newcommand{\Z}{\mathbb{Z}}

\newcommand{\N}{\mathbb{N}}

\let\strokeL\L
\DeclareRobustCommand{\L}{\ifmmode\mathbf{L}\else\strokeL\fi}

\makeatletter
\DeclareFontFamily{OMX}{MnSymbolE}{}
\DeclareSymbolFont{MnLargeSymbols}{OMX}{MnSymbolE}{m}{n}
\SetSymbolFont{MnLargeSymbols}{bold}{OMX}{MnSymbolE}{b}{n}
\DeclareFontShape{OMX}{MnSymbolE}{m}{n}{
    <-6>  MnSymbolE5
   <6-7>  MnSymbolE6
   <7-8>  MnSymbolE7
   <8-9>  MnSymbolE8
   <9-10> MnSymbolE9
  <10-12> MnSymbolE10
  <12->   MnSymbolE12
}{}
\DeclareFontShape{OMX}{MnSymbolE}{b}{n}{
    <-6>  MnSymbolE-Bold5
   <6-7>  MnSymbolE-Bold6
   <7-8>  MnSymbolE-Bold7
   <8-9>  MnSymbolE-Bold8
   <9-10> MnSymbolE-Bold9
  <10-12> MnSymbolE-Bold10
  <12->   MnSymbolE-Bold12
}{}

\let\llangle\@undefined
\let\rrangle\@undefined
\DeclareMathDelimiter{\llangle}{\mathopen}%
                     {MnLargeSymbols}{'164}{MnLargeSymbols}{'164}
\DeclareMathDelimiter{\rrangle}{\mathclose}%
                     {MnLargeSymbols}{'171}{MnLargeSymbols}{'171}
\makeatother

\DeclareFontShape{OT1}{cmr}{bx}{sc}{<-> cmbcsc10}{}

\title{A Spectral Model-Informed Neural Network for Inverse Source Problems}
 \author{Dinh-Liem Nguyen\thanks{Department of Mathematics, Kansas State University, Manhattan, KS 66506 (\href{mailto:dlnguyen@ksu.edu}{dlnguyen@ksu.edu})} \and  Nhung H. Nguyen\thanks{Department of Mathematics, Kansas State University, Manhattan, KS 66506 (\href{mailto:nhungnh@ksu.edu}{nhungnh@ksu.edu})} \and  Aravinth Ravi\thanks{Department of Mathematics, Kansas State University, Manhattan, KS 66506  (\href{mailto:arav0006@ksu.edu}{arav0006@ksu.edu})}} 
\begin{document}
\date{}
\maketitle

\begin{abstract}

In this paper, we propose an unsupervised, two-step model-informed deep learning framework for solving the inverse source problem. In the first step, boundary measurement data are transformed into an imaging function that encodes information about the shape and location of the unknown source. Leveraging this information, we derive a Fourier-based model equation in the spectral domain that relates the imaging function to the Fourier coefficients of the source function. In the second step, this model equation is incorporated into the training of a model-informed neural network to recover the parameters of interest. The proposed approach enables fast and accurate reconstruction of the source function while maintaining robustness to noise. The effectiveness of the method is demonstrated through numerical experiments in both two- and three-dimensional settings. For the two-dimensional case, we further compare our approach with the traditional least-squares method to validate its computational efficiency and reconstruction accuracy.
\end{abstract}


\section{Introduction}
Let $\Omega$ be a bounded Lipschitz domain in $\R^n$, where  $n=2, 3$, and $f \in L^2(\R^n)$ be a real-valued function with compact support in $\Omega$. For a source characterized by $f$,  we consider the following model problem
\begin{align}\label{eq:PDE}
    & \Delta u + k^{2}u  = f,\hspace{2.05 cm} \text{in } \Omega \subset \mathbb{R}^{n}, \\ \label{eq:radcon}
    &  \frac{\partial u}{\partial r} -iku  = \mathcal{O}\left(r^{-\frac{n-1}{2}}\right), \quad r=|x|\to \infty,
\end{align}
where $u$ is the radiated wave with wave number $k>0$, and~\eqref{eq:radcon} is the Sommerfeld radiation condition describing the outgoing behavior of  $u$. The direct problem of determining $u$  satisfying~\eqref{eq:PDE}-\eqref{eq:radcon}  given $f$ and $k$ is well-posed~\cite{Kirsch1996AnIT}. Let  $\nu(x)$ be the outward normal unit vector to $\partial\Omega$ at $x$ and $0<\underline{k} <\overline{k}$.  In this paper, we are interested in the following inverse problem.

\vspace{0.1cm}
\noindent
\textbf{Inverse problem:} Determine $f$, given Cauchy data $u$ and $\partial u /\partial \nu$ on $\partial  \Omega$ for  $k\in [\underline{k},\overline{k}]$. 
\vspace{0.1cm}

The uniqueness of the solution to the inverse source problem has been established in \cite{Bao2010AMI}. 
Inverse source problems arise in numerous applications in physics and engineering, including antenna synthesis and medical imaging. Theoretical aspects of these problems, particularly uniqueness and stability, have been extensively studied over the past two decades \cite{Acosta2012OnTM,Bao2010AMI,Cheng2016IncreasingSI,Badia2013StabilityEF,Li2020LipschitzSF}.

A substantial body of literature is also devoted to the development of numerical methods for inverse source problems. For the multi-frequency case, various reconstruction approaches have been proposed; see \cite{Alzaalig2017FastAS,Bao2015ARA,Eller2009AcousticSI,Griesmaier2017AFM,Ji2019InverseES,Nguyen2019ACN,Zhang2015FourierMF,nguyen2022reconstructing} and the references therein. In the single-frequency setting, numerous numerical techniques have likewise been developed; see \cite{harris2024direct,harris2023reconstruction, badia2011inverse,nara2008algebraic,zhang2018locating} and the references therein.

In recent years, there has been rapid growth in the application of deep learning algorithms to inverse problems; see \cite{Sun2024TheLR, Chen2020ARO,Khoo2018SwitchNetAN,Chen2019PhysicsinformedNN,Yin2020ANN,Colibazzi2023DeepplugandplayPG, Ning2024ADS, Nguyen2024TAENAM,le2022sampling,Pokkunuru2023ImprovedTO,bao2025pfwnn, Zhang2022SolvingTW,Zhang2023SolvingAI,Zhou2022ANN,Jagtap2020ConservativePN,nguyen2024tnet, ding2022coupling} and references therein. Most existing approaches rely on supervised learning frameworks that require large labeled datasets for training. These data-driven methods strongly depend on the distribution of the training data. While they often achieve high-resolution reconstructions when the true parameters lie within this distribution, their performance may deteriorate for out-of-distribution inputs. Moreover, only a limited number of studies have explored fully unsupervised learning strategies for inverse problems. In addition, several existing approaches \cite{Chen2019PhysicsinformedNN, Zhang2023SolvingAI} depend on internal or auxiliary data that may not always be available in practical scenarios. These limitations motivate the development of unsupervised deep learning methods for inverse problems.

In this paper, we propose an unsupervised two-step Fourier-based method for solving multifrequency inverse source problems that require only boundary measurement data. Our approach is inspired by Physics-Informed Neural Networks (PINNs) \cite{Raissi2019PhysicsinformedNN,Cuomo2022ScientificML,Raissi2017HiddenPM}, where physical models are incorporated into the training process through physics-based loss terms. The proposed framework integrates an imaging function with a model-based deep learning strategy.
In Step 1, the imaging function provides a stable geometric estimate (shape and location) of the unknown source. This imaging formulation can be rewritten as an equation relating the imaging data to the source function $f$. Exploiting the convolutional structure of this relation, we apply the Fourier transform to derive a model equation, which is incorporated into the neural network as a physics-based loss term. In Step 2, we construct the loss function using both the imaging data and the derived model equation, and train a fully connected neural network to efficiently learn the Fourier coefficients of $f$.

The proposed algorithm enables stable and efficient reconstruction of the source function. By learning only the Fourier coefficients rather than representing the unknown function over the entire computational domain, the number of trainable parameters is significantly reduced. Moreover, transforming the convolutional structure into pointwise multiplication in Fourier space lowers computational complexity. Importantly, the method does not require paired input–output training data and therefore avoids the generalization limitations inherent in supervised learning approaches.

The remainder of this paper is organized as follows. In Section \ref{section2}, we introduce the imaging function and present its theoretical analysis. Section \ref{section3} is devoted to the computation of the relevant Fourier coefficients and the derivation of the corresponding Fourier-based model equation. Building upon these results, Section \ref{section4} describes the proposed model-based neural network framework and the associated optimization procedure. Numerical experiments using simulated data are presented in Section \ref{section5}. Finally, conclusions are drawn in Section \ref{section6}.

\section{A stable imaging function }\label{section2}

In this section, we present a stable imaging function, $\mathcal{I}_k(z)$, constructed from the inverse problem data. This function is directly related to the unknown source function $f$ through a simple equation and also provides an estimate of the support of $f$. The imaging function $\mathcal{I}_{k}(z)$ was first introduced in \cite{le2022sampling} for the reconstruction of the shape and location of scattering objects. We first need the the volume potential for the radiated wave $u$. It is  well known that  the solution to problem \eqref{eq:PDE}-\eqref{eq:radcon} is given by  
\begin{align}\label{eq:Volume}
    u(x) = \int_{\Omega} \Phi_k(x,y) f(y) dy,\quad x\in \R^n,
\end{align}
where $\Phi_k$ is the Green's function of problem \eqref{eq:PDE}-\eqref{eq:radcon} given by
\[
\Phi_k(x,y) =
\begin{cases}
\displaystyle 
\frac{i}{4} H_0^{(1)}\!\bigl(k\lvert x - y\rvert\bigr), & n = 2, \\[8pt]
\displaystyle 
\frac{\exp\bigl(ik\lvert x - y\rvert\bigr)}{4\pi \lvert x - y\rvert}, & n = 3,
\end{cases}
\]
for $x, y \in \R^{n}$ and $x\neq y$. Here $H^{(1)}_0$ is the Hankel function of the first kind and zero order.

For $z\in \R^n$, the  imaging function is defined by
$$\mathcal{I}_{k}(z) = \int_{\partial \Omega} \left(\dfrac{\partial \operatorname{Im}\Phi_k(x,z)}{\partial\nu(x)}u(x,k) -  \operatorname{Im}\Phi_k(x,z)\dfrac{\partial u(x,k)}{\partial\nu(x)}\right) dS(x).$$
The next theorem partly explains the important behavior of  $\mathcal{I}_{k}(z)$ that it is relatively large when $z \in \text{supp}(f)$ and that it is small when $z \notin \text{supp}(f)$.

\begin{theorem}\label{thm:imaging}
For $z \in \R^n$,    the imaging function satisfies
    \begin{equation}\label{eq:Iz}
        \mathcal{I}_{k}(z) =  \int_{\Omega} \operatorname{Im}\Phi_k(z-y)f(y) dy , 
    \end{equation}
    where
    $$\operatorname{Im}\Phi_k(x) = \begin{cases}
        \frac{1}{4}J_0(k|x|),\qquad  n=2, \\ \notag
        \frac{k}{4\pi} j_0(k|x|),\hspace{0.65 cm}  n=3.
    \end{cases}$$
\end{theorem}

\begin{proof}
    The proof uses a similar argument given in \cite{le2022sampling}. By the volume integral \eqref{eq:Volume},  
    \begin{equation}\label{proof:Iz}
        \mathcal{I}_{k}(z)=\int_{\Omega}\int_{\partial \Omega}\left(\dfrac{\partial \operatorname{Im}\Phi_k(x,z)}{\partial\nu(x)}\Phi_k(x,y) -  \operatorname{Im}\Phi_k(x,z)\dfrac{\partial \Phi_k(x,y)}{\partial\nu(x)}\right)dS(x) f(y)dy.
    \end{equation}
    Since $\Delta \operatorname{Im}\Phi_k(z,y)+k^2\operatorname{Im}\Phi_k(z,y)=0$ for all $y,z\in \R^n$ and $\operatorname{Im}\Phi_k$ is a regular function, the Green's integral representation implies that
    $$\int_{\partial \Omega}\left(\dfrac{\partial \operatorname{Im}\Phi_k(x,z)}{\partial\nu(x)}\Phi_k(x,y) -  \operatorname{Im}\Phi_k(x,z)\dfrac{\partial \Phi_k(x,y)}{\partial\nu(x)}\right)dS(x) =\operatorname{Im}\Phi_k(z,y).$$
    Substituting this identity into \eqref{proof:Iz} yields \eqref{eq:Iz}.
\end{proof}
It is well-known that $J_0(k|y-z|)$ and $j_0(k|y-z|)$ have a significant peak when $z=y$ and decays rapidly when $z$ moves away from $y$. This key behavior and Theorem \ref{thm:imaging} provide a justification to the resolution of the imaging function.

As shown in the figures below, the imaging function offers a preliminary estimate of the unknown source's location and shape, which guides the application of the periodization technique in Section \ref{section3}.

\begin{figure}[H]
    \centering
    \begin{subfigure}[b]{0.4\textwidth}
         \centering
        \includegraphics[width=\textwidth]{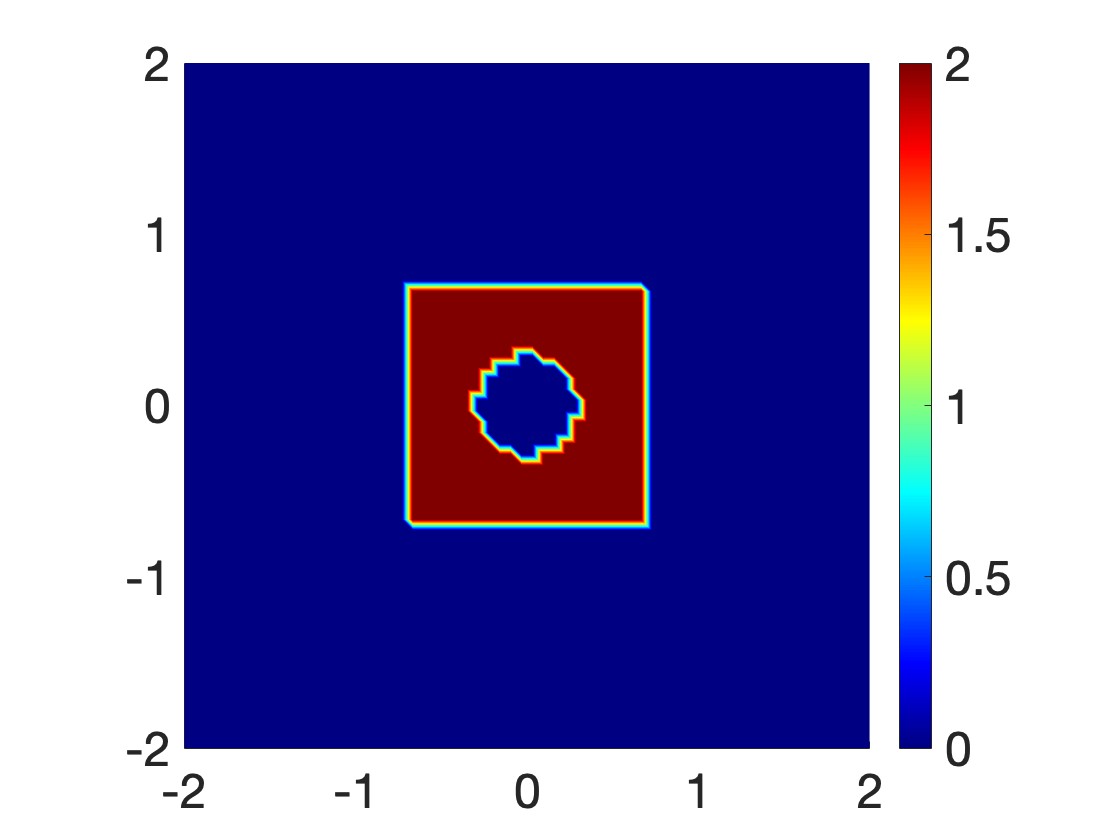}
        \caption{Source $f$}
        
    \end{subfigure}
    \begin{subfigure}[b]{0.4\textwidth}
             \centering
         \includegraphics[width=\textwidth]{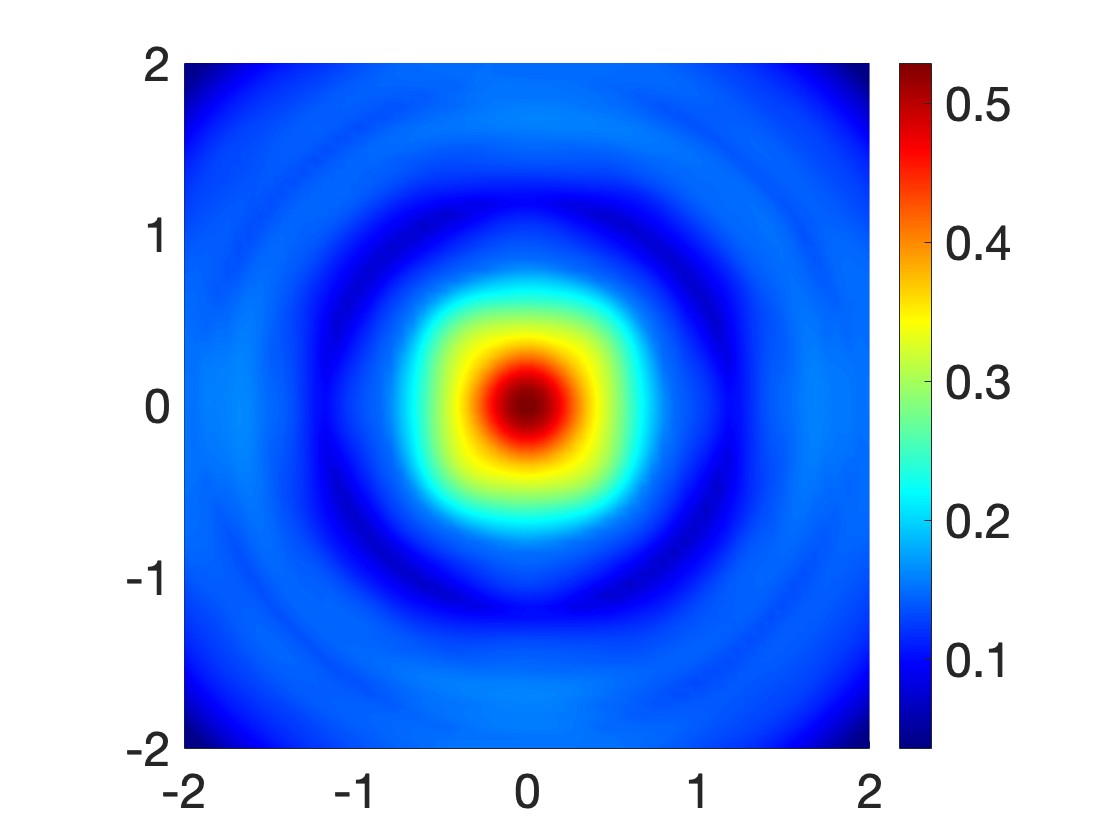}
        \caption{$\sum_{k} |\mathcal{I}_{k}|$ ($10$ evenly distributed $k$ in $[2,10]$) }
        \end{subfigure}
    \caption{Source $f$  and a preliminary estimate of its location and shape using $\mathcal{I}_{k}$.}
\end{figure}

The next theorem establishes the stability of the imaging function we consider.
\begin{theorem}  For $\delta>0$,
denote by $u_\delta(\cdot,k)$ the noisy data  satisfying 
\begin{align}\label{ineq:stability1}
\| u(\cdot,k) - u_\delta (\cdot,k)\|_{L^2(\partial \Omega)}  \leq \delta,
\end{align}
and 
\begin{align}\label{ineq:stability2}
\left\| \dfrac{\partial u(\cdot,k)}{\partial\nu} - \dfrac{\partial u_\delta(\cdot,k)}{\partial\nu}\right\|_{L^2(\partial \Omega)}  \leq \delta.
\end{align}
Define $$\mathcal{I}_{k,\delta}(z) = \int_{\partial \Omega} \left(\dfrac{\partial \operatorname{Im}\Phi_k(x,z)}{\partial\nu(x)}u_\delta(x,k) -  \operatorname{Im}\Phi_k(x,z)\dfrac{\partial u_\delta(x,k)}{\partial\nu(x)}\right) dS(x).$$  Then the following stability property holds
$$
|\mathcal{I}_k(z)- \mathcal{I}_{k,\delta}(z) |\leq C \delta, \quad \text{for all } z \in \mathbb{R}^n,
$$
where $C>0$ is independent of $z$ and $\delta$.
\end{theorem}

\begin{proof} For all $z\in \Omega$, by Cauchy-Schwarz inequality, we have
        \begin{align*}
        |\mathcal{I}_k(z)- \mathcal{I}_{k,\delta}(z) | 
        & \leq \int_{\partial \Omega} |u(x, k) - u_{\delta}(x, k)  | \left| \dfrac{\partial \operatorname{Im}\Phi_k(x,z)}{\partial\nu(x)} \right| dS(x) \\ \notag 
        &\hspace{2 cm}+ \int_{\partial \Omega} \left|\dfrac{\partial u(x,k)}{\partial\nu(x)} - \dfrac{\partial u_\delta(x,k)}{\partial\nu(x)}  \right| \left| \operatorname{Im}\Phi_k(x,z) \right| dS(x)\\
        & \leq \| u(\cdot,k) - u_\delta (\cdot,k)\|_{L^2(\partial \Omega)} \left\|\dfrac{\partial \operatorname{Im}\Phi_k(\cdot,z)}{\partial\nu}\right\|_{L^2(\partial \Omega)}\\ \notag
        &\hspace{2 cm}+\left\| \dfrac{\partial u(\cdot,k)}{\partial\nu} - \dfrac{\partial u_\delta(\cdot,k)}{\partial\nu}\right\|_{L^2(\partial \Omega)}\left\|\operatorname{Im}\Phi_k(\cdot,z)\right\|_{L^2(\partial \Omega)} .
    \end{align*}
    We complete the proof by using \eqref{ineq:stability1}, \eqref{ineq:stability2} and denoting $$C:=\dfrac{1}{2}\max\left\{\sup_{z\in\Omega}\left\|\dfrac{\partial \operatorname{Im}\Phi_k(\cdot,z)}{\partial\nu}\right\|_{L^2(\partial \Omega)},\sup_{z\in\Omega}\left\|\operatorname{Im}\Phi_k(\cdot,z)\right\|_{L^2(\partial \Omega)}\right\}.$$
\end{proof}

\begin{remark}
    If $\partial\Omega$ is the boundary of a ball with sufficiently large radius, the radiation condition implies that $\partial u/\partial \nu \approx ik u$. Accordingly, $\mathcal{I}_k(z)$ can be adapted to employ the far-field measurement $u$ as follows
\begin{align}\notag
    \mathcal{I}_{k}(z) =\int_{\partial \Omega} \left(\dfrac{\partial \operatorname{Im}\Phi_k(x,z)}{\partial\nu(x)}u(x,k) - ik \operatorname{Im}\Phi_k(x,z) u(x,k) \right) dS(x).
\end{align}
\end{remark}

\section{A simple model equation in Fourier domain}\label{section3}
Note that \eqref{eq:Iz} can be considered as a model equation. However, using a direct discretization of this equation usually leads to large sets of parameters particularly for the 3D case. To address this challenge, it is advantageous to reformulate the equation in the spectral domain, where functions are represented through their Fourier series expansions. In this framework, a function can be approximated using a relatively small number of Fourier coefficients, enabling efficient computations while preserving the essential characteristics of the original function. 
To this end, we will periodize equation \eqref{eq:Iz} and utilize the convolution structure with explicitly computed Fourier coefficients of the periodized kernel function $\operatorname{Im}\Phi_k$. By combining this periodized formulation with the Fourier transform, we obtain a model equation that facilitates efficient computation. The periodization technique is largely inspired by \cite{vainikko2000fast}. 

Let ${B_\rho}$ be the ball with radius $\rho$ that is centered at the origin such that 
$\text{supp}(f)\subset \overline{B_\rho}$. We define
\begin{equation}
    \K_k (x) = \begin{cases}\label{per:ImG}
    \operatorname{Im}\Phi_k(x),\hspace{0.5 cm} x\in B_R, \\ 
    0,\hspace{1.8 cm} x\in G_R\setminus \overline{B_R}, 
\end{cases}\qquad R\ge 2\rho,
\end{equation}
where $G_R$ is the open cube centered at the origin with side length $2R$,
$$G_R=\{x=(x_1,\ldots,x_n)\in \R^n: |x_\ell|<R, \ell=1,\ldots,n \}.$$
Then we extend  $\K_k$ from $G_R$ to $\R^n$ as a $2R$-periodic function. The source function $f$ is also periodized in the same way by $f_{\text{per}}$. Then the periodized version of \eqref{eq:Iz} is given by
\begin{equation}\label{eq:periodizedIz}
    \mathcal{I}_k(z)=\int_{G_R} \K_k(z-y) f_{\text{per}}(y)dy,\quad z\in \R^n.
\end{equation}
Using the estimate of the location and shape of $\text{supp}(f)$ provided by the imaging function discussed in the previous section, the ball $B_\rho$ can be chosen efficiently in the implementation of our algorithm. Although the periodized equation \eqref{eq:periodizedIz} differs from the original equation \eqref{eq:Iz}, the two equations coincide  on \(B_\rho\) which is sufficient for the determination of $f$. 

Our next goal is to transform \eqref{eq:periodizedIz} to the Fourier domain. First,
recall that the trigonometric orthonormal basis of $L^2(G_R)$ is given by
$$\varphi_j(x)=(2R)^{-n/2} e^{i\pi j\cdot\frac{x}{R}},\qquad j=(j_1,\ldots,j_n)\in \Z^n.$$
For $g\in L^2(G_R)$ and $j\in \Z^n$, the Fourier coefficients of $g$ are given by
$$\widehat{g}(j)=\int_{G_R} g(x)\overline{\varphi}_j(x)dx.$$
Now, observe that  the right-hand side of \eqref{eq:periodizedIz} exhibits a convolution structure. From this, we derive the \textbf{model equation}
\begin{equation}\label{eq:modeleq}
    \mathcal{I}_k(z)=(2R)^{n/2}\mathcal{F}^{-1}\left[\widehat{\mathcal{K}}_k(j)\widehat{f}_{\text{per}}(j) \right](z),\qquad j\in \mathbb{Z}^n, \, z\in B_\rho,
\end{equation}
where $\mathcal{F}$ is the Fourier transform.
 To use this model equation, we need explicit   Fourier coefficients of $\K_k$.
For our calculations, it is convenient to introduce the parameters
$$\lambda_j = k^2 - \dfrac{\pi^2}{R^2}|j|^2,$$
where $|j|=\sqrt{j_1^2+\ldots+j_n^2}$ with $j\in \Z^n$.

We now state a theorem that yields an explicit expression for $\widehat{\K}_k$.
\begin{theorem}\label{thm:Fcoef}
    The Fourier coefficients of $\K_k$ are given by
    \begin{itemize}
        \item for $n=2$,
        \begin{equation}\label{Fco:n2}
            \widehat{\K}_k(j) = \begin{cases}
        \frac{\pi R}{8} \left[J_0^2(kR)+J_1^2(kR)\right], \hspace{6.5 cm}  \pi |j|= kR, \\ 
            -\frac{1}{4}\cdot\frac{\pi R}{k^2R^2-\pi^2|j|^2}\left[\pi|j|J_0(kR)J_1(\pi|j|)-kR J_1(kR)J_0(\pi|j|)\right],\hspace{0.85cm} \pi |j|\neq kR, 
        \end{cases} 
        \end{equation}
       
        \item for $n=3$,
        \begin{equation}\label{Fco:n3}
            \widehat{\K}_k(j) = \begin{cases}
        (2R)^{-3/2}\frac{ R^2}{2}\left\{ kR \left[j_0^2(kR)+j_1^2(kR)\right]-j_0(kR)j_1(kR)\right\}, \hspace{2.25cm}  \pi |j|= kR, \\
           kR(2R)^{-3/2}\frac{R^2}{k^2R^2-\pi^2|j|^2}\left[kRj_0(\pi |j|)j_1(kR)-\pi|j|j_0(kR)j_1(\pi|j|)\right],\hspace{0.7cm} \pi |j|\neq kR,|j|\neq 0,\\ 
           \dfrac{1}{k}(2R)^{-3/2}\left(\dfrac{1}{k}\sin(kR)-R\cos(kR) \right),\hspace{4.9cm} |j|=0,
        \end{cases} 
        \end{equation}
    \end{itemize}
    for all $j\in \Z^n$. Moreover,
    \begin{equation}\label{asymp}
        \widehat{\K}_k(j) =  \begin{cases}
    \mathcal{O}\left(|j|^{-3/2}\right), \quad  \text{if }  n=2,\\ 
    \mathcal{O}\left(|j|^{-2}\right), \hspace{0.65
    cm}\text{if }  n=3,
\end{cases}
    \end{equation}
as $|j|\to \infty$.
\end{theorem}

\begin{proof}
First, let $\mathcal{P}_k$ denote the periodized Green’s function associated with $\Phi_k$, defined analogously to \eqref{per:ImG}. Since $\mathcal{K}_k = \operatorname{Im}\mathcal{P}_k$, it immediately follows that $\widehat{\mathcal{K}}_k = \operatorname{Im}\widehat{\mathcal{P}}_k$. The calculation of its Fourier coefficients $\widehat{\mathcal{P}}_k$ is carried out in \cite{vainikko2000fast} for both $n=2$ and $n=3$.  Using these results, for $n=2$, we have
$$
        \widehat{\K}_k(j) = \begin{cases}
        \frac{\pi}{8k}\left\{J_0(kR)\left[J_1(kR)+\frac{1}{2}kR(J_0(kR)-J_2(kR))\right]+kRJ_1^2(kR) \right\}, \hspace{1 cm}  \pi |j|= kR, \\ \notag
            -\frac{1}{4}\cdot\frac{\pi R}{k^2R^2-\pi^2|j|^2}\left[\pi|j|J_0(kR)J_1(\pi|j|)-kR J_1(kR)J_0(\pi|j|)\right],\hspace{1.8cm} \pi |j|\neq kR.
        \end{cases} 
        $$
        When $\pi|j| = kR$, using the recurrence relation $J_2(kR) = \frac{2}{kR}J_1(kR)-J_0(kR)$, we can simplify $\widehat{\K}_k(j)$ as
    $$\widehat{\K}_k(j) =\dfrac{\pi R}{8} \left[J_0^2(kR)+J_1^2(kR)\right],$$
    which gives \eqref{Fco:n2}. 
    Similar for the case $n=3$, we have
    $$\widehat{\K}_k(j) = \begin{cases}
        \frac{ R}{2k}(2R)^{-3/2}\left(1-\frac{1}{kR}\cos(kR)\sin(kR)\right), \hspace{5.3cm}  \pi |j|= kR, \\ \notag
            \frac{k}{\pi}\cdot\frac{R^2}{k^2R^2-\pi^2|j|^2}(2R)^{-3/2}\left(\frac{\pi}{k}\sin(kR)\cos(\pi|j|)-\frac{R}{|j|}\cos(kR)\sin(\pi|j|)\right),\hspace{0.3 cm} \pi |j|\neq kR, |j|\neq 0, \\ \notag
            \frac{1}{k}(2R)^{-3/2}\left(\frac{1}{k}\sin(kR)-R\cos(kR) \right),\hspace{5.6cm} |j|=0.
        \end{cases}  $$
 With the observation that $$j_0(x) = \dfrac{\sin(x)}{x},\qquad j_1(x) = \dfrac{\sin(x)}{x^2}-\dfrac{\cos(x)}{x},$$
we can write $\widehat{\K}_k$ as in \eqref{Fco:n3}.
 The asymptotic form \eqref{asymp} is obtained from the behaviors of Bessel functions 
 $$J_\alpha(z) = \mathcal{O}\left(z^{-1/2}\right),\qquad j_\alpha(z) = \mathcal{O}\left(z^{-1}\right), \quad \text{for all } \alpha\in \N$$
 as $|z|\to \infty$.
\end{proof}

The above theorem demonstrates that $|\widehat{\K}_k|$ decays rapidly, with a rate of $|j|^{-3/2}$ for $n=2$ and of $|j|^{-2}$ for $n=3$ as $|j|\to\infty$.
The next theorem shows that, for $|j|\ge kR/\pi$, the maximum of $|\widehat{\mathcal{K}_k}(j)|$ is attained at $|j|=kR/\pi$, which guides the choice of an appropriate truncation for the Fourier series representation of $\K_k$. This truncation is important for the efficiency of training our neural network, as discussed in the next section.
\begin{theorem}\label{thm:ring} Let 
$$\mu = \begin{cases}
    \frac{\pi R}{8} \left[J_0^2(kR)+J_1^2(kR)\right], \hspace{5.15 cm} n=2, \\ \notag
    (2R)^{-3/2}\frac{ R^2}{2}\left[ kR [j_0^2(kR)+j_1^2(kR)]-j_0(kR)j_1(kR)\right],\quad n=3.
\end{cases}$$
Then, for all $j\in \Z^n$ such that $\pi|j|\ge kR \ge 1$, we have $$|\widehat{\K}_k(j)|\le \mu.$$

\end{theorem}
\begin{proof}
    For $n=2$, since
    \begin{equation}
        \label{eq:term1}
        \int_0^{t} z J^2_0 (z)dz = \dfrac{1}{2}t^2\left[J_0^2(t)+J_1^2(t)\right]\qquad \text{for any } t>0,
    \end{equation}
we have
\begin{equation}
        \notag
        \int_0^{R} r J^2_0 (kr)dr = \dfrac{R^2}{2}\left[J_0^2(kR)+J_1^2(kR)\right]=
        \dfrac{4R}{\pi}\mu.
    \end{equation}
    Let $j\in \Z^2$ be such that $\pi |j|\neq kR$. Since
    \begin{align}
    \notag
    \int_0^R r J_0 (kr)J_0 \left(\frac{\pi|j|}{R}r\right)dr &= R\frac{kJ_{1}(kR)J_0 \left(\pi|j|\right)-\frac{\pi|j|}{R}J_0(kR)J_{1}\left(\pi|j|\right)}{k^2-\left(\frac{\pi|j|}{R}\right)^2} \\ \notag
    & = \dfrac{ R^2}{\pi^2 |j|^2- k^2R^2}\left[\pi |j| J_0(kR)J_{1}(\pi|j|) -  kR J_{1}(kR)J_0 \left(\pi|j|\right)\right],
\end{align}
we can rewrite
\begin{equation}\label{eq:term2}
    \widehat{\K}_k(j)=\dfrac{\pi}{4R}\int_0^R r J_0 (kr)J_0 \left(\frac{\pi|j|}{R}r\right)dr .
\end{equation}
By the Cauchy-Schwarz inequality and \eqref{eq:term1}, we have
$$| \widehat{\K}_k(j)|\le \dfrac{\pi}{4R}\sqrt{\int_0^R rJ_0^2(kr)dr}\sqrt{\int_0^R rJ_0^2\left(\frac{\pi|j|}{R}r\right)dr}= \dfrac{\pi}{4R}\sqrt{\dfrac{4R}{\pi}\mu}\sqrt{\int_0^R rJ_0^2\left(\frac{\pi|j|}{R}r\right)dr}.$$
Employing the relation \eqref{eq:term1} yields
$$\int_0^R rJ_0^2\left(\frac{\pi|j|}{R}r\right)dr = \dfrac{R^2}{2}\left[J_0^2(\pi|j|)+J_1^2(\pi|j|)\right].$$
Consider $v(t)=J_0^2(t)+J_1^2(t)$ for $t>0$. This function is decreasing since
$$v'(t)=\dfrac{-2J_1^2(t)}{t}<0,\quad\text{ for } t>0.$$
Therefore,
$$\int_0^R rJ_0^2\left(\frac{\pi|j|}{R}r\right)dr\le \dfrac{R^2}{2}\left[J_0^2(kR)+J_1^2(kR)\right]= \dfrac{4R}{\pi}\mu,$$
for all $j\in \Z^2, \pi|j|> kR$. This implies that
$| \widehat{\K}_k(j)|\le \mu,$
for all $j\in \Z^2$ such that $\pi|j|\ge kR$.

For $n=3$, since $j_{-1}(z) = \dfrac{j_0(z)}{z}-j_1(z)$ and
\begin{equation}
    \label{eq:term2-3D}
    \int_0^{t} z^2j_0^2(z)dz=\dfrac{t^3}{2}\left[j_0^2(t)-j_{-1}(t)j_{1}(t)\right] \qquad \text{for any } t>0,
\end{equation}
we have 
$$\int_0^{kR} z^2j_0^2(z)dz=\dfrac{(kR)^3}{2}(j_0^2(kR)-j_{-1}(kR)j_{1}(kR))=\dfrac{(kR)^2}{2}\left\{ kR \left[j_0^2(kR)+j_1^2(kR)\right]-j_0(kR)j_1(kR)\right\}.$$
Then
\begin{equation}\notag
    \int_0^{R} r^2j_0^2(kr)dr=\dfrac{R^2}{2k}\left\{ kR \left[j_0^2(kR)+j_1^2(kR)\right]-j_0(kR)j_1(kR)\right\}=\dfrac{R^2}{2k}\dfrac{2}{R^2}(2R)^{3/2} \mu=\dfrac{(2R)^{3/2}}{k}\mu.
\end{equation}
Let $j\in \Z^3$ such that $\pi |j|\neq kR$. Since
    \begin{align}
    \notag
    \int_0^R r^2 j_0(kr)j_0\left(\dfrac{\pi|j|}{R}r\right)dr &= \dfrac{R^4}{\pi^2|j|^2-k^2R^2}\left[\dfrac{\pi|j|}{R}j_0(kR)j_{1}\left(\pi|j|\right)-kj_{1}(kR)j_0\left(\pi|j|\right)\right] \\ \notag
    & = \dfrac{ R^3}{\pi^2 |j|^2- k^2R^2}\left[\pi|j|j_0(kR)j_{1}\left(\pi|j|\right)-kRj_{1}(kR)j_0\left(\pi|j|\right)\right],
\end{align}
we have
$$\widehat{\K}_k(j)=k(2R)^{-3/2}\int_0^R r^2 j_0(kr)j_0\left(\dfrac{\pi|j|}{R}r\right)dr.$$
By the Cauchy-Schwarz inequality and \eqref{eq:term2-3D}, we obtain
\begin{align}
    \notag
    |\widehat{\K}_k(j)|&\le k(2R)^{-3/2}\sqrt{\int_0^R r^2 j_0^2(kr)dr}\sqrt{\int_0^R r^2 j_0^2\left(\dfrac{\pi|j|}{R}r\right)dr}\\ \notag
    &= k(2R)^{-3/2}\sqrt{\dfrac{(2R)^{3/2}}{k}\mu}\sqrt{\int_0^R r^2 j_0^2\left(\dfrac{\pi|j|}{R}r\right)dr}.
\end{align}
Adopting \eqref{eq:term2-3D}, we have
$$\int_0^R r^2 j_0^2\left(\dfrac{\pi|j|}{R}r\right)dr=\dfrac{R^3}{2\pi|j|}\left\{\pi|j|\left[j_0^2(\pi|j|)+j_1^2(\pi|j|)\right]-j_0(\pi|j|)j_1(\pi|j|)\right\}.$$
Consider $w(t) =  j_0^2(t)+j_1^2(t)-j_0(t)j_1(t)/t$ for $t\ge1$.
\begin{align}
\notag
 w'(t)&= \dfrac{-3tj_1^2(t)-tj_0^2(t)+3j_0(t)j_1(t)}{t^2}\le \dfrac{-3j_1^2(t)-j_0^2(t)+3j_0(t)j_1(t)}{t^2}\\ \notag
 &=\dfrac{1}{t^2}\left[-\left(j_0(t)-\dfrac{3}{2}j_1(t) \right)^2-\dfrac{3}{4}j_1^2(t) \right]\le 0.
\end{align}
This shows that $w$ is decreasing, so
$$\int_0^R r^2 j_0^2\left(\dfrac{\pi|j|}{R}r\right)dr\le \dfrac{R^3}{2 kR}\left\{kR\left[j_0^2(kR)+j_1^2(kR)\right]-j_0(kR)j_1(kR)\right\}=\dfrac{(2R)^{3/2}}{k}\mu,$$
for all $j\in \Z^3, \pi|j|> kR\ge 1$.
Therefore,
$|\widehat{\K}_k(j)|\le \mu,$
for all $j\in \Z^3$ such that $\pi|j|\ge kR\ge 1$.
\end{proof}

Theorems \ref{thm:Fcoef} and \ref{thm:ring} show that 
$|\widehat{\K}_k|$ decays beyond the ring $|j|= kR/\pi$.
 We can observe these characteristics in the Figure \ref{fig:kernel-hat}.

\begin{figure}[H]
    \centering
    \begin{subfigure}[b]{0.3\textwidth}
        \includegraphics[width=\linewidth]{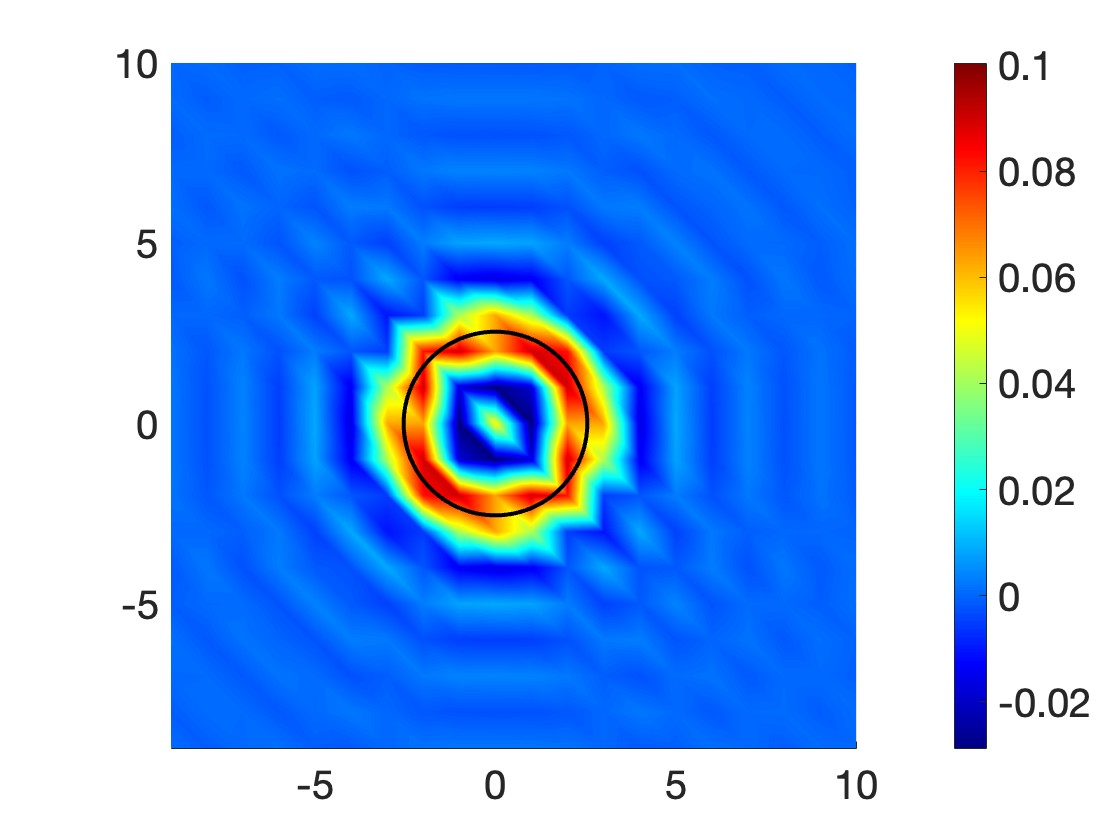}
        \caption{$k=4$}
    \end{subfigure}
    \begin{subfigure}[b]{0.3\textwidth}
            \includegraphics[width=1.0\linewidth]{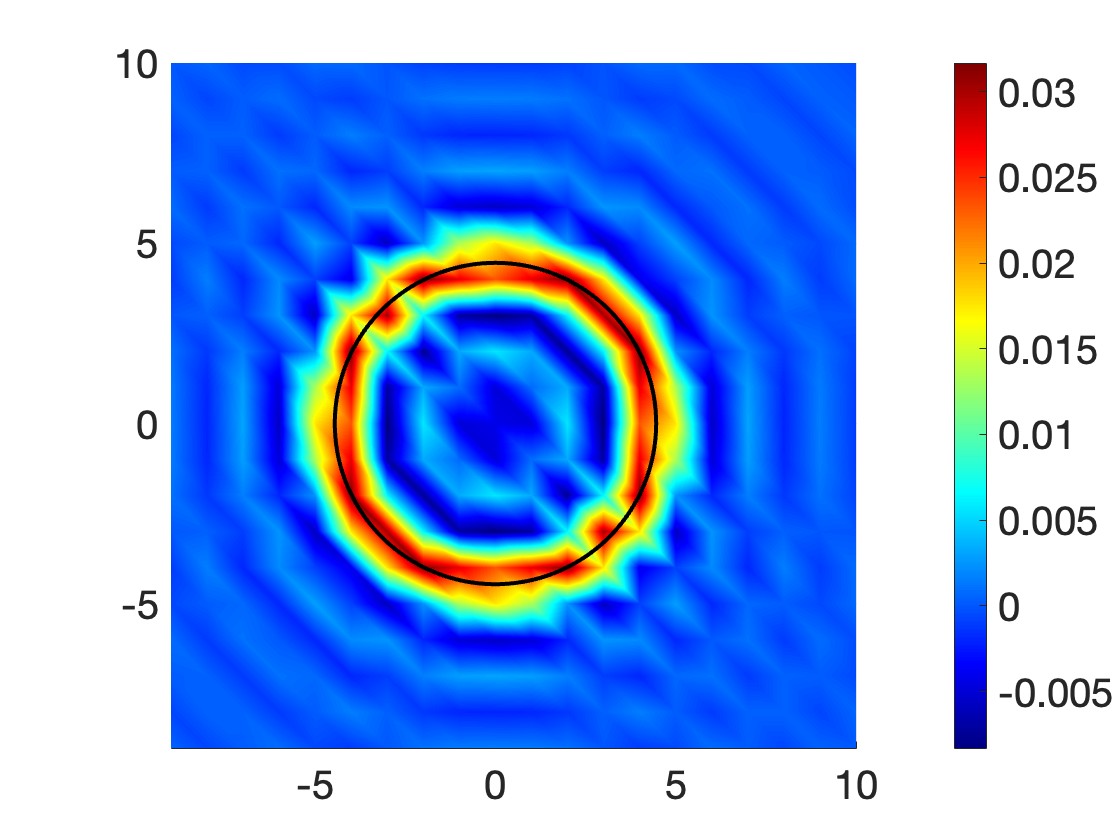}
            \caption{$k=7$}
        \end{subfigure}
    \begin{subfigure}[b]{0.3\textwidth}
        \includegraphics[width=1.0\linewidth]{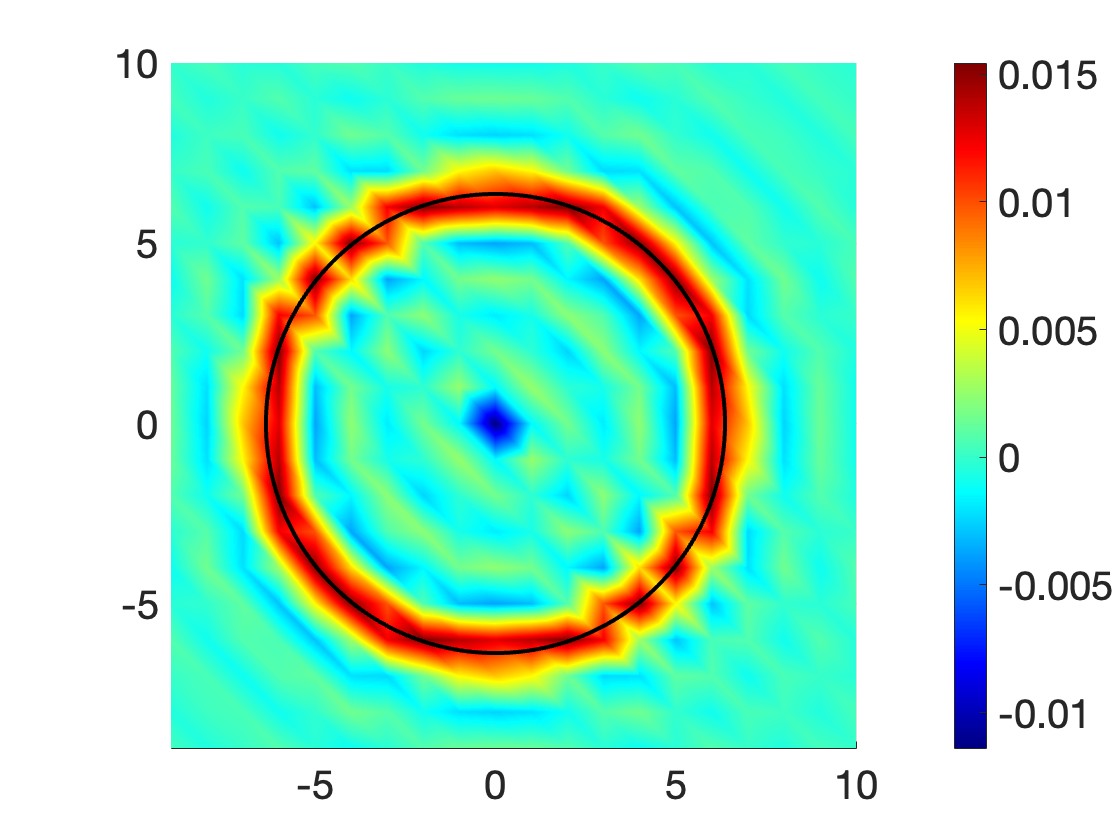}
        \caption{$k=10$}
    \end{subfigure}
    \caption{$|\widehat{\K}_{k}|$ with the circle $|j|=kR/\pi$ (black) for $n=2$
    }\label{fig:kernel-hat}
\end{figure}
This enables a substantial reduction in the number of Fourier coefficients required by the model, since the right-hand side of the model equation \eqref{eq:modeleq} can be  approximated using only the dominant coefficients. In particular, we truncate the series to the range 
$j \in [-J+1,J]^n \subset \Z^n$, 
where $J$ is taken larger but not too far from $kR/\pi$.

\section{Model-informed neural network} \label{section4}

In this section, we will introduce the ideas behind our model-informed neural network. The underlying principle of our model-informed neural network is the optimization of a parametric model based on the model equation \ref{eq:modeleq}. Hence, we will define our hypothesis space, $\mathcal{H}$, and furthermore, formalize our optimization process in this section. 

Our hypothesis space to approximate the Fourier coefficients of the source function, $\widehat{f}_{\text{per}}$, will be the class of fully-connected feedforward neural networks. In our optimization process, we make the following assumptions:
\begin{itemize}
    \item There exist a model $\widehat{f}_{\Theta}$ in the hypothesis space $\mathcal{H}$ for which the loss function $\mathcal{L}$ \eqref{eq:loss function} is small. 
    \item Given that there exists a model $\widehat{f}_{\Theta}$ with a small error   $\mathcal{L}$, the corresponding relative error $\mathcal{E}$ \eqref{eq:total error} is small as well. 
\end{itemize} 

Our eventual goal is to find a model $\widehat{f}_{\Theta} \in \mathcal{H}$ such that $\widehat{f}_{\Theta}  \approx \widehat{f}_{\text{per}}$.  As mentioned above, together with the chosen hypothesis space $\mathcal{H}$, our model equation $\eqref{eq:modeleq}$  is used as the basis for the loss function
\begin{align}\label{eq:loss function}
    \mathcal{L}: \mathcal{H} & \rightarrow \R \\ \notag
    \widehat{f}_{\Theta} & \mapsto  \frac{1}{M}\sum_{m=1}^{M} \left| \mathcal{I}_{k_{m}}(z) - (2R)^{n/2}\mathcal{F}^{-1} \left[\widehat{\K}_{k_{m}}(j)\widehat{f}_{\Theta}(j) \right](z)\right|^{2}, \quad j\in \mathbb{Z}^n, \, z\in B_\rho,
\end{align}
where $M$ denotes the number of wave numbers. $\mathcal{L}$ motivates the formulation of our optimization process 
\begin{align*}
    \min_{\widehat{f}_{\Theta} \in \mathcal{H}} \mathcal{L}[\widehat{f}_{\Theta}] := \min_{\widehat{f}_{\Theta} \in \mathcal{H}}  \left( \frac{1}{M}\sum_{m=1}^{M} \left| \mathcal{I}_{k_{m}}(z) - (2R)^{n/2}\mathcal{F}^{-1} \left[\widehat{\K}_{k_{m}}(j)\widehat{f}_{\Theta}(j) \right](z)\right|^{2} \right),  \quad j\in \mathbb{Z}^n, \, z\in B_\rho.
\end{align*}
This overall training process for $n=3$ is shown below. 

\begin{figure}[h!]
    \centering
    \includegraphics[scale = 0.4]{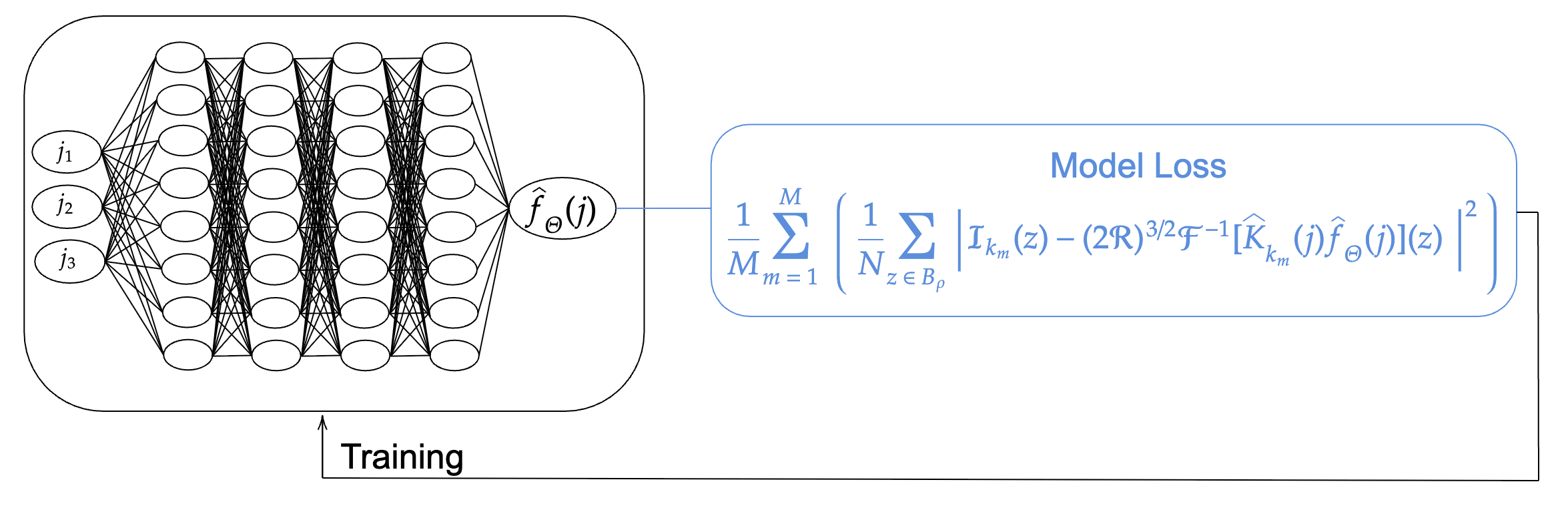}
    \caption{Neural network model in $3$D}
    \label{fig:Model for Inverse Source}
\end{figure} 

\paragraph{Choice of hyperparameters:} As the neural network aims to approximate the Fourier coefficients of the unknown source $f$, the input of the neural network is the set of quadrature points $$\{ (j_{1},j_{2}) \in  \mathbb{Z}^{2}: j_{\ell} = -6,-5, ..., 6,7,  \ell=1,2\} $$  for $n=2$ and $$ \{ (j_{1},j_{2},j_{3}) \in \mathbb{Z}^{3}: j_{\ell} = -5,-4, ..., 5,6,  \ell=1,2,3\},$$ for $n=3$ . The output of our network will be complex-valued numbers. Hence, the hypothesis space is made by fully connected feedforward neural networks with $n$ input neurons and $2$ output neurons. The depth and width of the neural network are $4$ and $64$ respectively with the $\tanh$ as the activation function in the hidden layers. The optimization process will be initialized by the Glorot normal initializer and will be performed by the Adams optimizer. The learning rate is $0.001$ and training is done for $50,000$ iterations. 

\begin{algorithm}[H]
\caption{Model-informed Neural Network Algorithm}\label{alg:algorithm}
\label{alg:Model_Informed_Neural_Network_Algorithm}

\SetKwInOut{Input}{Input}

\Input{$N$: number of collocation points in $B_\rho$ \\
$M$:  number of wave numbers\\
\hspace*{0.3em}$\tau$: learning rate}

\BlankLine
Initialize the network architecture $\widehat{f}_{\Theta}$ with the hyperparameters given above.  \\ 
\While{not converged}{

\For{$m=1,\dots,M$}{ Compute
\[
\mathrm{MSE}_{k_{m}}(\widehat{f}_{\Theta})
= \frac{1}{N} \sum_{i=1}^{N} \left| \mathcal{I}_{k_{m}}(z_{i}) - (2R)^{n/2}\mathcal{F}^{-1} \left[\widehat{\K}_{k_{m}}(j)\widehat{f}_{\Theta}(j) \right](z_{i})\right|^{2}.
\]
}\
Compute
\[
\mathrm{MSE}(\widehat{f}_{\Theta})
= \frac{1}{M}\sum_{m=1}^{M}
\mathrm{MSE}_{k_{m}}(\widehat{f}_{\Theta});
\]

Update $\Theta \leftarrow \Theta - \tau\nabla_{\Theta}\mathrm{MSE}(\widehat{f}_{\Theta})$\
}
\SetKwInOut{Output}{Output}
\Output{$\widehat{f}^{\ast}_{\Theta}$}
\end{algorithm}

    The convergence of the algorithm is assessed through the stabilization of the relative error. In our numerical examples, the relative error stabilizes after $50,000$ iterations (see Figure \ref{fig:Update process and Loss curve (2D) - Kite}). 

    Furthermore, with the Fourier coefficients $\widehat{f}^{\ast}_{\Theta}$ obtained from the neural network, we can reconstruct the (real-valued) source $f$ using 
    $$f_{\text{comp}}=\mathrm{Re} \left[\sum_{j\in [-J+1,J]^n}\widehat{f}^{\ast}_{\Theta}(j) \varphi_j \right].$$

\section{Numerical results}\label{section5}
In this section, we demonstrate the ability of our algorithm to reconstruct the various coefficients of interest from boundary field data for both cases $2$D  and $3$D. For each case, we have $10$ evenly spaced wave numbers between $2$ and $10$. For $n=2$, $\partial \Omega$ is the circle of radius $100$ centered at $0$, uniformly discretized into $96$ points. For $n=3$, $\partial \Omega$ is the sphere of radius $100$ centered at $0$, uniformly discretized into $24$ azimuthal and polar points each. The boundary data is obtained by computing the volume potential \eqref{eq:Volume}. To simulate noisy data, we add $10\%$ noise to the boundary  data using the following formula   
\[u_{\delta} = u+ 0.1 \times \frac{\mathcal{N}}{\left\Vert\mathcal{N}\right\Vert_F} \left\Vert{u}\right\Vert_F,\]
where $\mathcal{N}$ is  the noise matrix consisting of random entries $a + bi$ for $a, b \in (-1, 1)$, drawn from an uniform distribution and $|| \cdot ||_{F}$ denotes the Frobenius norm.

The sampling domain for the $\Omega$ is the $[-2,2]^{n}$ grid, uniformly discretized into $64$ sampling points in each direction. We introduce the following metric as a quantitative measure of our algorithm:
\begin{align}\label{eq:total error}
    \mathcal{E}_{\text{Fourier}} \left[ \widehat{f}_{\Theta}^{\ast}\right] & = \frac{  \left\| \widehat{f}_{\Theta}^{\ast}(j)  - \widehat{f}_{\text{per}}(j)
    \right\|_{F}}{\left\|\widehat{f}_{\text{per}}(j)
    \right\|_{F}}.
\end{align}
This metric quantifies the difference between the true and predicted Fourier coefficients, thus acting as a direct assessment of the learning capacity of the algorithm. 

\subsection{$2$D experimental configuration and results} We show the reconstruction ability for the following sources: a square-shaped source (Figure \ref{fig: New Square Results}), a kite-shaped source (Figure \ref{fig: New Kite Results}), a triangle-shaped source (Figure \ref{fig: New Triangle Results}) and a square-with-hole-shaped source (Figure \ref{fig: New Rect with Hole Results}). For the square-shaped source, we explicitly computed the Fourier coefficients using the basis provided in Section \ref{section3}. For the remaining shapes, we computed the respective Fourier coefficients using the Fast Fourier Transform. Table \ref{table: new 2D table} reflects a low relative error, indicating accurate reconstructions. The algorithm is able to accurately identify the boundaries, locations, sizes, and coefficient values of the source functions.  In particular, the results are still satisfactory when noise is added to the boundary data. Additionally, the proposed algorithm also exhibits fewer artifacts in the background region compared to the traditional linear least squares method.

Table \ref{table: new 2D table} also reflects a low computation time, which arises from a low computational cost of $\mathcal{O}(MN^{2} \log N)$. By defining the algorithm in the spectral domain, we overcome the need for $N^{4}$ operations when evaluating the convolutional structure in the equation \eqref{eq:Iz}.  
The convergence behavior of the optimization process for both the kite-shaped and triangle-shaped sources is illustrated in the relative error versus iteration plot and the training loss versus iteration plot (Fig. \ref{fig:Update process and Loss curve (2D) - Kite}). The relative error is computed according to equation \eqref{eq:total error}, while the training loss is evaluated using the model loss function defined in \eqref{eq:loss function}. Both curves exhibit stabilization after approximately the 50,000th iteration, suggesting that the optimization process has reached a plateau. This behavior indicates that subsequent parameter updates become marginal, and further training iterations are unlikely to result in substantial performance improvements.

\begin{figure}[H]
    \centering
    \begin{subfigure}[b]{0.34\textwidth}
        \includegraphics[width=\linewidth]{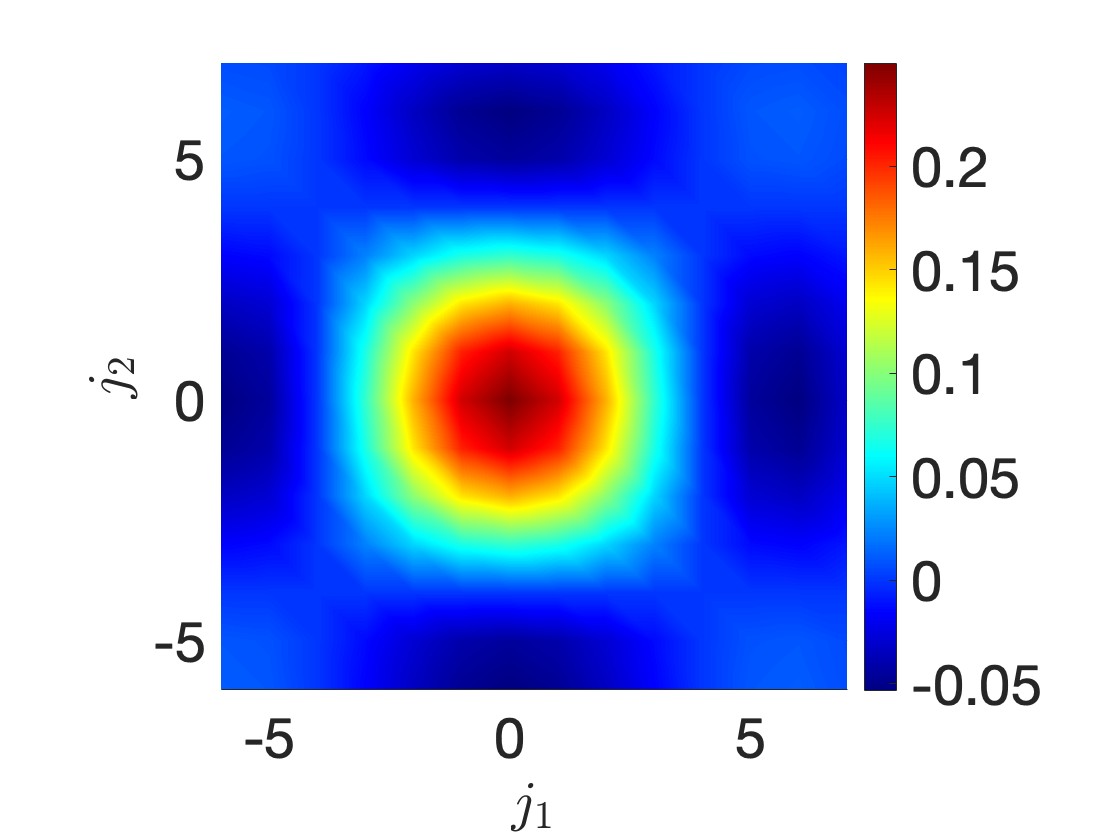}
        \caption{$\mathrm{Re}(\widehat{f}_{\text{per}})$ ($\mathrm{Im}(\widehat{f}_{\text{per}})=0)$}
        \label{trueF:square}
    \end{subfigure}
    \hspace{-0.5cm}
    \begin{subfigure}[b]{0.34\textwidth}
        \includegraphics[width=\linewidth]{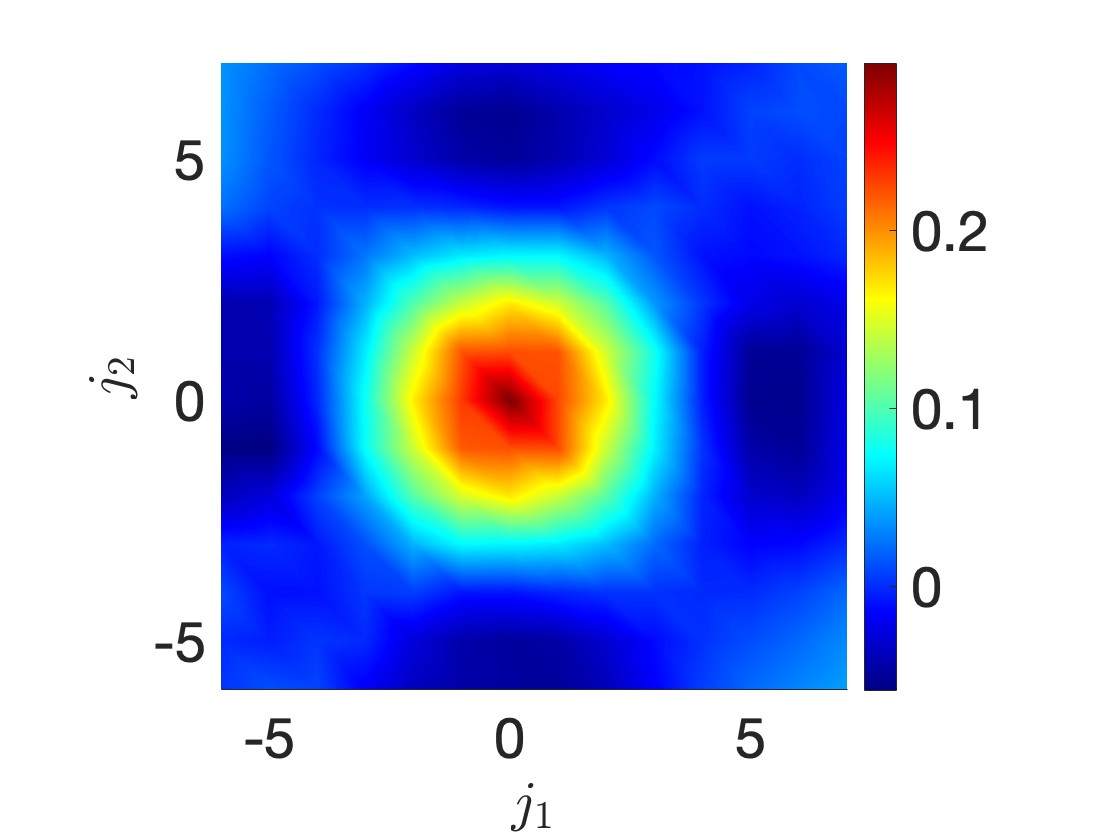}
        \caption{$\mathrm{Re}(\widehat{f}_{\Theta}^{\ast})$ 
        }
        \label{predFreal:square}
    \end{subfigure}
    \hspace{-0.5cm}
    \begin{subfigure}[b]{0.34\textwidth}
        \includegraphics[width=\linewidth]{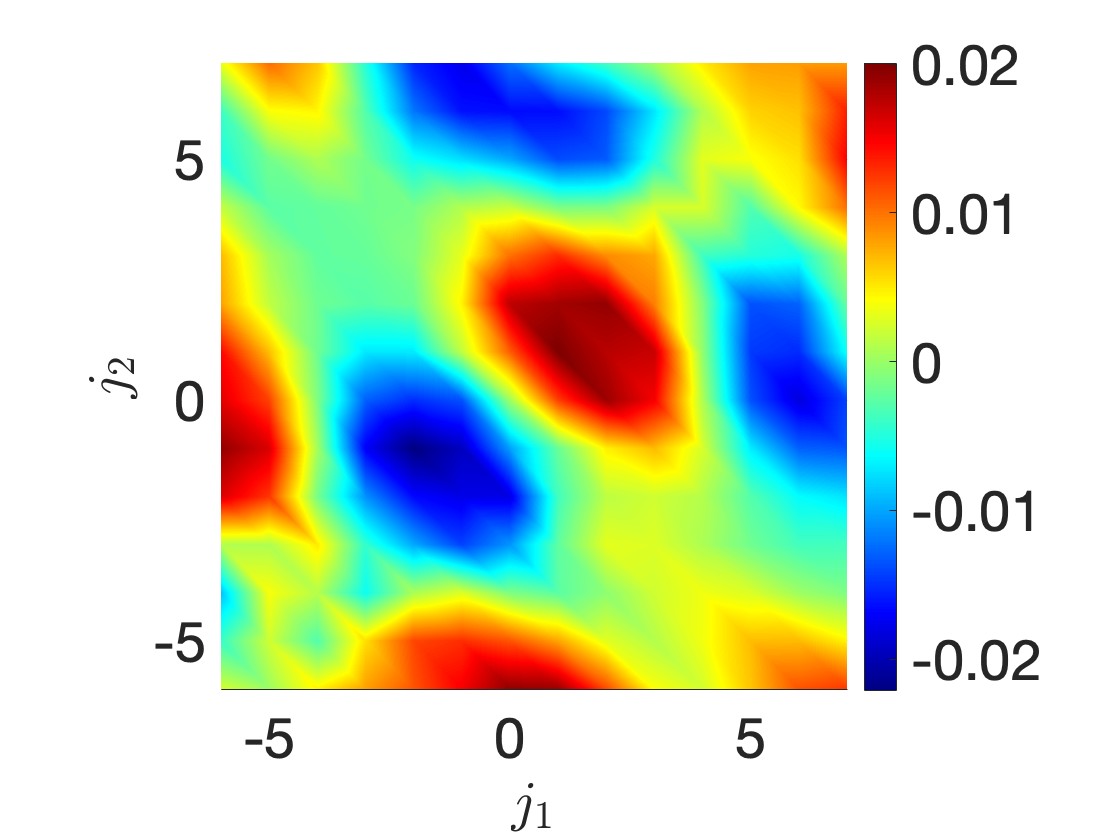}
        \caption{$\mathrm{Im}(\widehat{f}_{\Theta}^{\ast})$ 
        }
        \label{predFimag:square}
    \end{subfigure}

    \medskip

    \begin{subfigure}[b]{0.34\textwidth}
        \includegraphics[width=\linewidth]{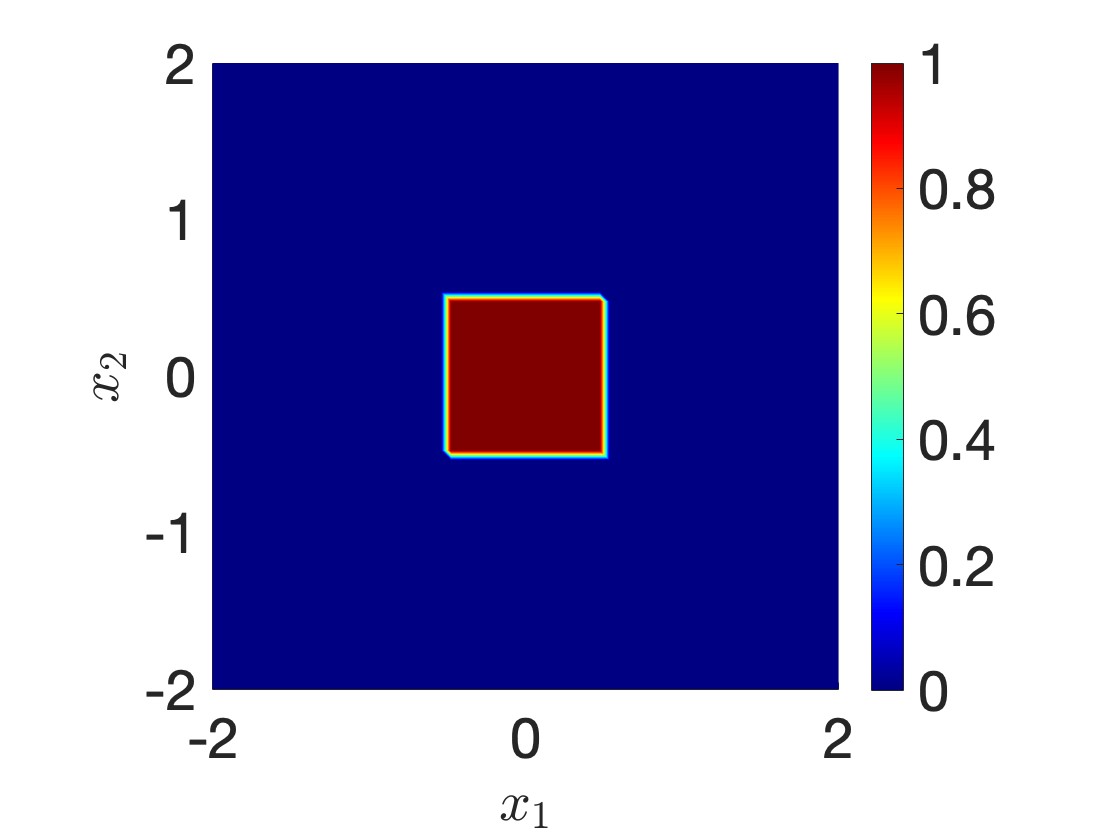}
        \caption{ $f$}
        \label{true:square}
    \end{subfigure}
    \hspace{-0.5cm}
    \begin{subfigure}[b]{0.34\textwidth}
        \includegraphics[width=\linewidth]{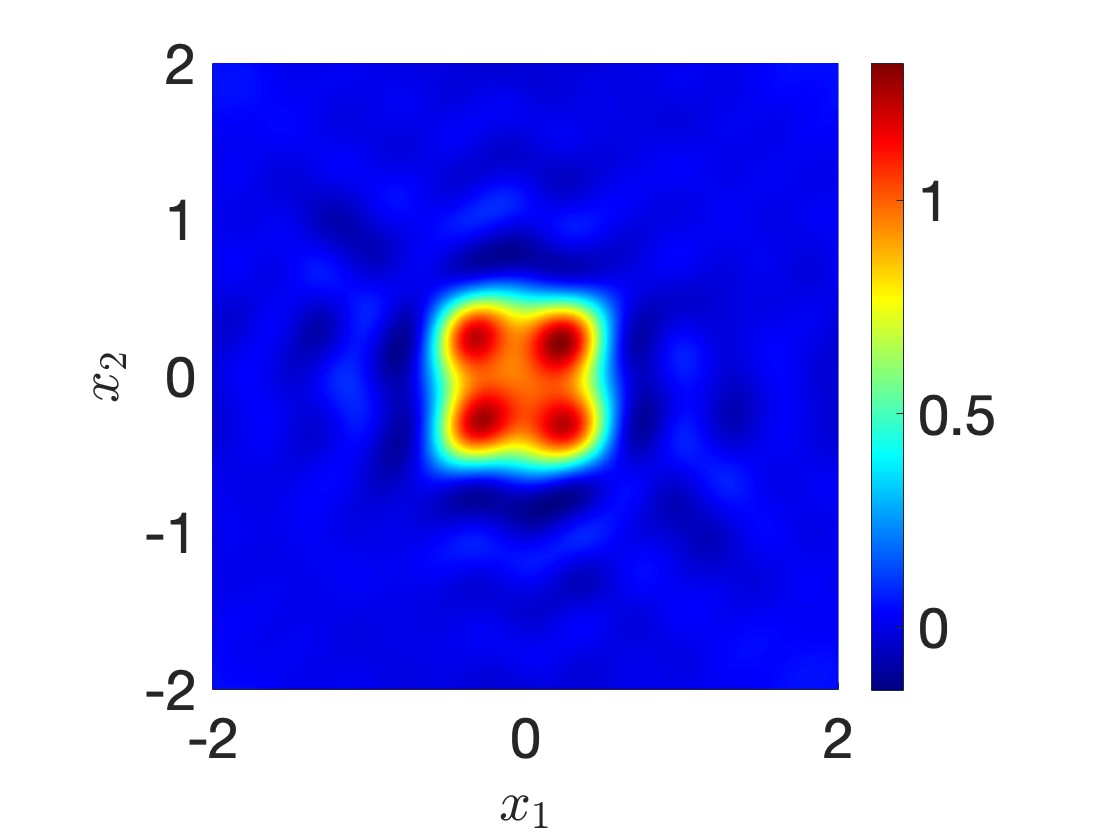}
        \caption{$f_{\text{comp}}$ 
        }
        \label{pred:square}
    \end{subfigure}
    \hspace{-0.5cm}
        \begin{subfigure}[b]{0.34\textwidth}
        \includegraphics[width=1.0\linewidth]{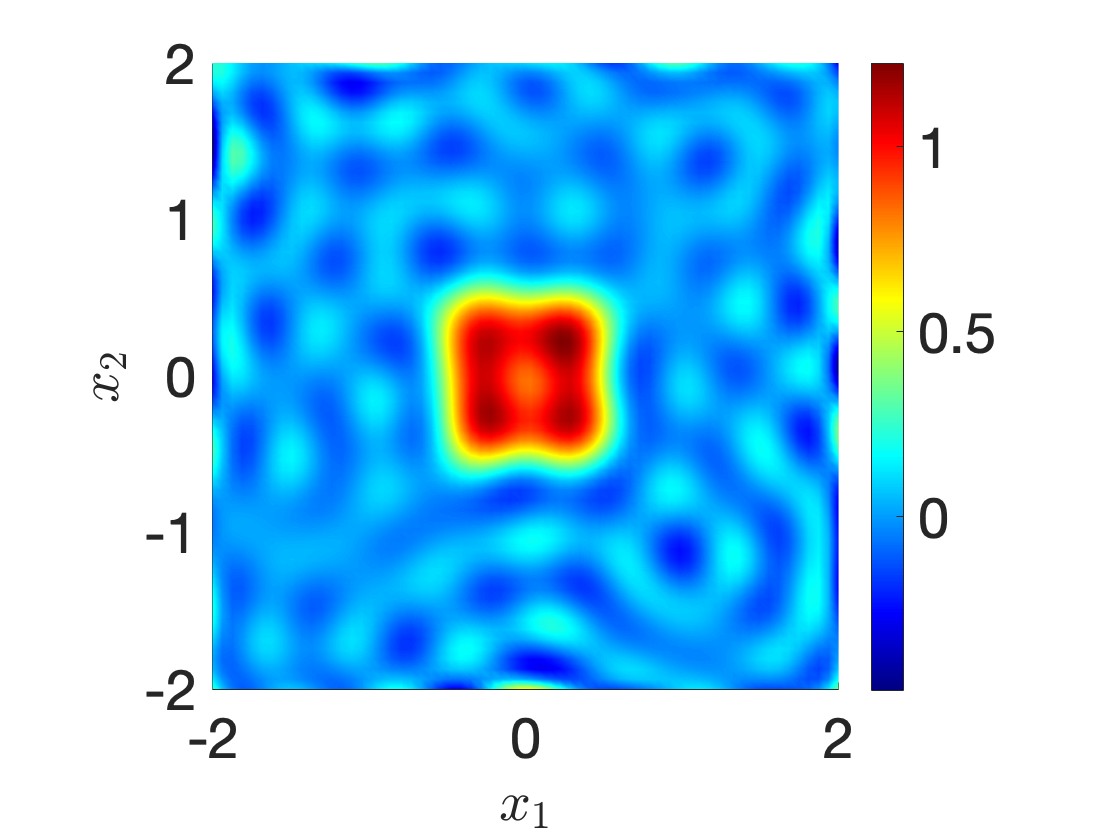}
        \caption{Linear least square}
        \label{lls:square}
    \end{subfigure}
   \caption{Reconstruction results for the square-shaped source. 
Top row: exact Fourier coefficients \eqref{trueF:square} and the real \eqref{predFreal:square} and imaginary \eqref{predFimag:square} parts predicted by the proposed method. 
Bottom row: comparison of the true source \eqref{true:square}, our reconstruction \eqref{pred:square}, and the linear least-squares reconstruction \eqref{lls:square}.}
    \label{fig: New Square Results}
\end{figure}

\begin{figure}[H]
    \centering
    \begin{subfigure}[b]{0.25\textwidth}
        \includegraphics[width=\linewidth]{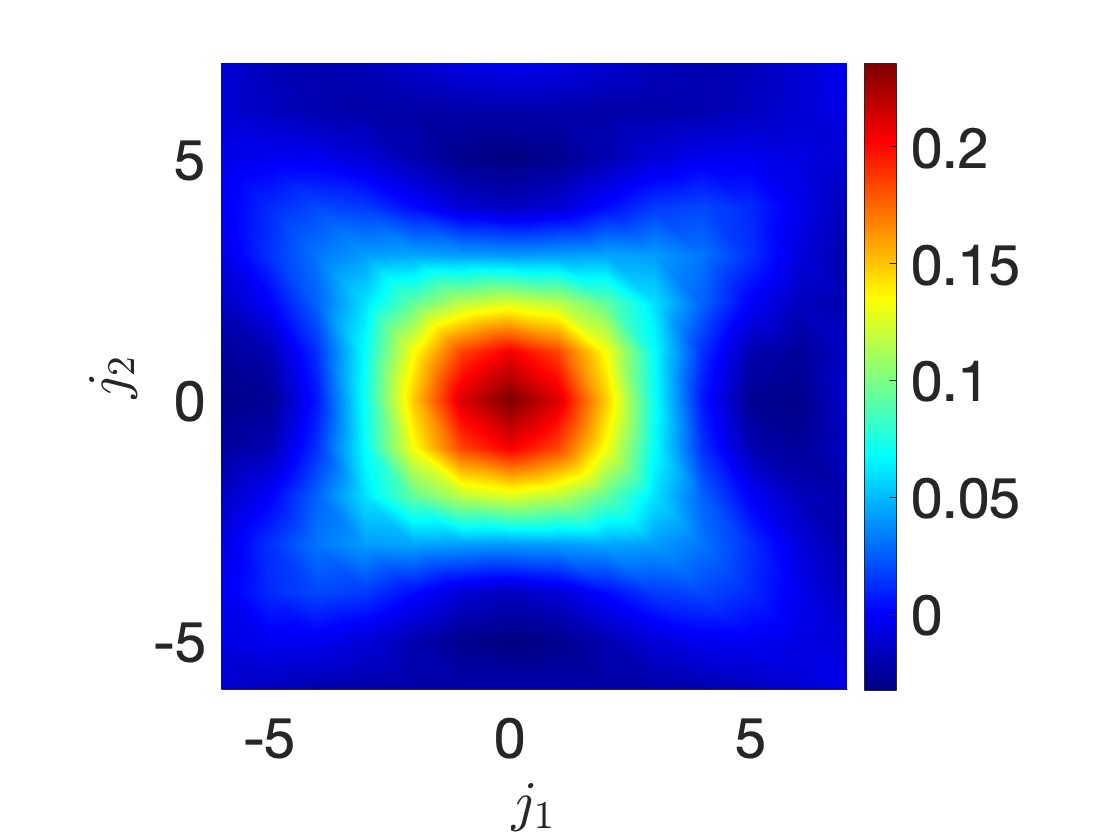}
        \caption{$\mathrm{Re}({\widehat{f}_{\text{per}}}$)}
        \label{trueFreal:kite}
    \end{subfigure}
    \hspace{-0.5cm}
    \begin{subfigure}[b]{0.25\textwidth}
        \includegraphics[width=\linewidth]{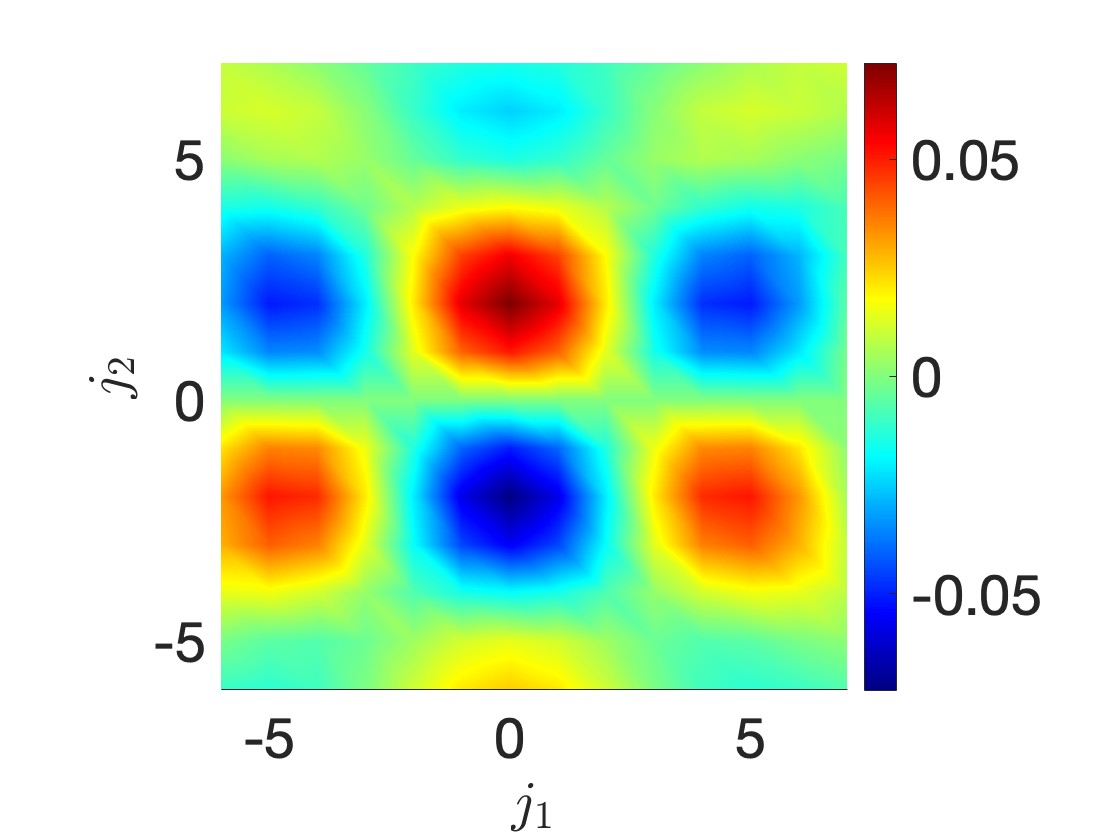}
        \caption{$\mathrm{Im}({\widehat{f}_{\text{per}}}$)}
        \label{trueFimag:kite}
    \end{subfigure}
    \hspace{-0.5cm}
    \begin{subfigure}[b]{0.25\textwidth}
        \includegraphics[width=\linewidth]{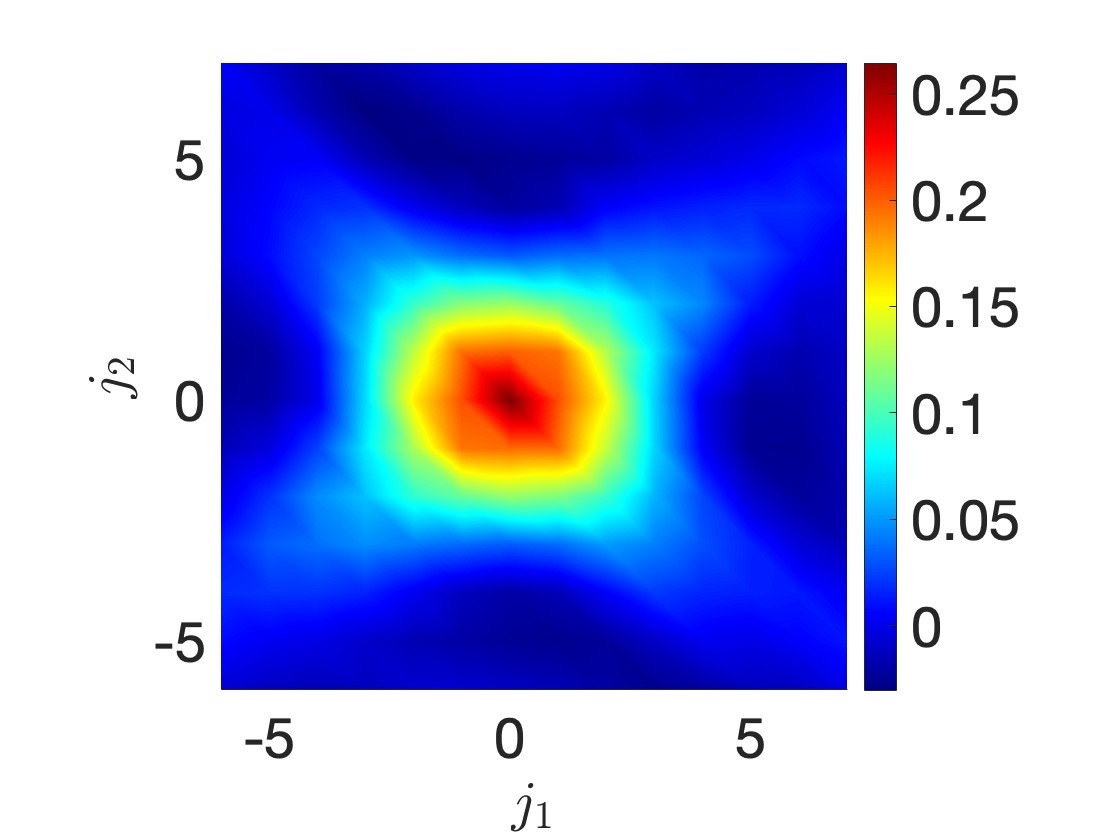}
        \caption{$\mathrm{Re}(\widehat{f}_{\Theta}^{\ast})$
        }
        \label{predFreal:kite}
    \end{subfigure}
    \hspace{-0.5cm}
    \begin{subfigure}[b]{0.25\textwidth}
        \includegraphics[width=\linewidth]{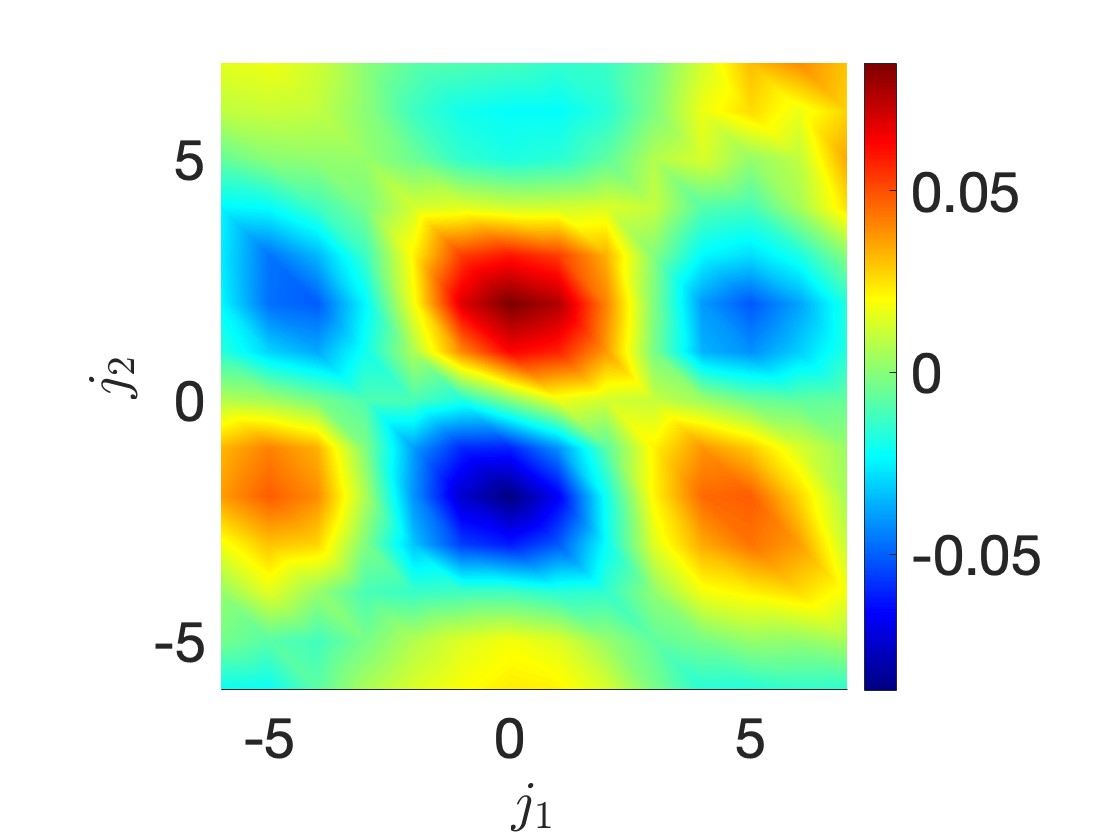}
        \caption{$\mathrm{Im}(\widehat{f}_{\Theta}^{\ast})$
        }
        \label{predFimag:kite}
    \end{subfigure}

    \medskip
    
    \begin{subfigure}[b]{0.3\textwidth}
        \includegraphics[width=\linewidth]{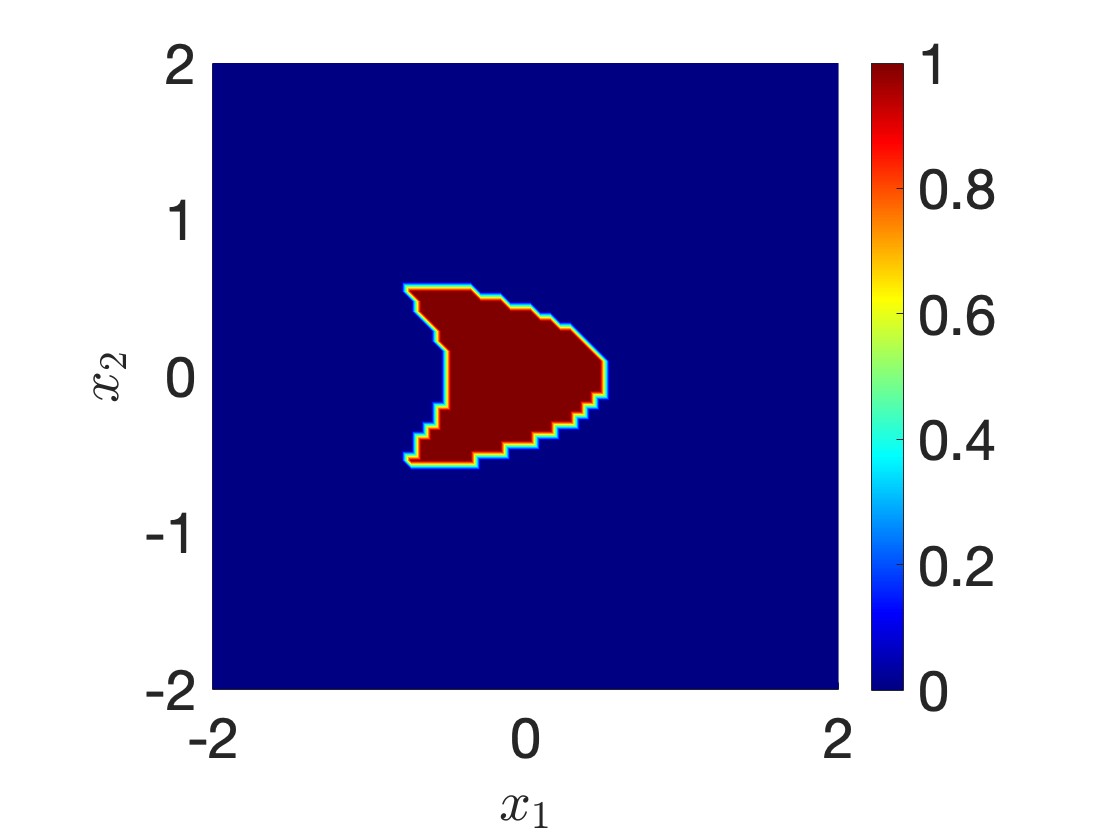}
        \caption{$f$}
        \label{true:kite}
    \end{subfigure}
    \begin{subfigure}[b]{0.3\textwidth}
        \includegraphics[width=\linewidth]{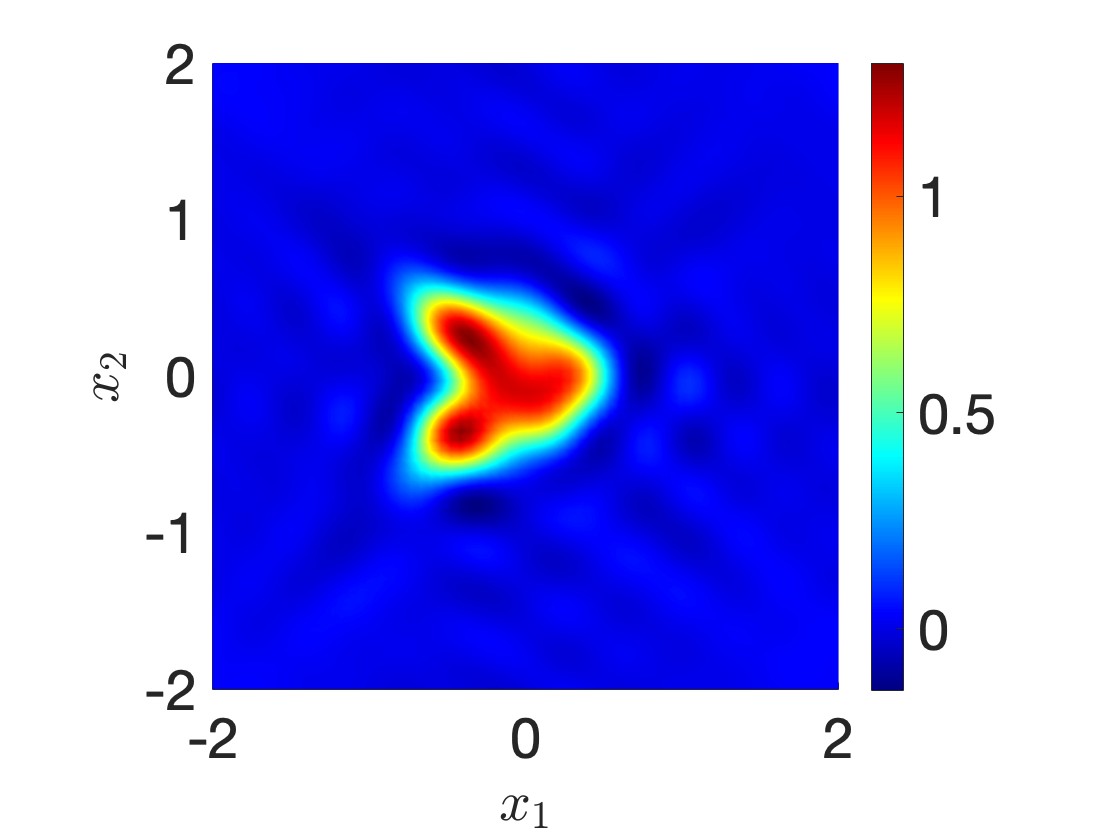}
        \caption{$f_{\text{comp}}$ 
        }
        \label{pred:kite}
    \end{subfigure}
    \begin{subfigure}[b]{0.3\textwidth}
            \includegraphics[width=1.0\linewidth]{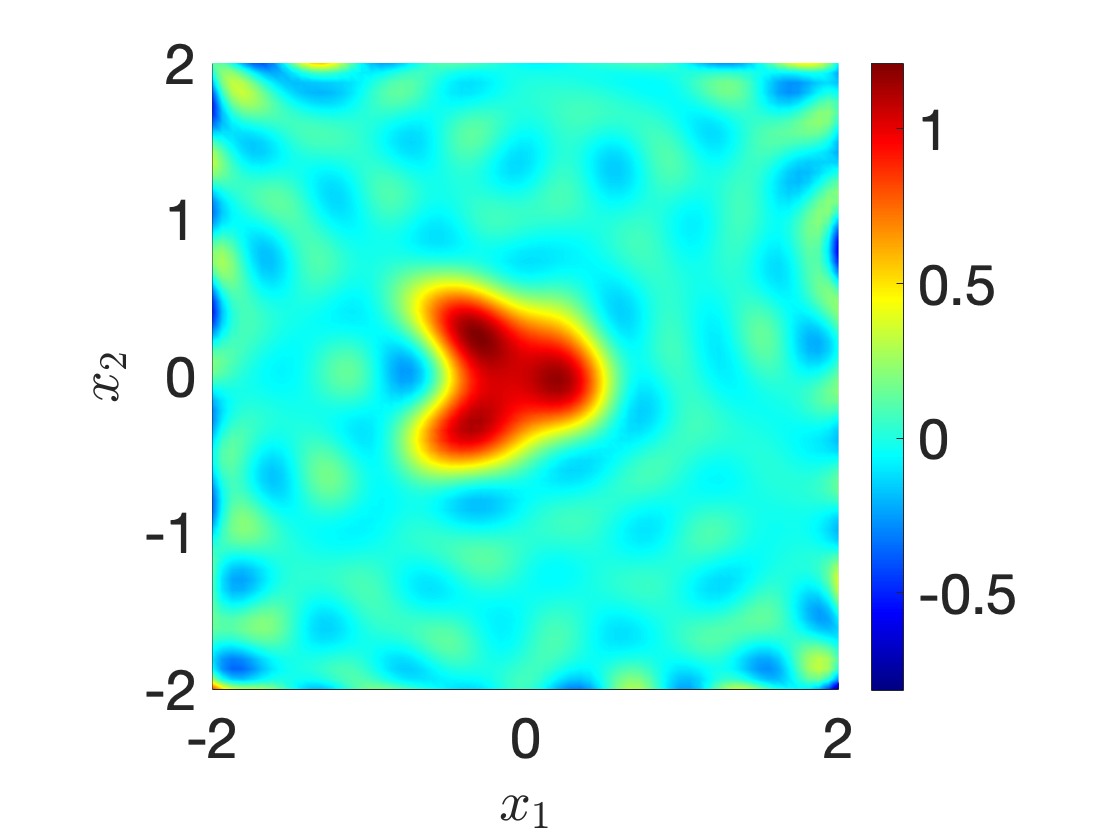}
            \caption{Linear least squares}
            \label{lls:kite}
        \end{subfigure}
\caption{Reconstruction results for the kite-shaped source. 
Top row: real \eqref{trueFreal:kite} and imaginary \eqref{trueFimag:kite} parts of the exact Fourier coefficients and their counterparts predicted by the proposed method \eqref{predFreal:kite}, \eqref{predFimag:kite}. 
Bottom row: comparison of the true source \eqref{true:kite} and the reconstructions obtained by the proposed method \eqref{pred:kite} and the linear least-squares approach \eqref{lls:kite}.}
    \label{fig: New Kite Results}
\end{figure}

\begin{figure}[H]
    \begin{subfigure}[b]{0.25\textwidth}
        \includegraphics[width=\linewidth]{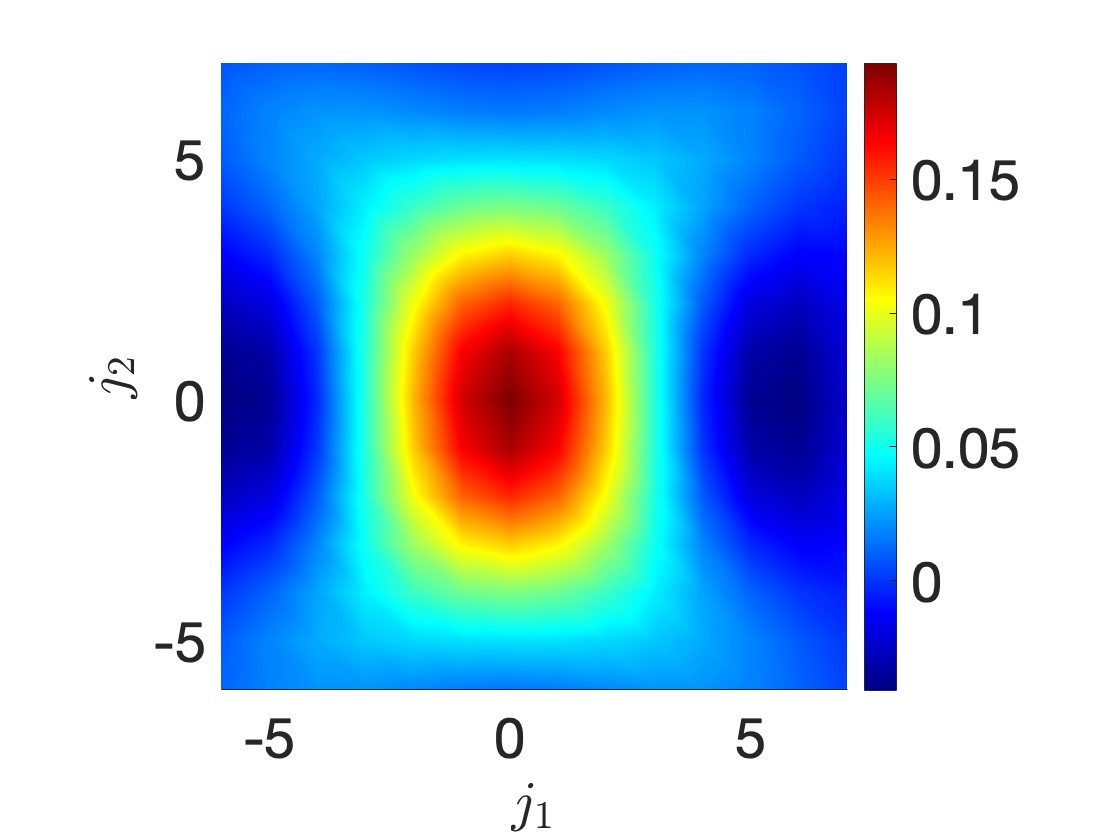}
        \caption{$\mathrm{Re}(\widehat{f}_{\text{per}})$}
        \label{trueFreal:tri}
    \end{subfigure}
    \hspace{-0.5cm}
    \begin{subfigure}[b]{0.25\textwidth}
        \includegraphics[width=\linewidth]{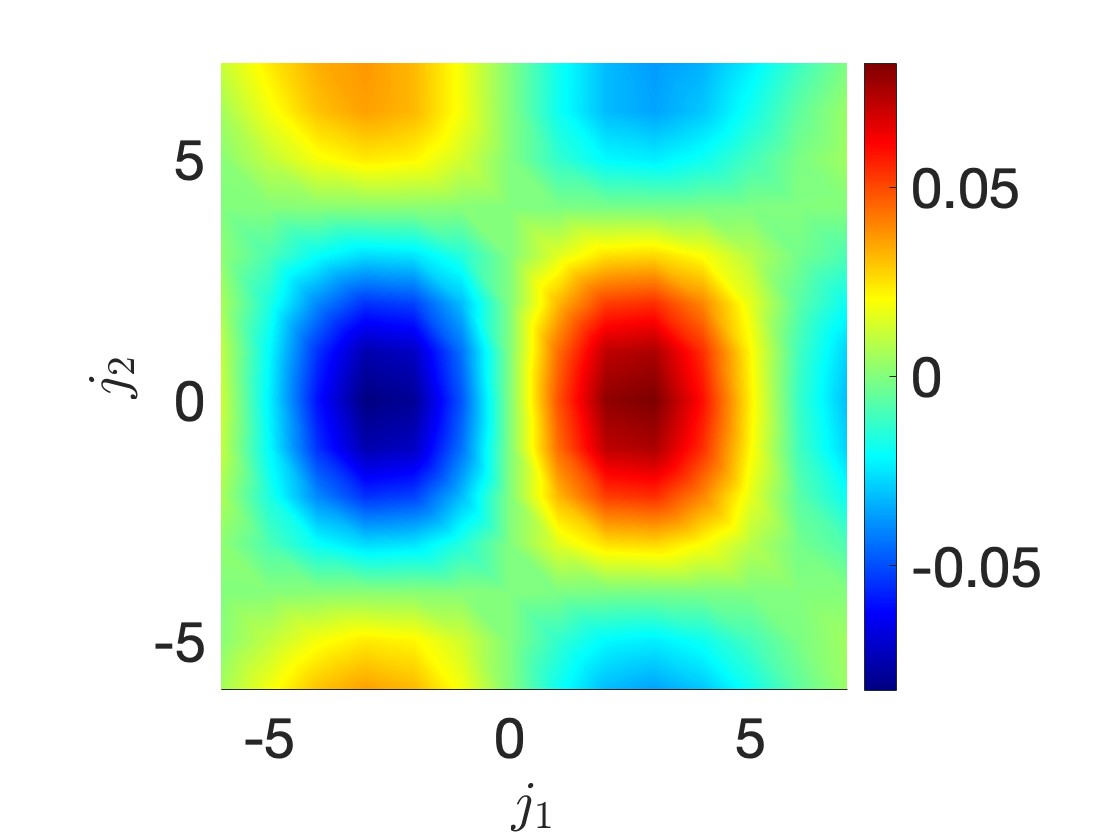}
        \caption{$\mathrm{Im}(\widehat{f}_{\text{per}})$}
        \label{trueFimag:tri}
    \end{subfigure}
    \hspace{-0.5cm}
    \begin{subfigure}[b]{0.25\textwidth}
        \includegraphics[width=\linewidth]{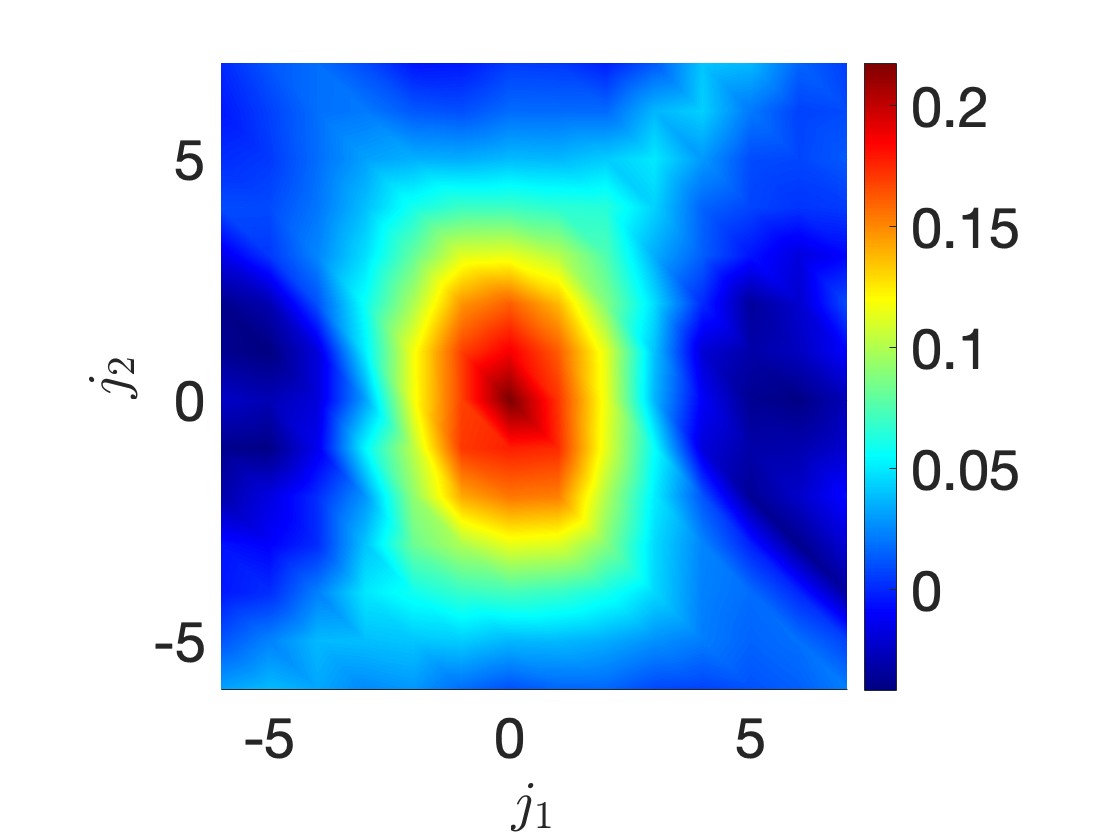}
        \caption{$\mathrm{Re}(\widehat{f}_{\Theta}^{\ast})$
        }
        \label{predFreal:tri}
    \end{subfigure}
    \hspace{-0.5cm}
    \begin{subfigure}[b]{0.25\textwidth}
        \includegraphics[width=\linewidth]{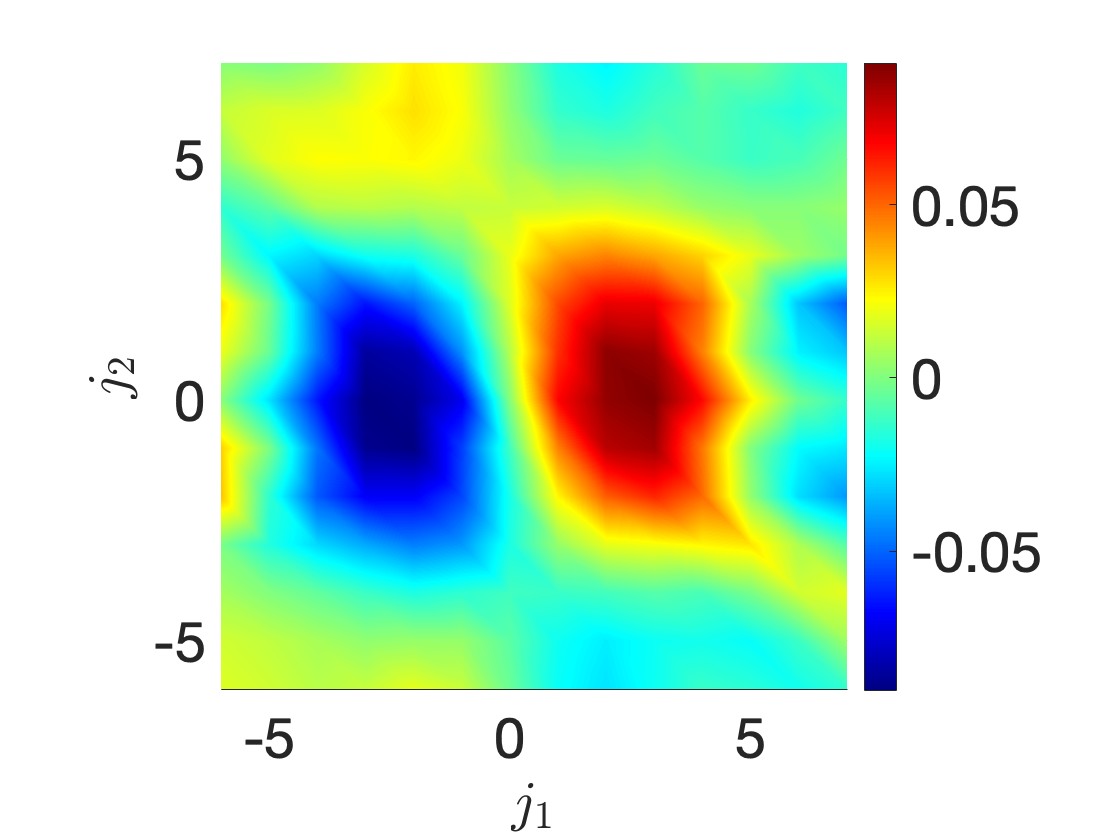}
        \caption{$\mathrm{Im}(\widehat{f}_{\Theta}^{\ast})$
        }
        \label{predFimag:tri}
    \end{subfigure}

    \medskip

    \centering
    \begin{subfigure}[b]{0.3\textwidth}
        \includegraphics[width=\linewidth]{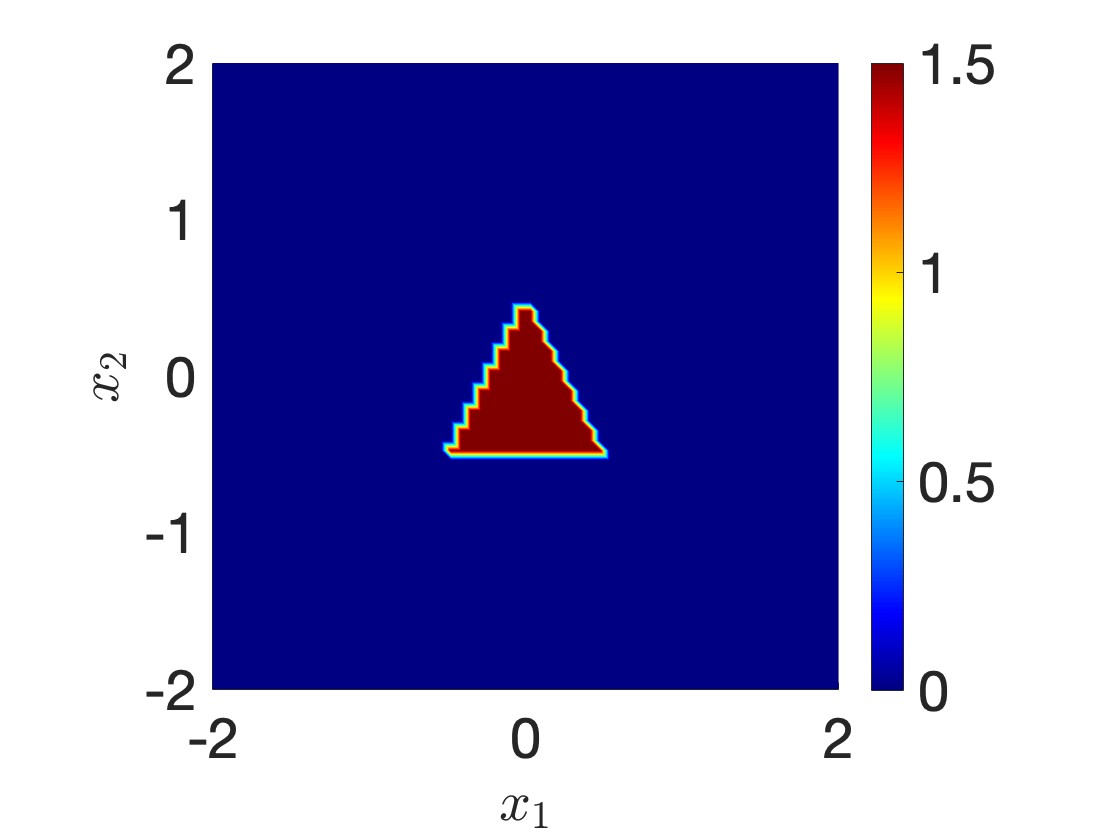}
        \caption{ $f$}
        \label{true:tri}
    \end{subfigure}
    \begin{subfigure}[b]{0.3\textwidth}
        \includegraphics[width=\linewidth]{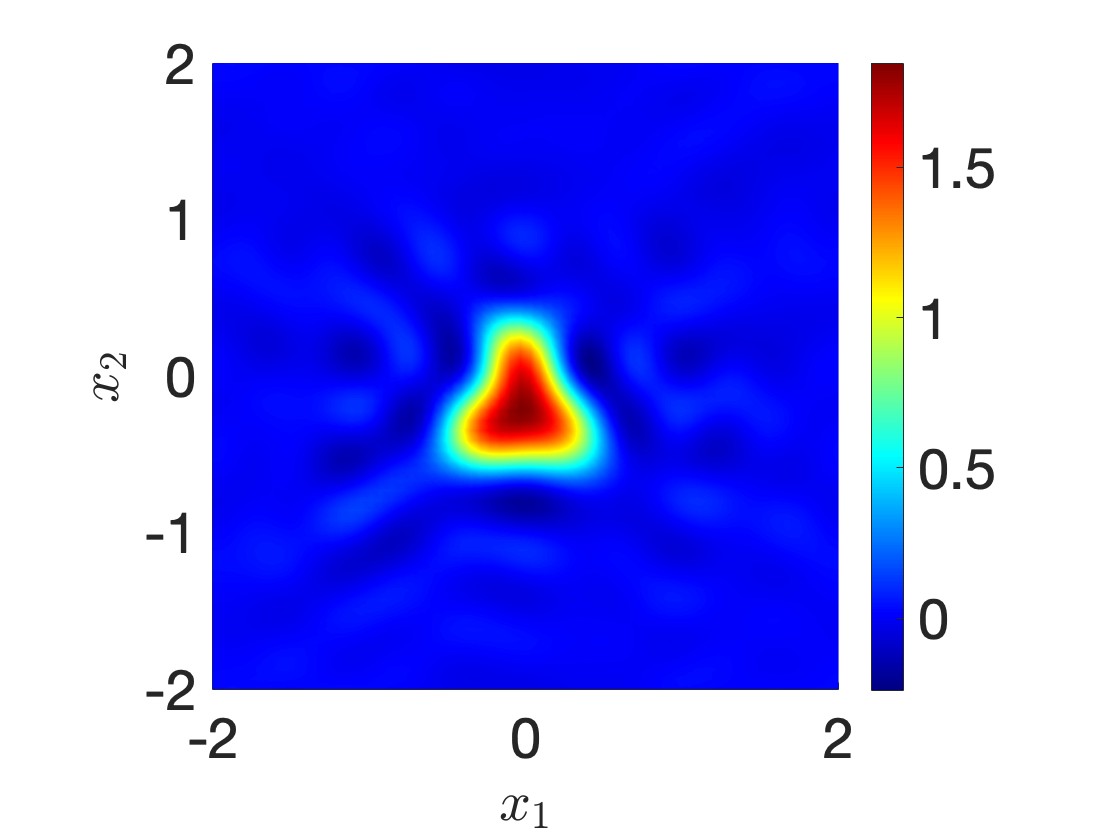}
        \caption{$f_{\text{comp}}$
        }
        \label{pred:tri}
    \end{subfigure}
    \begin{subfigure}[b]{0.3\textwidth}
            \includegraphics[width=1.0\linewidth]{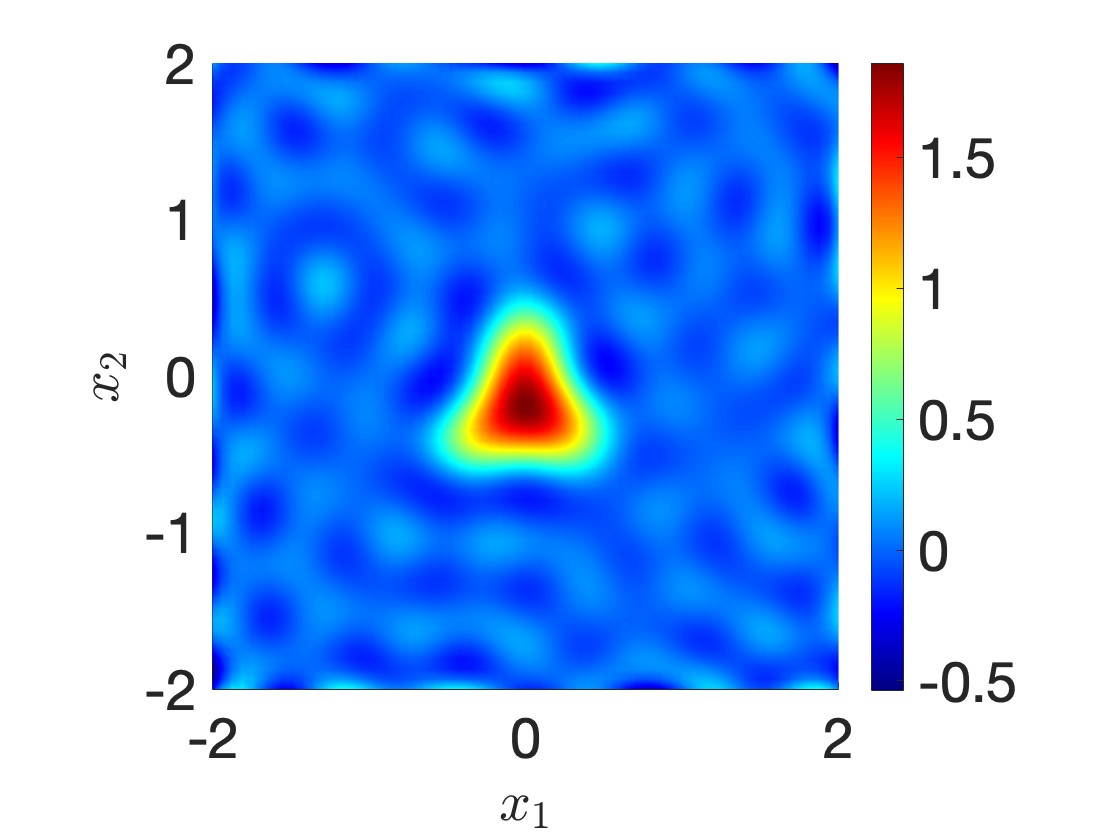}
            \caption{Linear least squares}
            \label{lls:tri}
        \end{subfigure}
   \caption{Reconstruction results for the triangle-shaped source. 
Top row: real \eqref{trueFreal:tri} and imaginary \eqref{trueFimag:tri} parts of the exact Fourier coefficients and their counterparts predicted by the proposed method \eqref{predFreal:tri}, \eqref{predFimag:tri}. 
Bottom row: comparison of the true source \eqref{true:tri} and the reconstructions obtained by the proposed method \eqref{pred:tri} and the linear least-squares approach \eqref{lls:tri}.}
    \label{fig: New Triangle Results}
\end{figure}

\begin{figure}[H]
\centering
    \begin{subfigure}[b]{0.3\textwidth}
        \includegraphics[width=\linewidth]{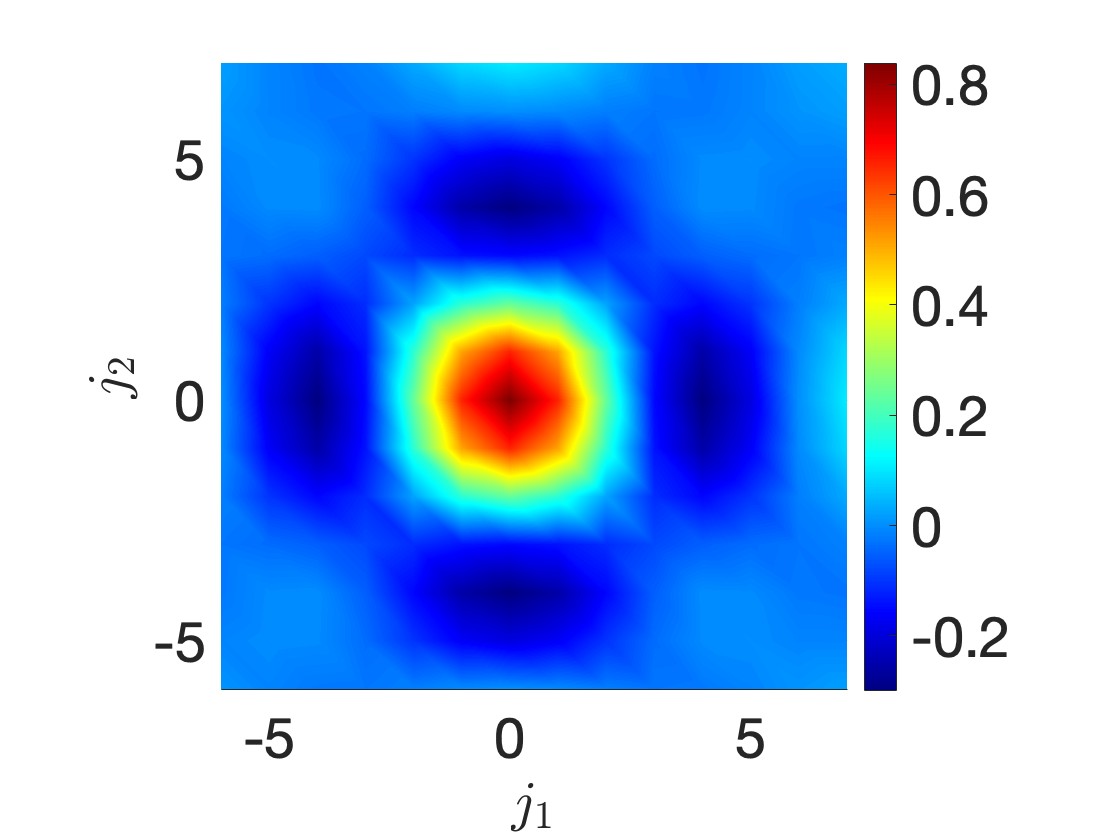}
        \caption{$\mathrm{Re}(\widehat{f}_{\text{per}})$ ($\mathrm{Im}(\widehat{f}_{\text{per}})=0)$}
        \label{trueF:squareh}
    \end{subfigure}
    \begin{subfigure}[b]{0.3\textwidth}
        \includegraphics[width=\linewidth]{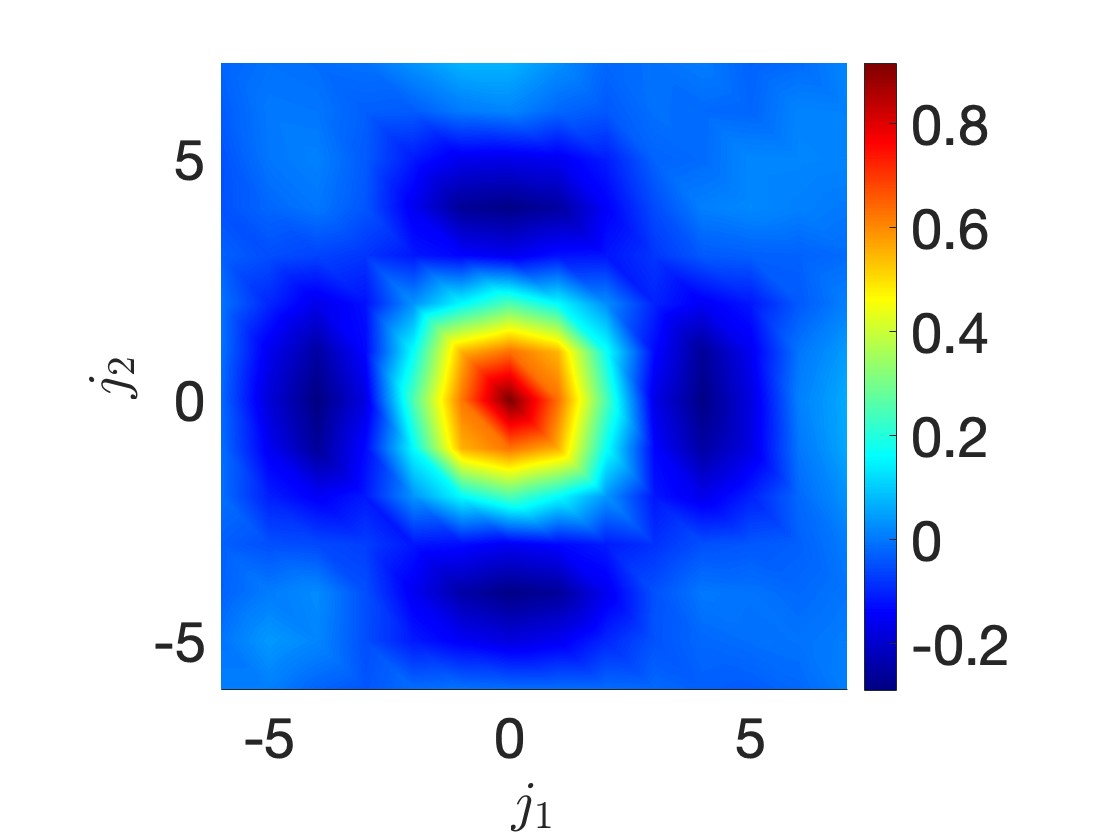}
        \caption{$\mathrm{Re}(\widehat{f}_{\Theta}^{\ast})$ 
        }
        \label{predFreal:squareh}
    \end{subfigure}
    \begin{subfigure}[b]{0.3\textwidth}
        \includegraphics[width=\linewidth]{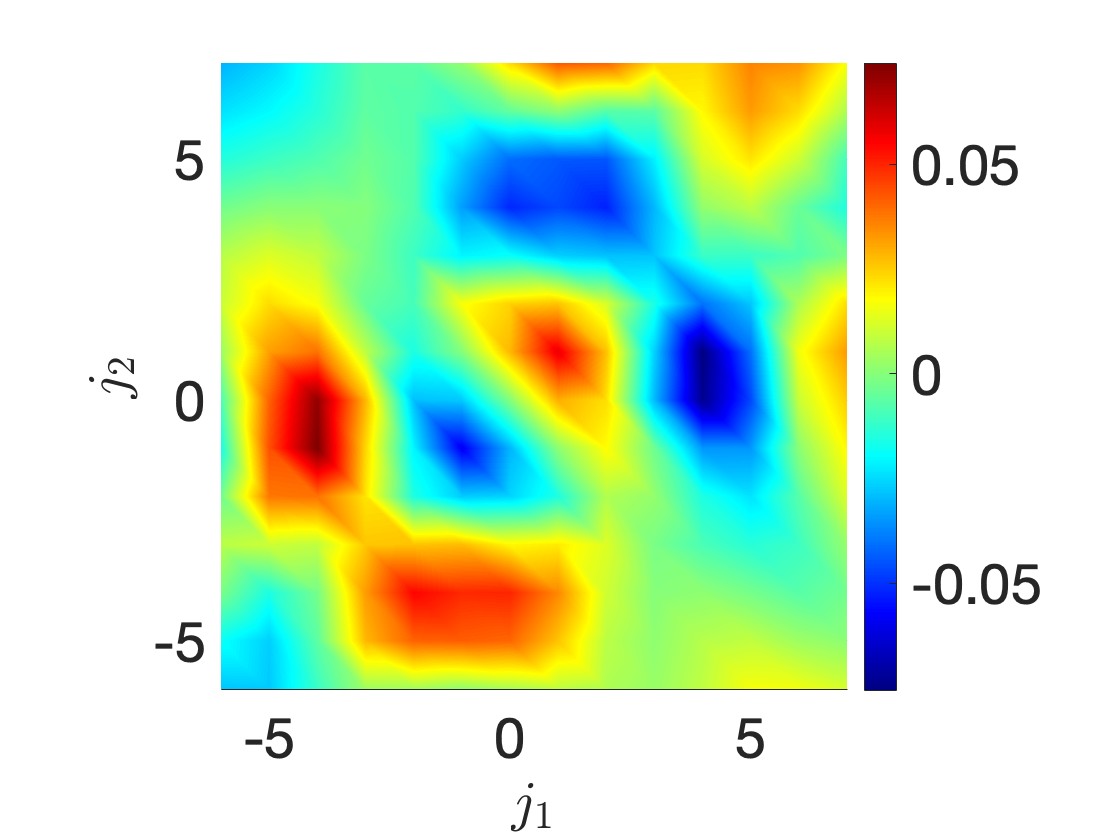}
        \caption{$\mathrm{Im}(\widehat{f}_{\Theta}^{\ast})$
        }
        \label{predFimag:squareh}
    \end{subfigure}

    \centering
    \begin{subfigure}[b]{0.3\textwidth}
        \includegraphics[width=\linewidth]{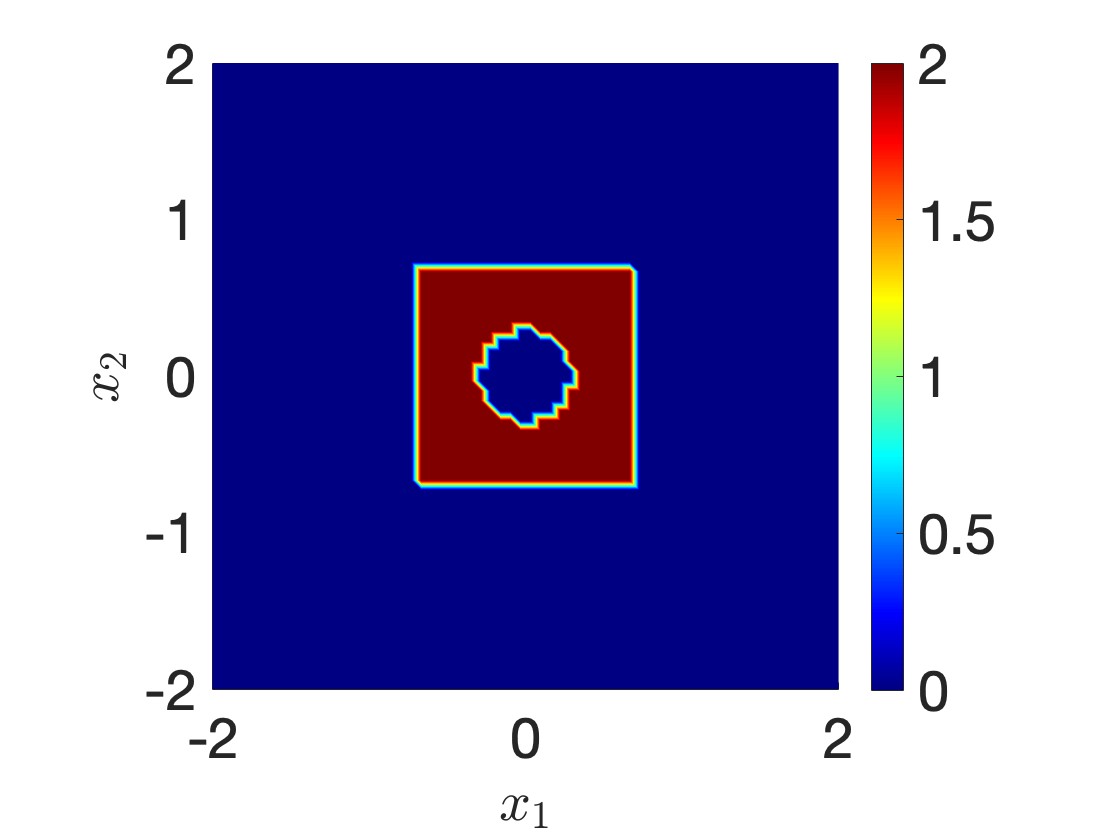}
        \caption{ $f$}
        \label{true:squareh}
    \end{subfigure}
    \begin{subfigure}[b]{0.3\textwidth}
        \includegraphics[width=\linewidth]{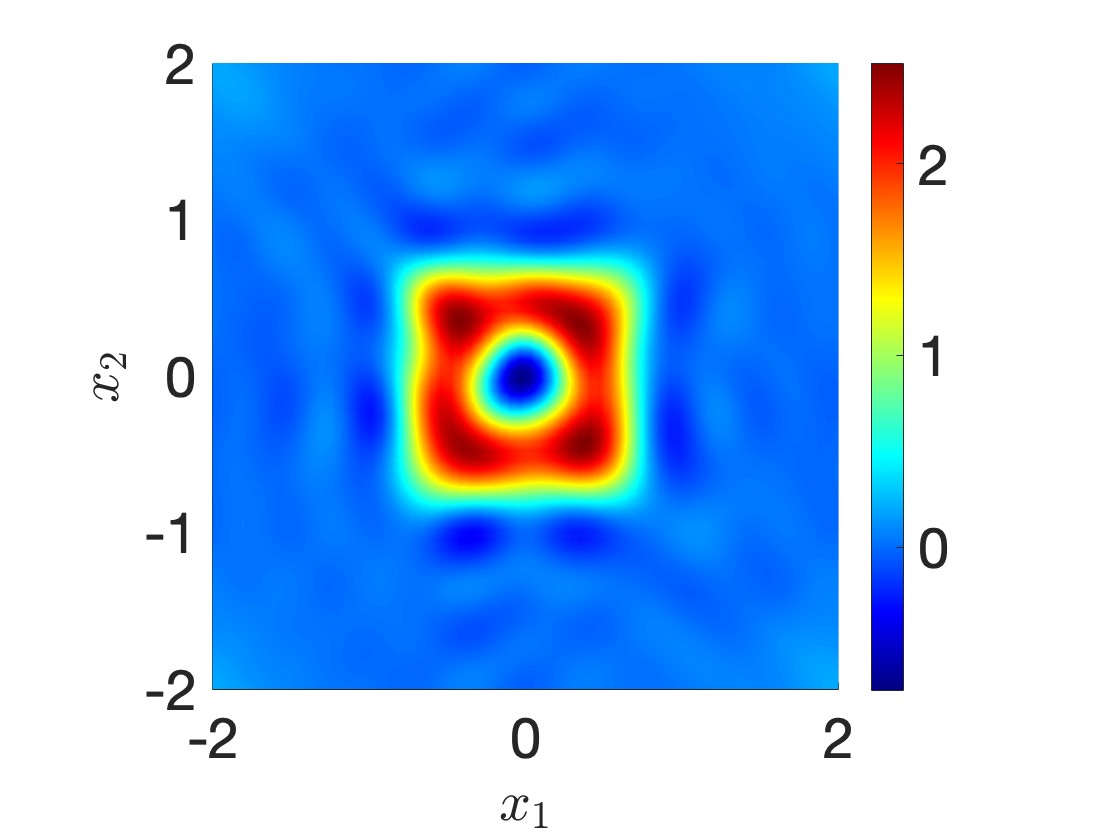}
        \caption{$f_{\text{comp}}$  
        }
        \label{pred:squareh}
    \end{subfigure}
    \begin{subfigure}[b]{0.3\textwidth}
            \includegraphics[width=1.0\linewidth]{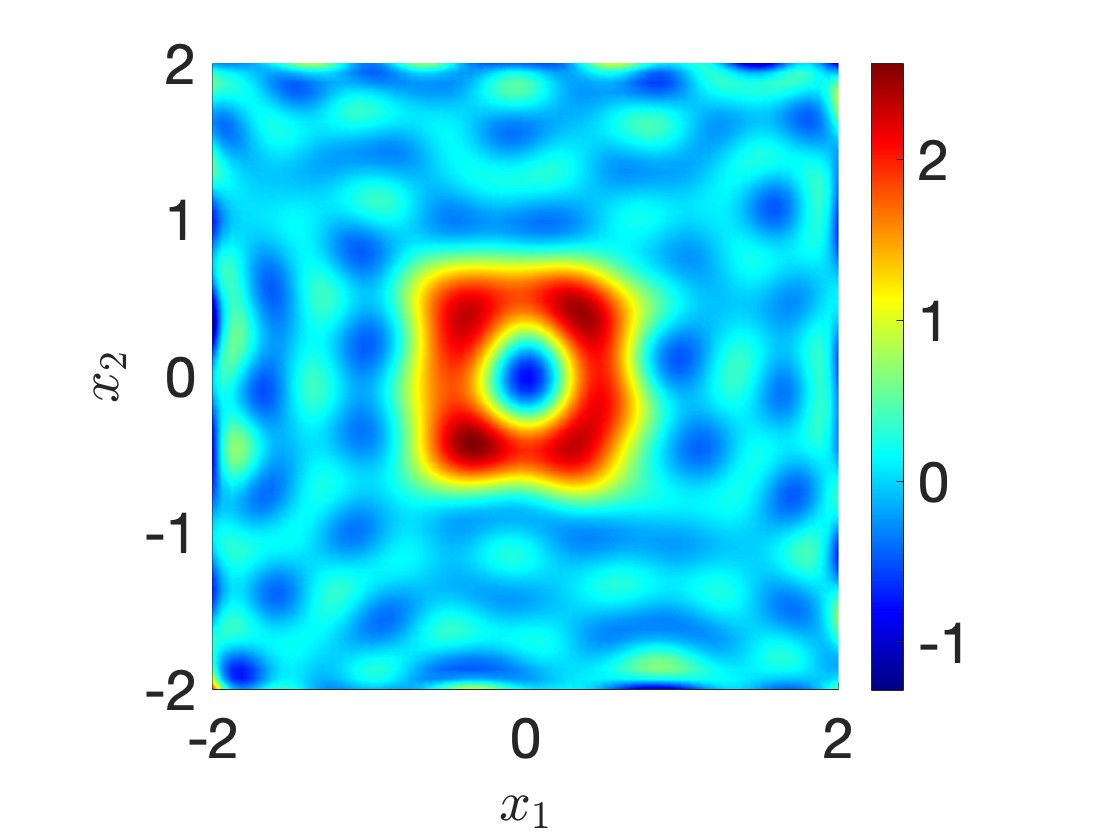}
            \caption{Linear least squares}
            \label{lls:squareh}
        \end{subfigure}
   \caption{Reconstruction results for the square-with-hole-shaped source. 
Top row: exact Fourier coefficients \eqref{trueF:squareh} and their real \eqref{predFreal:squareh} and imaginary \eqref{predFimag:squareh} parts predicted by the proposed method. 
Bottom row: comparison of the true source \eqref{true:squareh} and the reconstructions obtained by the proposed method \eqref{pred:squareh} and the linear least-squares approach \eqref{lls:squareh}.}
    \label{fig: New Rect with Hole Results}
\end{figure}

\begin{table}[H]
    \centering
    \begin{tabular}{|c|c|c|c|c|c|}
    \hline
    Source  &  $\mathcal{E}_{\text{Fourier}}\left[\widehat{f}_{\Theta}^{\ast}\right]$ & Computation time \\
    \hline
    Square  & 13.82\%  & 2 min 41 sec \\
    \hline
    Kite &  13.34\%  & 3 min 4 sec \\
    \hline
    Triangle & 13.92\%  & 2 min 19 sec \\
    \hline
    Rect with Hole  & 10.97\%  & 4 min 39 sec \\
    \hline
    \end{tabular}
    \caption{Relative error and computational time for $2$D reconstructions}
    \label{table: new 2D table}
\end{table}

\begin{figure}[H]
    \centering
   
    \begin{subfigure}[b]{0.4\textwidth}
        \includegraphics[width=\linewidth]{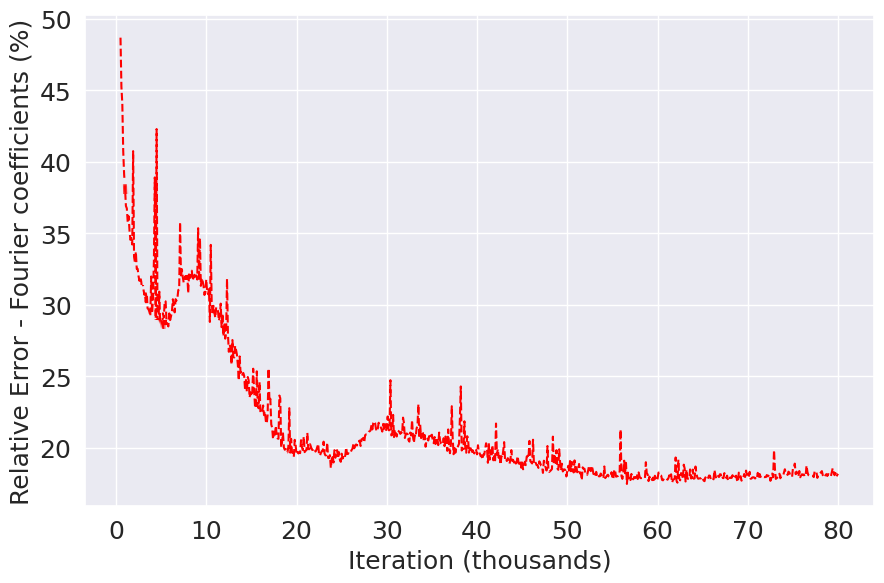}
        \caption{Relative error for  kite-shaped source}
       
        \label{fig: Relative Error against epochs - kite}
    \end{subfigure}
    \begin{subfigure}[b]{0.4\textwidth}
        \includegraphics[width=\linewidth]{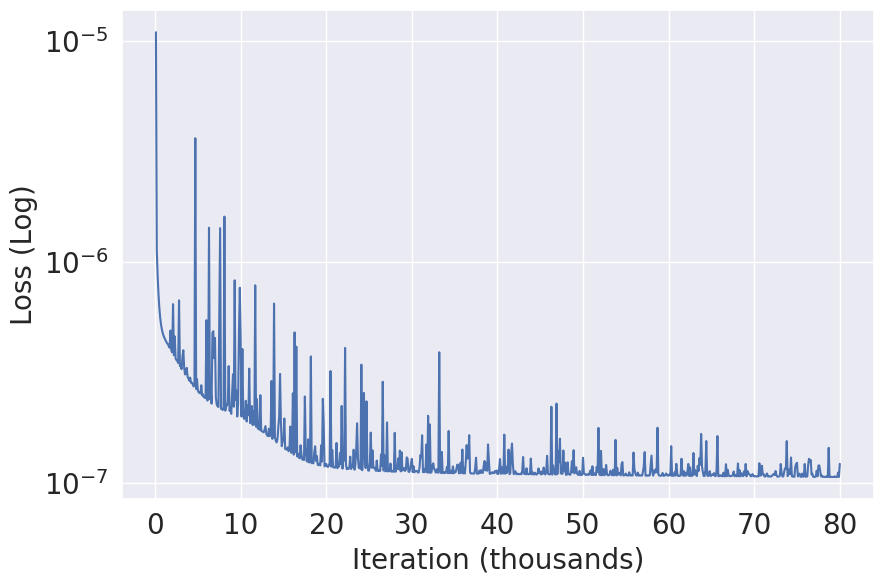}
        \caption{Log loss for kite-shaped source}
        \label{fig: Log loss against epochs - kite}
    \end{subfigure}
    \vspace{0.4 cm}
     \begin{subfigure}[b]{0.4\textwidth}        \includegraphics[width=\linewidth]{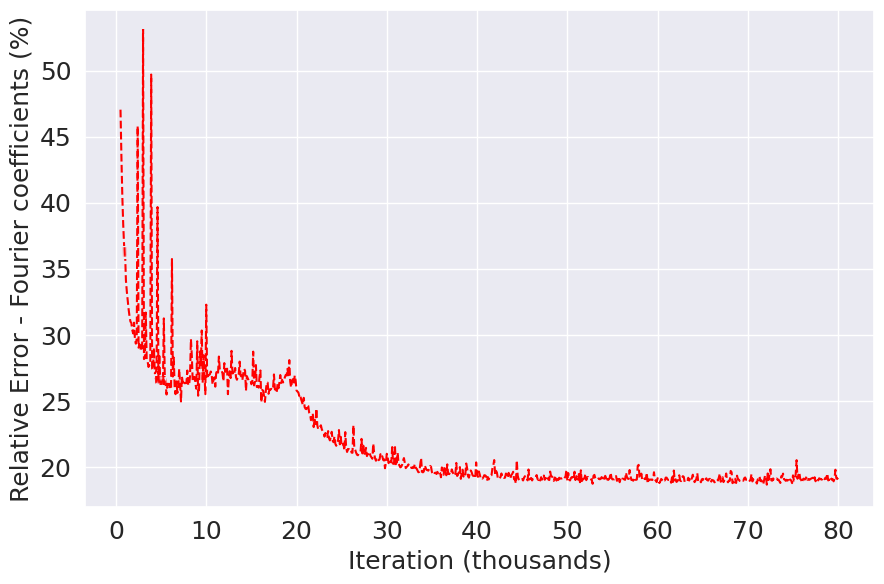}
    \caption{Relative error for  triangle-shaped source}
    \label{fig: Relative Error against epochs - Tri}
    \end{subfigure}
    \begin{subfigure}[b]{0.4\textwidth}
    \includegraphics[width=\linewidth]{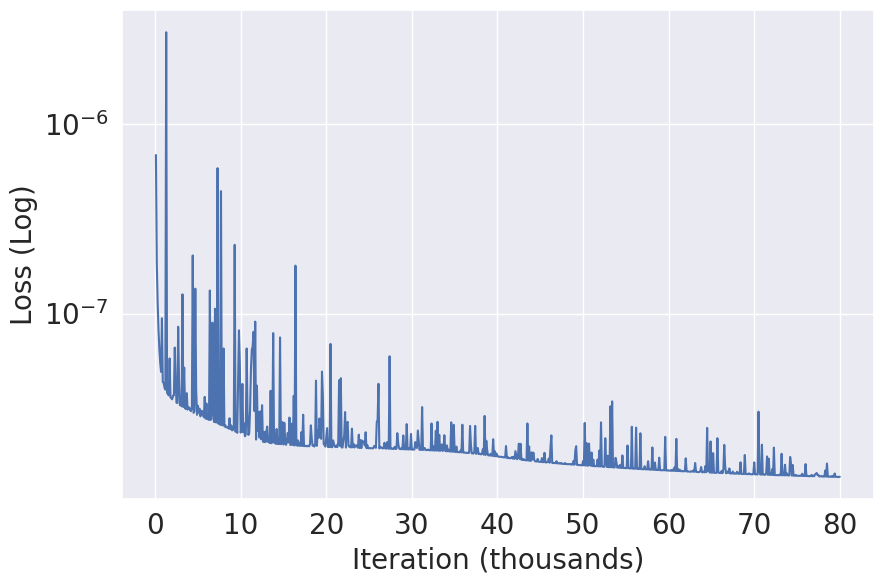}
        \caption{Log loss  for triangle-shaped source }
        \label{fig: Log loss against epochs - Tri}
    \end{subfigure}
    \caption{(a, c) Relative errors between the true and predicted Fourier coefficients plotted against iteration number. (b, d) Evolution of the loss functions plotted against iteration number.} 
    \label{fig:Update process and Loss curve (2D) - Kite}
\end{figure}

\subsection{$3$D experimental configuration and results} Similarly to the $2$D case, the algorithm yields a good reconstruction for the following $3$D sources relatively quickly: a cube-shaped source (Figure \ref{fig:new cube-3d}), a sphere-shaped source (Figure \ref{fig:new sphere-3d}), a pyramid-shaped source (Figure \ref{fig:new pyramid-3d}). The algorithm is able to accurately identify the boundaries, locations, sizes, and coefficient values of the source functions. This is numerically substantiated by Table \ref{table: new 3D table}. The relatively short time taken (Table \ref{table: new 3D table}) is attributed to the low computational complexity of $\mathcal{O}(M N^{3} \log N)$. By defining the algorithm in the spectral domain, we overcome the need for $N^{6}$ operations when evaluating the convolutional structure in the equation \ref{eq:Iz}.  The iso-value for the iso-surface plots is set to be $30\%$ of the peak-value of the reconstruction.

\begin{figure}[H]
\centering

 \begin{subfigure}{0.325\textwidth}
            \includegraphics[scale = 0.14]{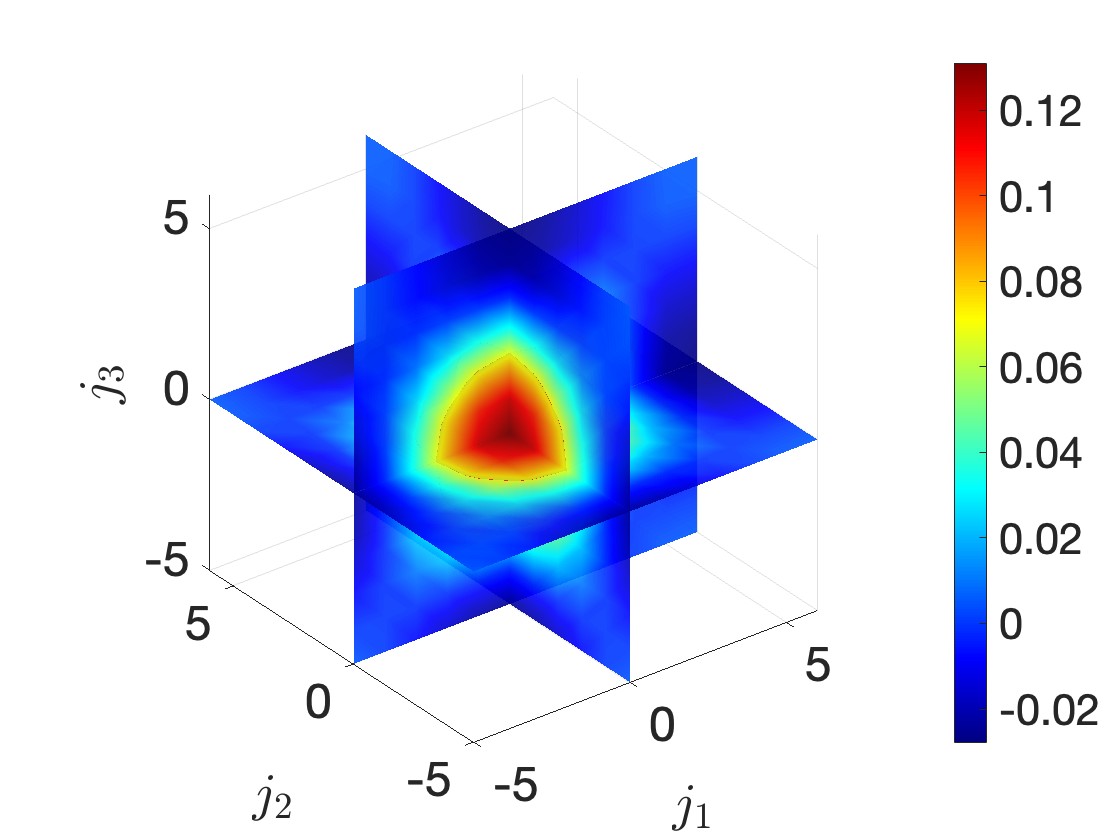}
            \caption{$\mathrm{Re}(\widehat{f}_{\text{per}})$ $(\mathrm{Im}(\widehat{f}_{\text{per}})=0)$}
            \label{trueF:cube}
        \end{subfigure}
                \begin{subfigure}{0.325\textwidth}
            \includegraphics[scale = 0.14]{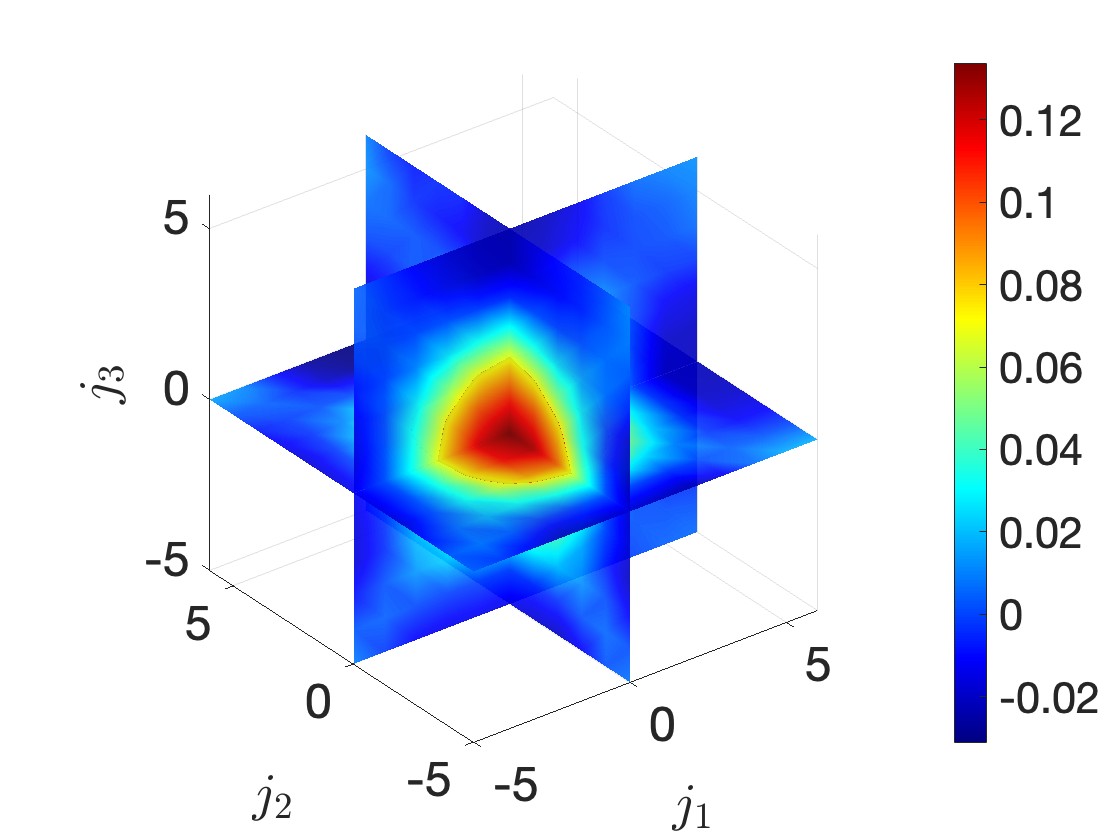}
            \caption{$\mathrm{Re}\left(\widehat{f}^{\ast}_{\Theta}\right)$}
            \label{predFreal:cube}
        \end{subfigure}
        \begin{subfigure}{0.325\textwidth}
            \includegraphics[scale = 0.14]{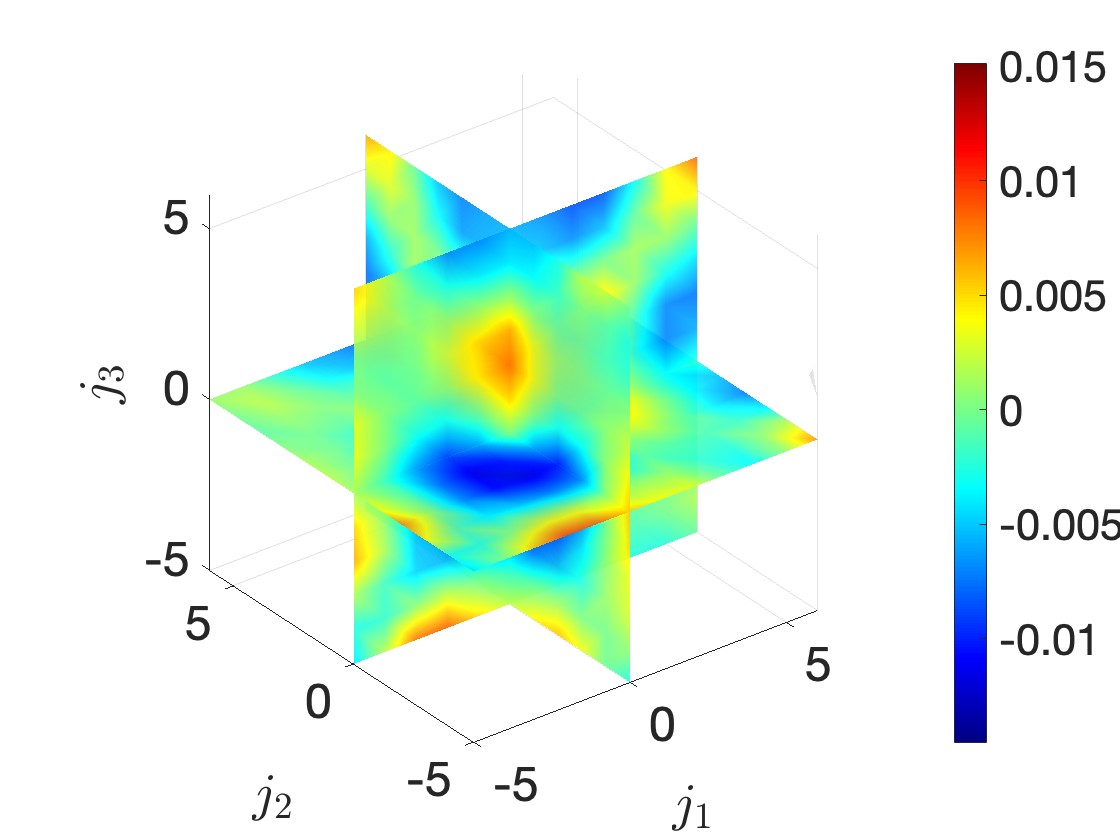}
            \caption{$\mathrm{Im}\left(\widehat{f}^{\ast}_{\Theta}\right)$}
            \label{predFimag:cube}
        \end{subfigure}

    \begin{center}

        \begin{subfigure}
        {0.325\textwidth}
            \includegraphics[scale = 0.14]{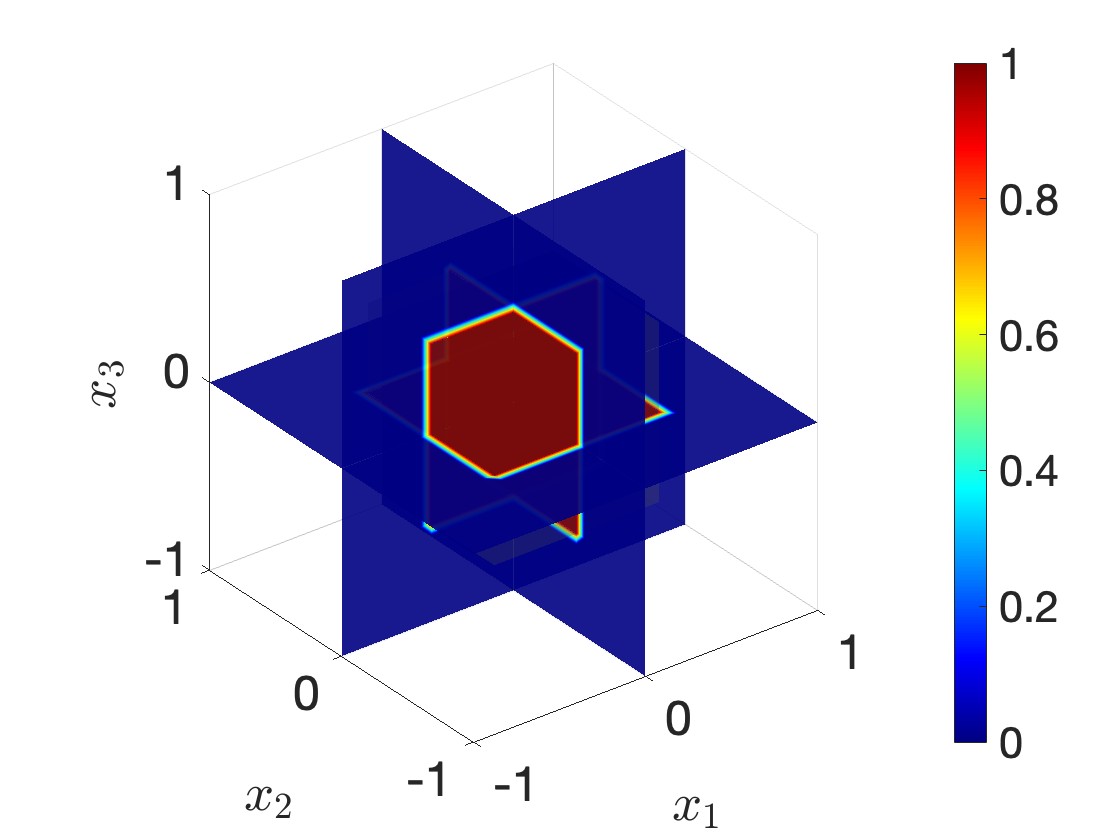}
            \caption{$f$}
            \label{truecross:cube}
 \end{subfigure}
            \begin{subfigure}
     {0.325\textwidth}
            \includegraphics[scale = 0.14]{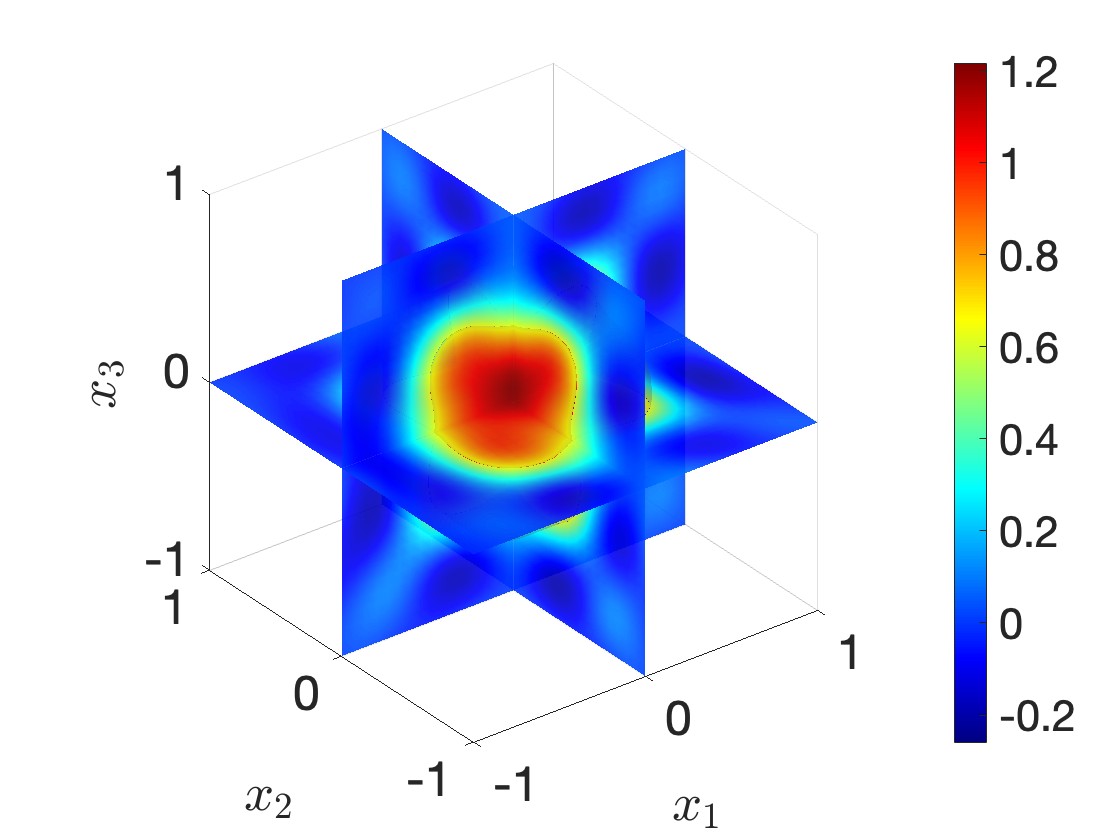}
        \caption{$f_{\text{comp}}$}
            \label{predcross:cube}
        \end{subfigure}
        
        \begin{subfigure}{0.325\textwidth}
            \includegraphics[scale = 0.15]{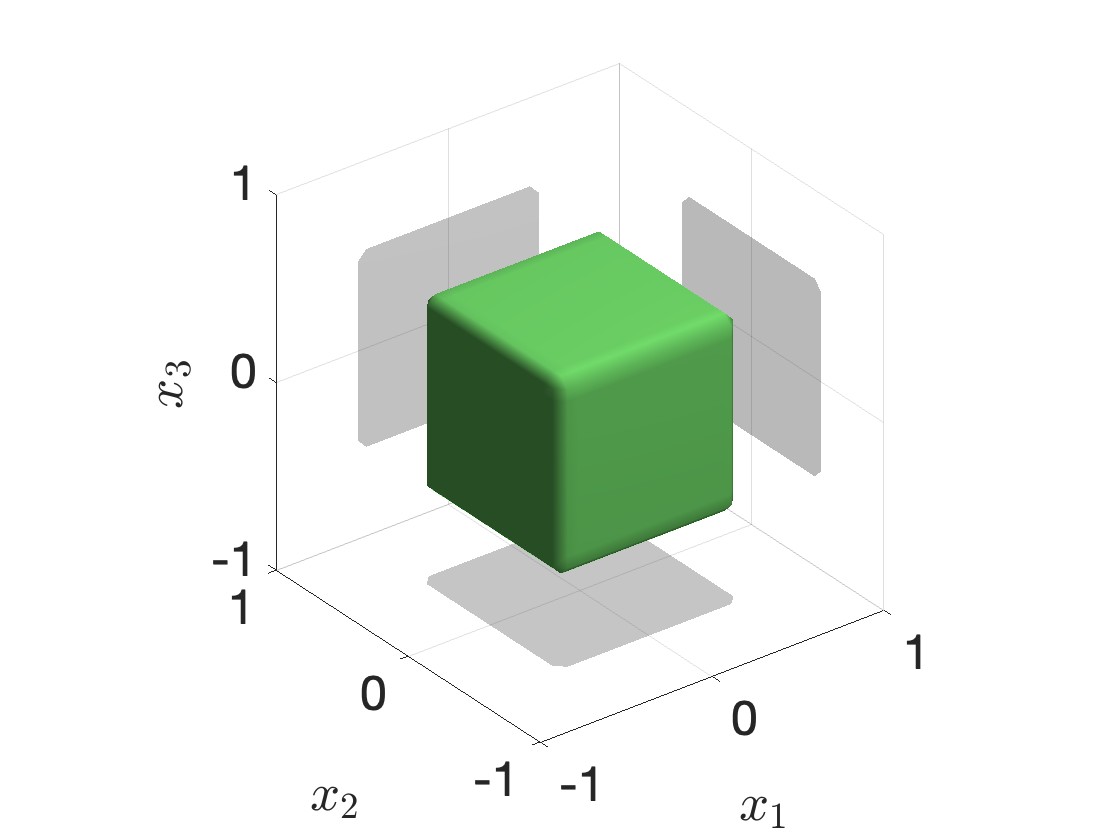}
            \caption{Iso-surface plot
            of (d)}
            \label{trueiso:cube}
        \end{subfigure}
          \begin{subfigure}{0.325\textwidth}
            \includegraphics[scale = 0.15]{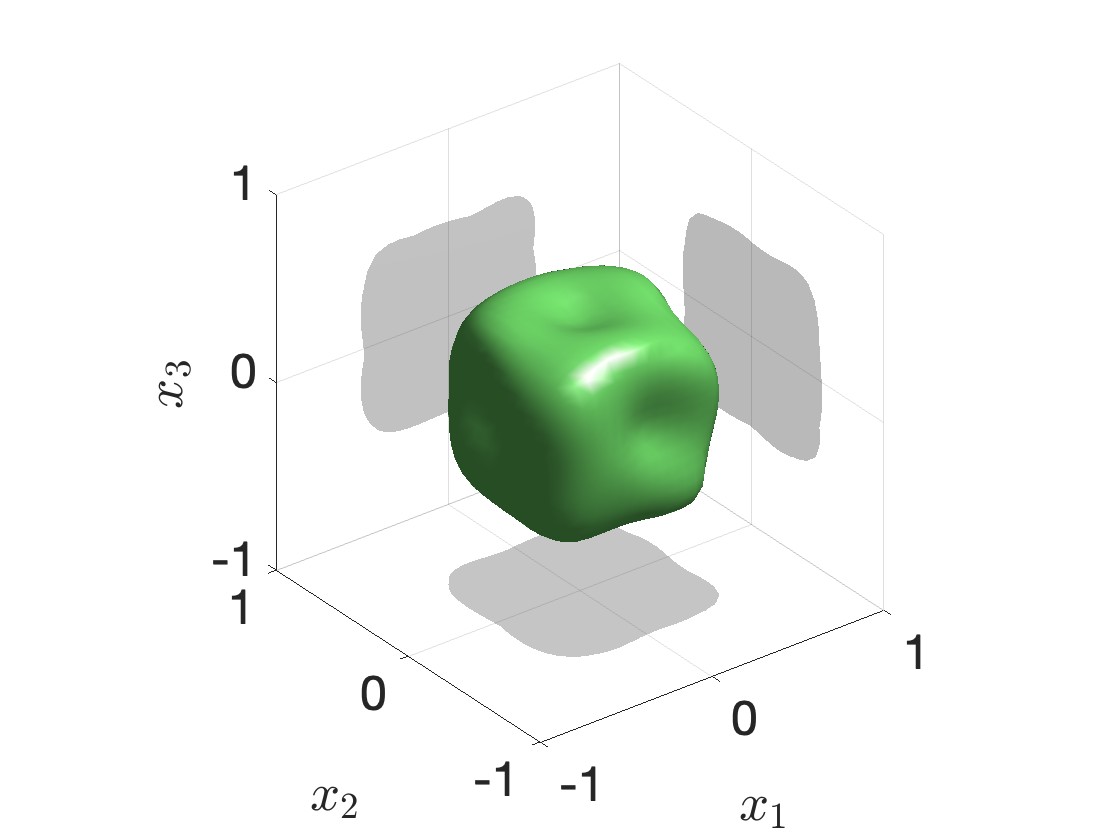}
            \caption{Iso-surface plot
            of (e)}
            \label{prediso:cube}
        \end{subfigure}

         \end{center}

 \caption{Reconstruction results for the cube-shaped source. 
Top row: exact Fourier coefficients \eqref{trueF:cube} and the real \eqref{predFreal:cube} and imaginary \eqref{predFimag:cube} parts of the coefficients predicted by the proposed method. 
Bottom rows: cross-sectional views and corresponding iso-surface representations of the true source \eqref{truecross:cube}, \eqref{trueiso:cube} and its reconstruction \eqref{predcross:cube}, \eqref{prediso:cube}.}

    \label{fig:new cube-3d}
\end{figure}

\begin{figure}[H]
    \begin{center}

    \begin{subfigure}{0.3\textwidth}
            \includegraphics[scale = 0.14]{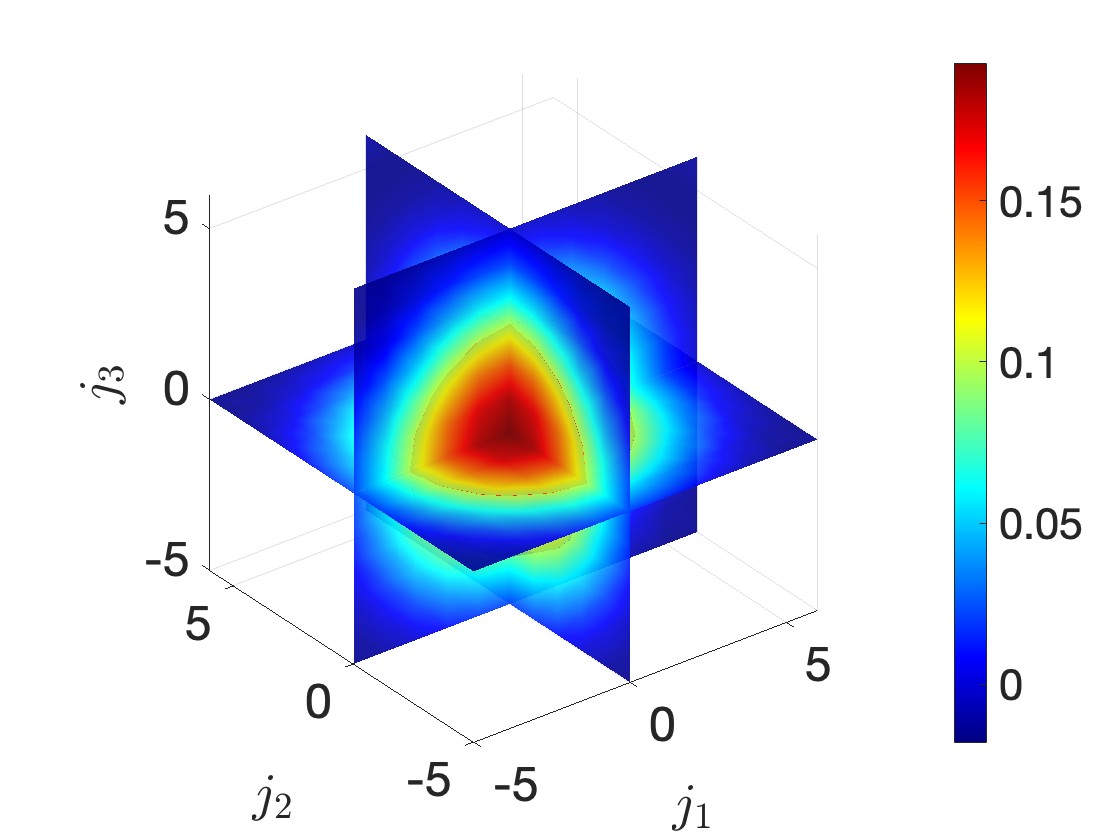}
            \caption{$\mathrm{Re}(\widehat{f}_{\text{per}})$ $(\mathrm{Im}(\widehat{f}_{\text{per}})=0)$}
            \label{trueF:sphere}
        \end{subfigure}
         \begin{subfigure}{0.3\textwidth}
            \includegraphics[scale = 0.14]{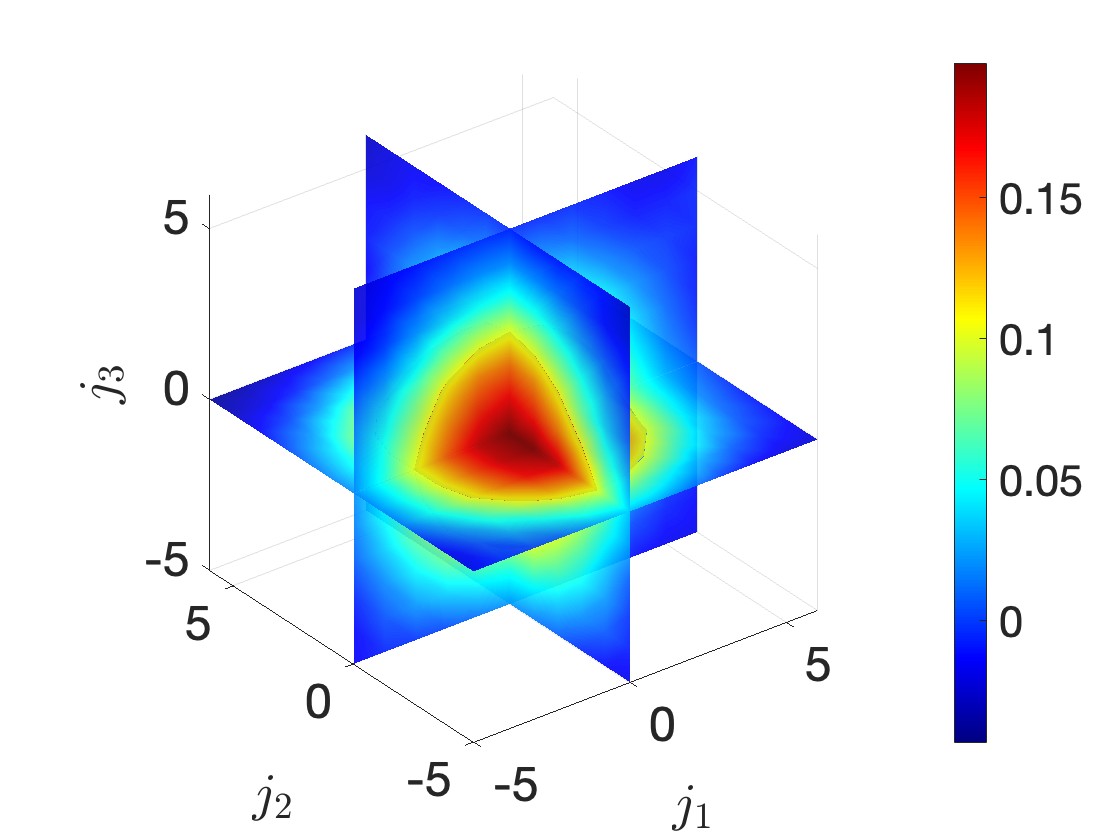}
            \caption{$\mathrm{Re}\left(\widehat{f}^{\ast}_{\Theta}\right)$}
            \label{predFreal:sphere}
        \end{subfigure}
        \begin{subfigure}{0.3\textwidth}
            \includegraphics[scale = 0.14]{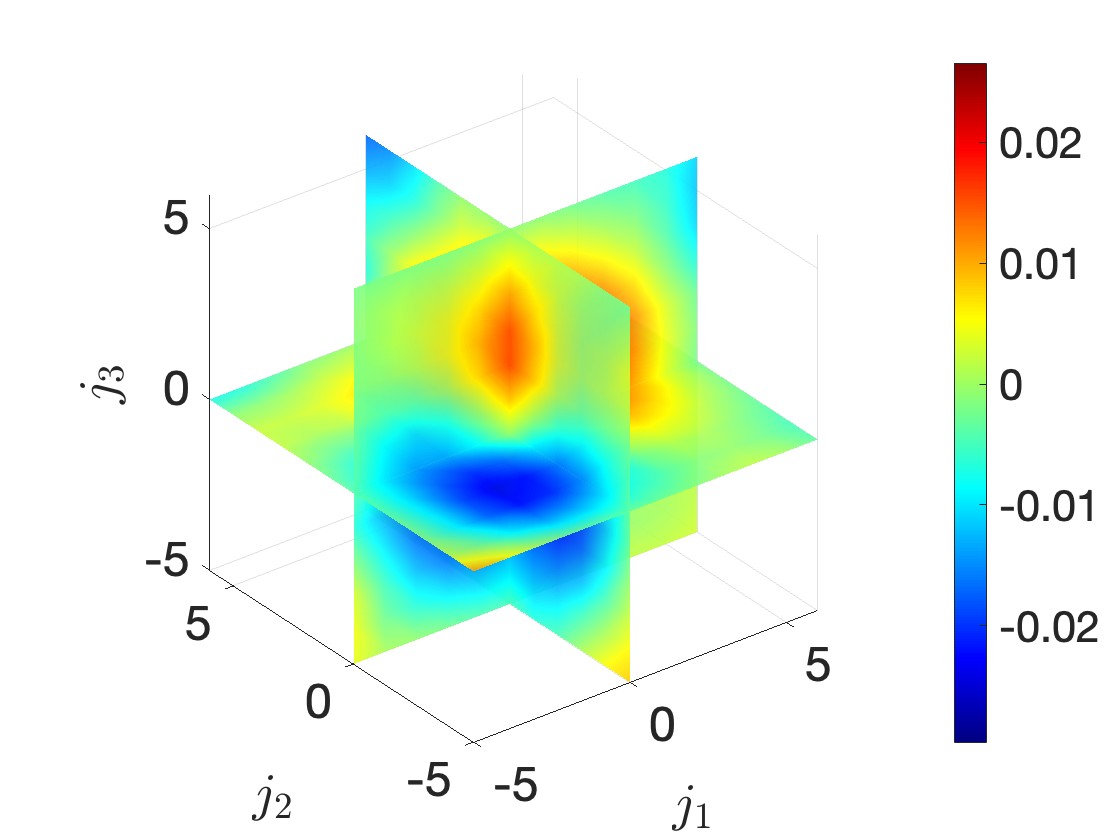}
            \caption{$\mathrm{Im}\left(\widehat{f}^{\ast}_{\Theta}\right)$}
            \label{predFimag:sphere}
        \end{subfigure}

   \begin{subfigure}{0.325\textwidth}
            \includegraphics[scale = 0.14]{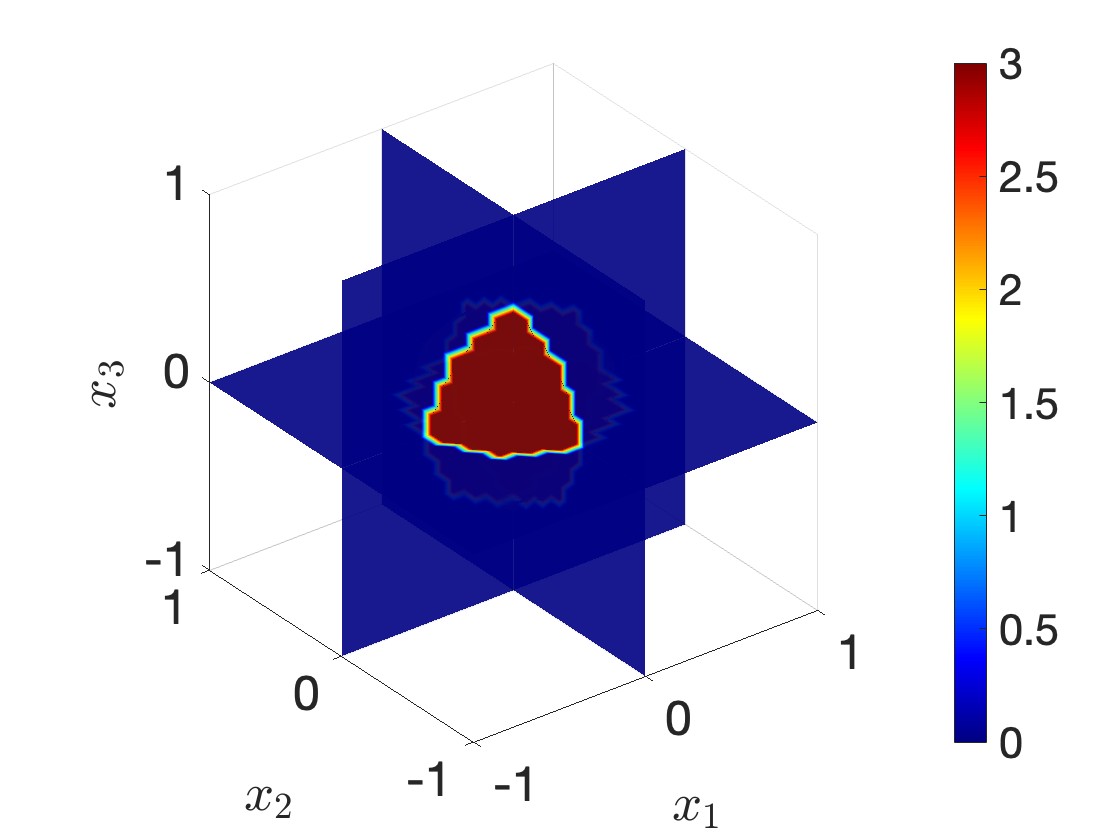}
            \caption{$f$}
            \label{truecross:sphere}
        \end{subfigure}
        \begin{subfigure}{0.325\textwidth}
            \includegraphics[scale = 0.14]{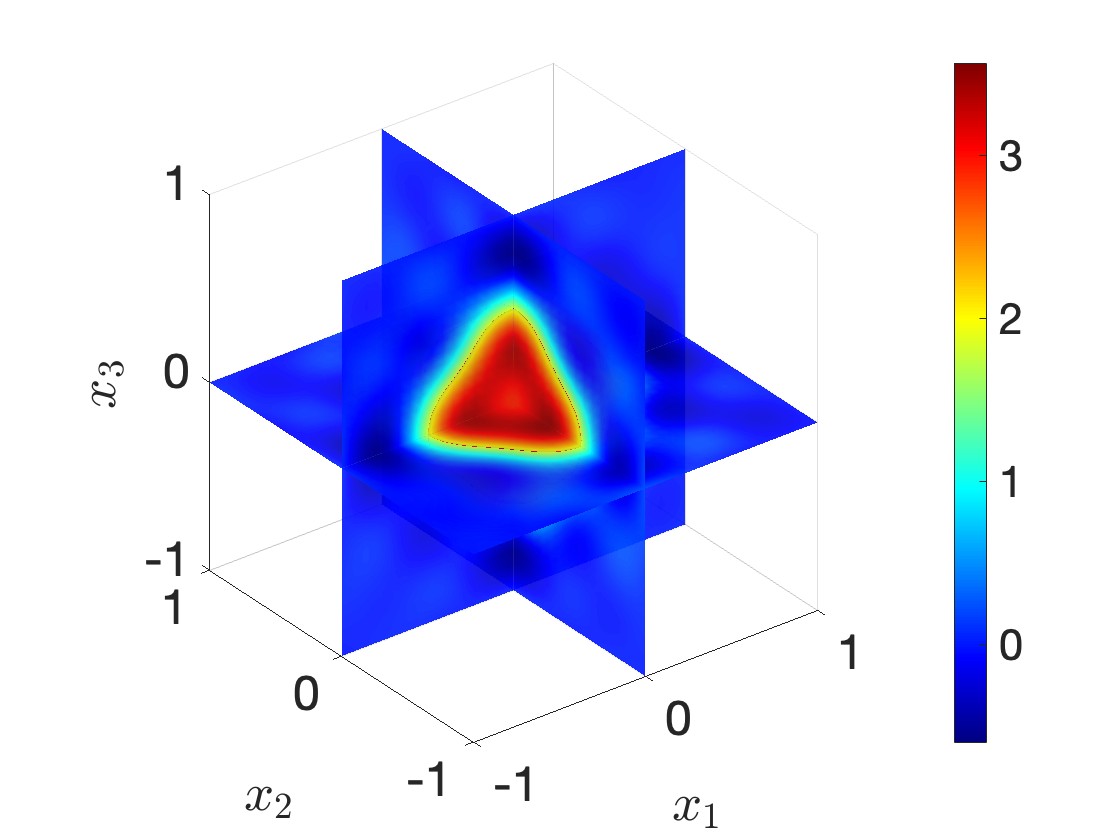}
            \caption{$f_{\text{comp}}$}
            \label{predcross:sphere}
        \end{subfigure}
        
        \begin{subfigure}[b]{0.325\textwidth}
            \includegraphics[scale = 0.14]{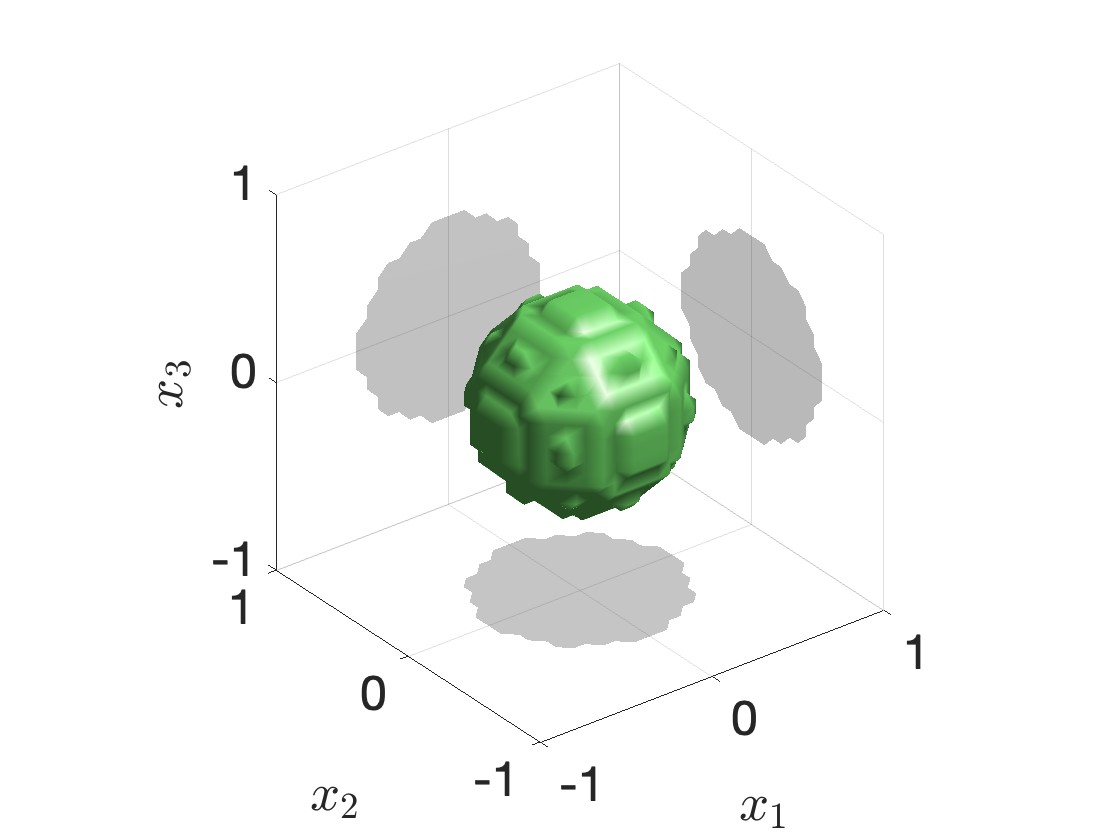}
            \caption{Iso-surface plot
            of (d)}
            \label{trueiso:sphere}
        \end{subfigure}
         \begin{subfigure}{0.325\textwidth}
            \includegraphics[scale = 0.14]{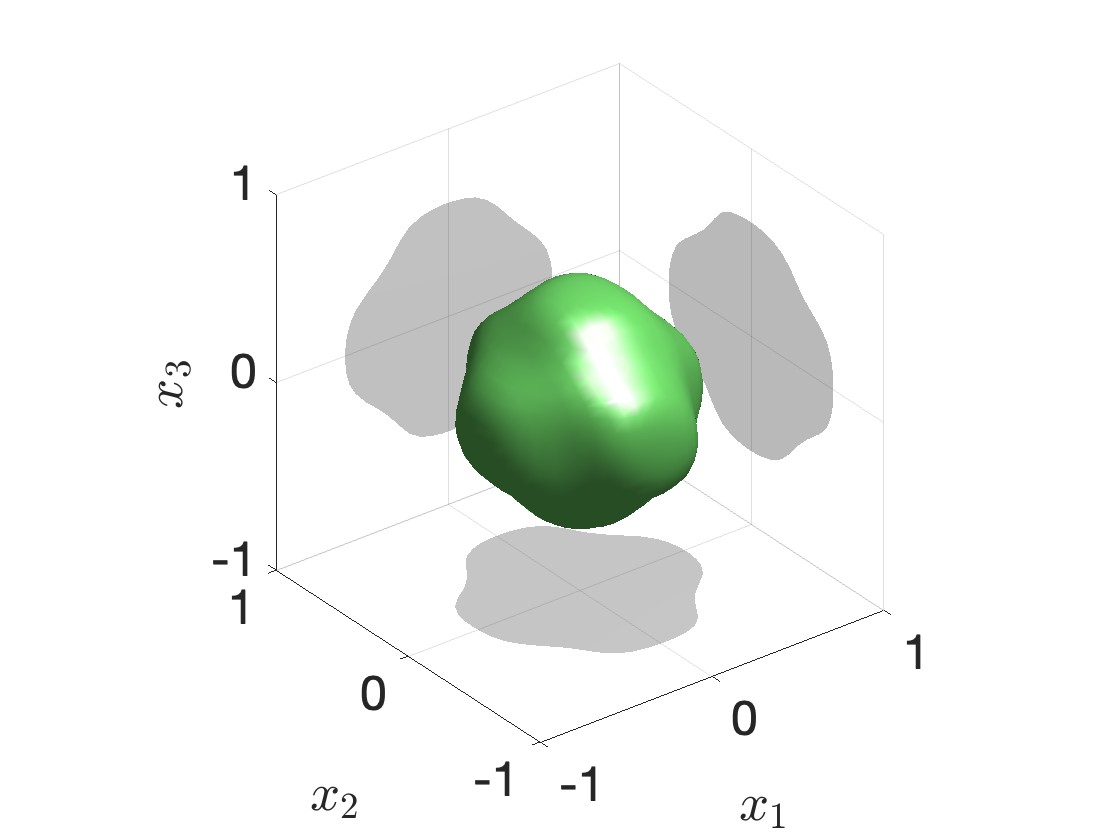}
            \caption{Iso-surface plot
            of (e)}
            \label{prediso:sphere}
        \end{subfigure}

         \end{center}

    \caption{Reconstruction results for the sphere-shaped source. 
Top row: exact Fourier coefficients \eqref{trueF:sphere} and the real \eqref{predFreal:sphere} and imaginary \eqref{predFimag:sphere} parts predicted by the proposed method. 
Bottom rows: cross-sectional views and corresponding iso-surface representations of the true source \eqref{truecross:sphere}, \eqref{trueiso:sphere}, and its reconstruction \eqref{predcross:sphere}, \eqref{prediso:sphere}.}
    \label{fig:new sphere-3d}
\end{figure}

\begin{figure}[H]
    \begin{center}
     \begin{subfigure}{0.3\textwidth}
            \includegraphics[scale = 0.14]{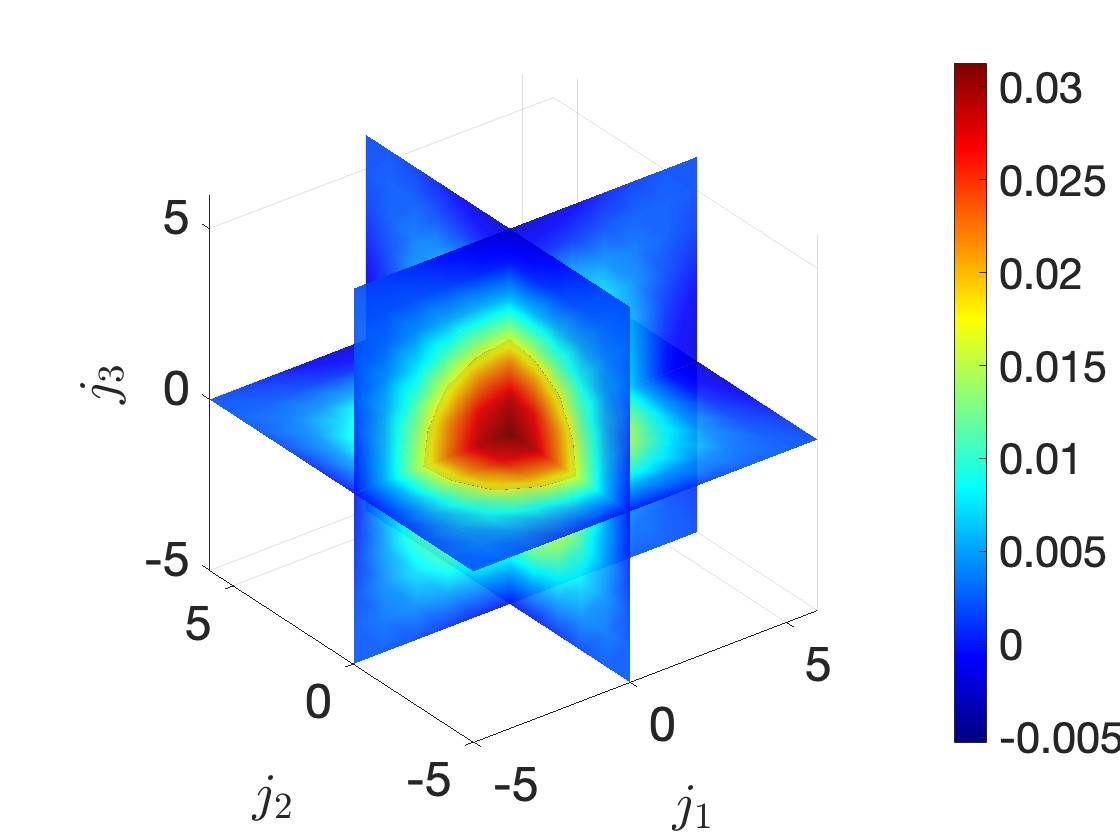}
             \caption{$\mathrm{Re}(\widehat{f}_{\text{per}})$ $(\mathrm{Im}(\widehat{f}_{\text{per}})=0)$}
            \label{trueF:pyramid}
        \end{subfigure}
        \hspace{0.02cm}
        \begin{subfigure}{0.3\textwidth}
            \includegraphics[scale = 0.14]{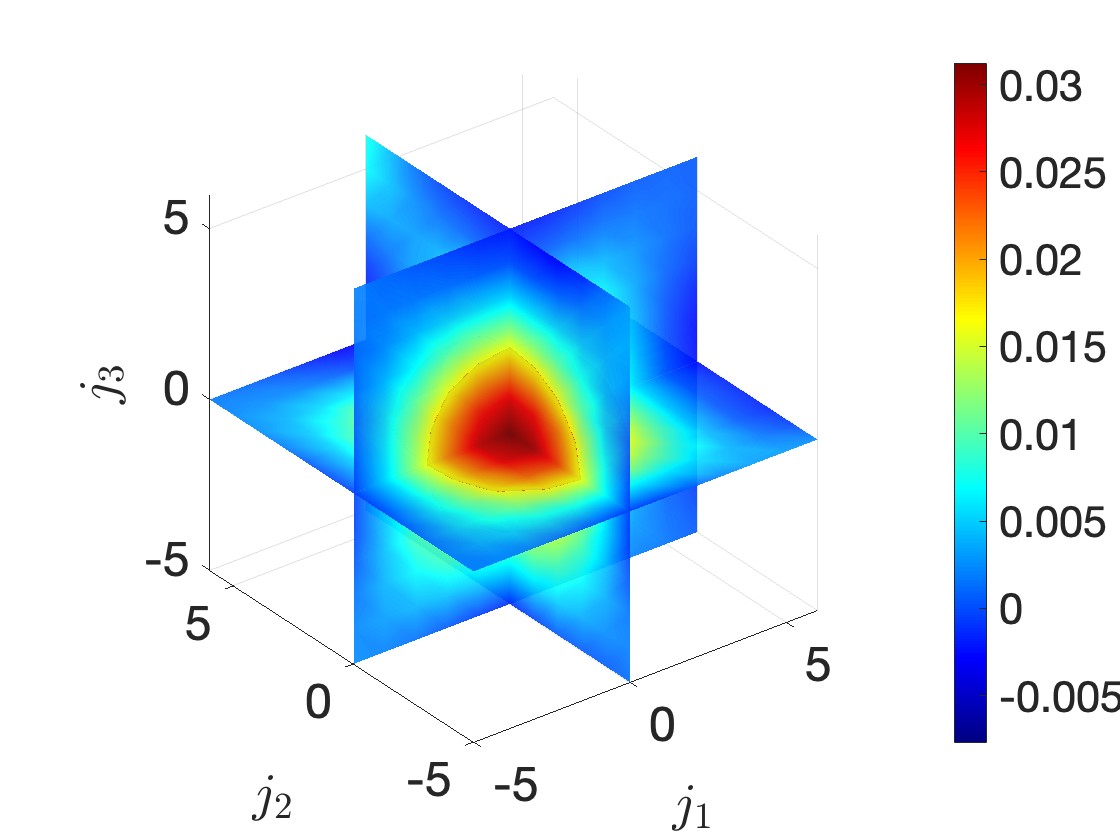}
            \caption{$\mathrm{Re}\left(\widehat{f}^{\ast}_{\Theta}\right)$}
            \label{predFreal:pyramid}
        \end{subfigure}
        \hspace{0.02cm}
        \begin{subfigure}{0.3\textwidth}
            \includegraphics[scale = 0.14]{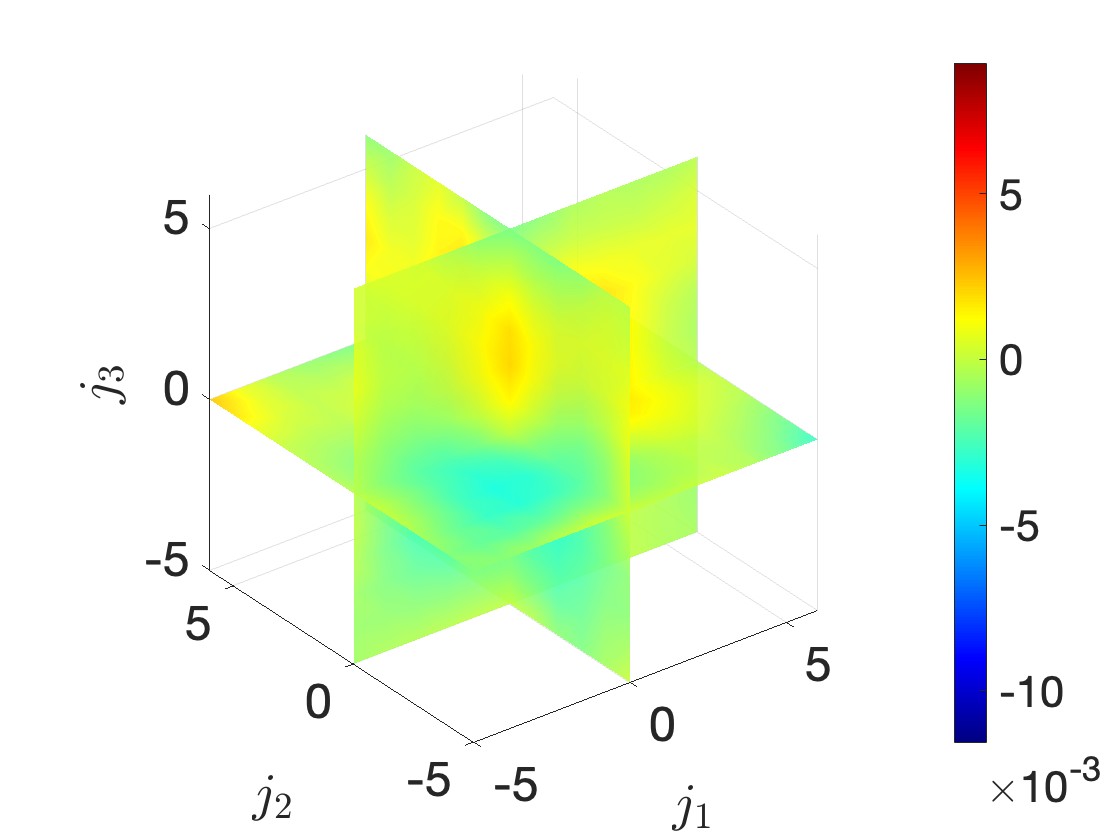}
            \caption{$\mathrm{Im}\left(\widehat{f}^{\ast}_{\Theta}\right)$}
            \label{predFimag:pyramid}
        \end{subfigure}
        
\begin{subfigure}{0.325\textwidth}
            \includegraphics[scale = 0.14]{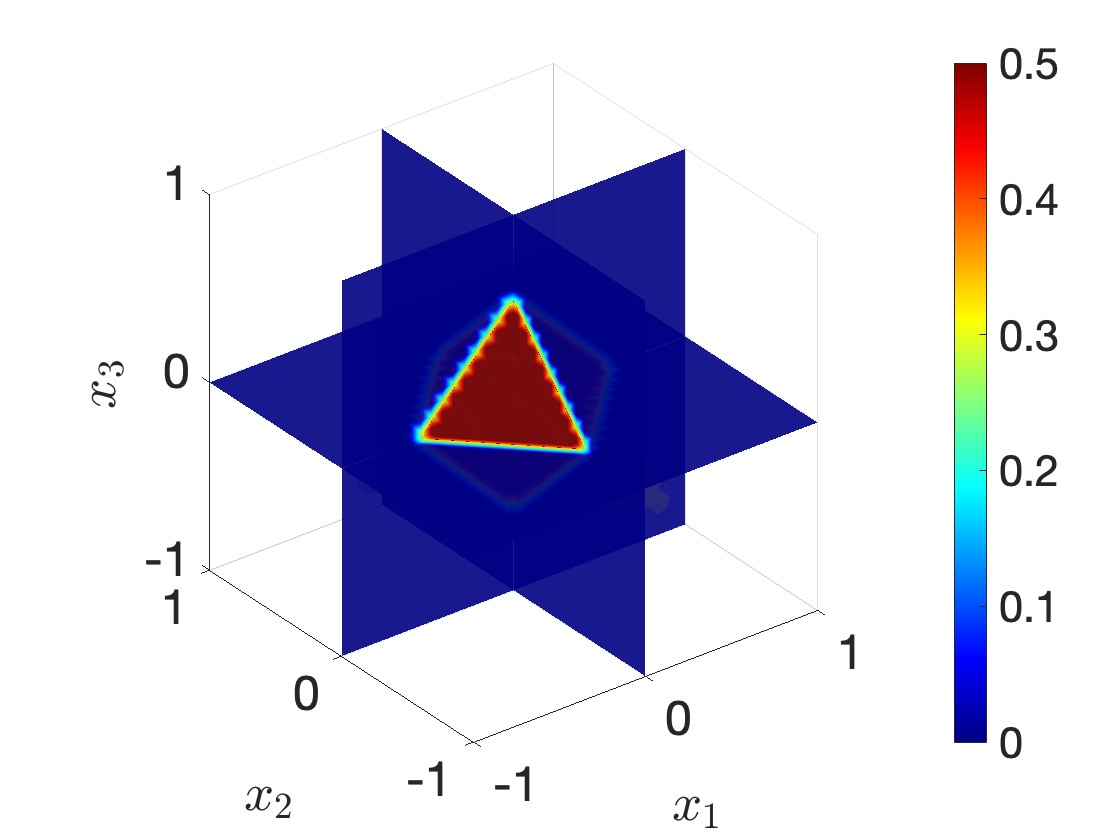}
            \caption{$f$}
            \label{truecross:pyramid}
        \end{subfigure}
        \begin{subfigure}{0.325\textwidth}
            \includegraphics[scale = 0.14]{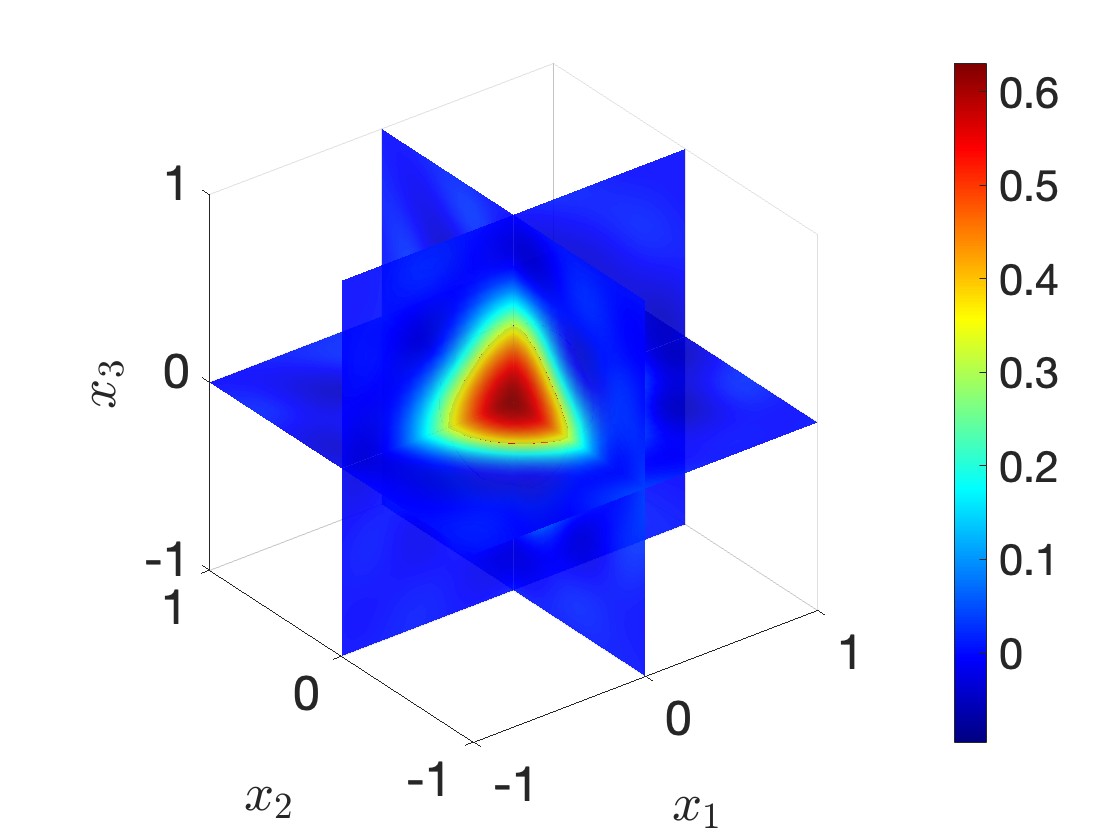}
            \caption{$f_{\text{comp}}$}
            \label{predcross:pyramid}
        \end{subfigure}

        \begin{subfigure}[b]{0.325\textwidth}
            \includegraphics[scale = 0.14]{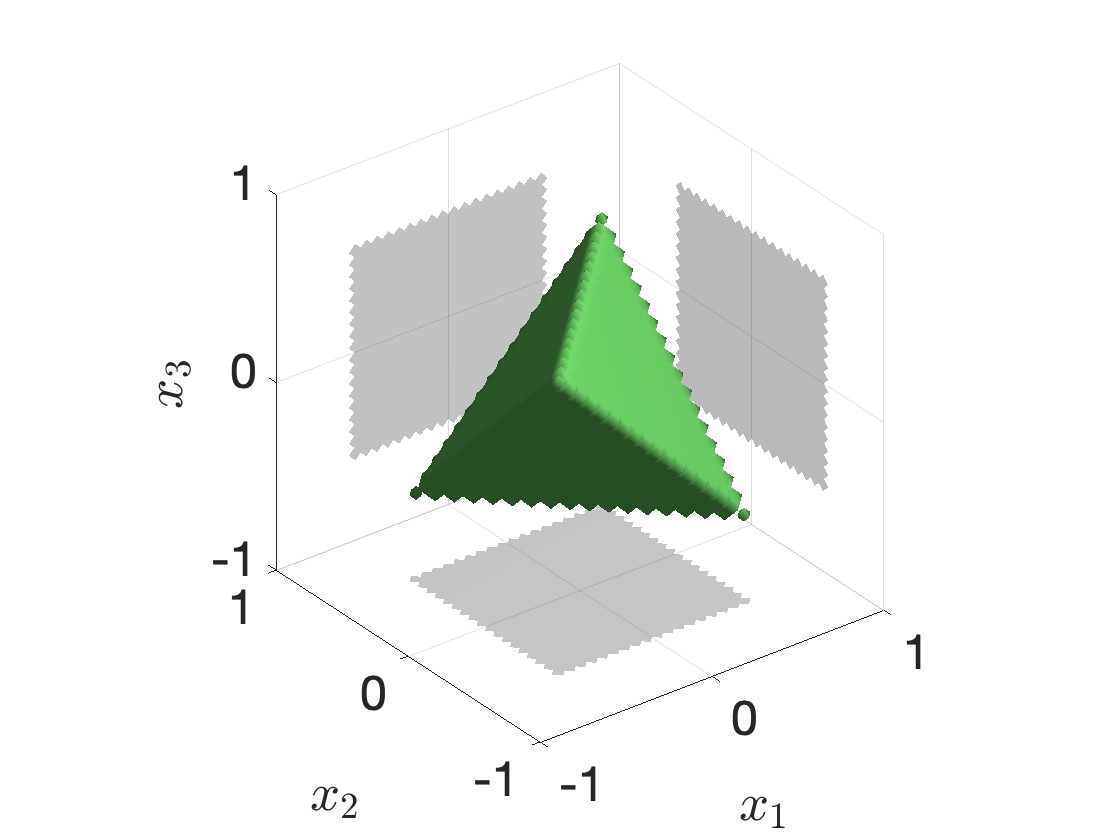}
            \caption{Iso-surface plot
            of (d)}
            \label{trueiso:pyramid}
        \end{subfigure}
        \begin{subfigure}{0.325\textwidth}
            \includegraphics[scale = 0.14]{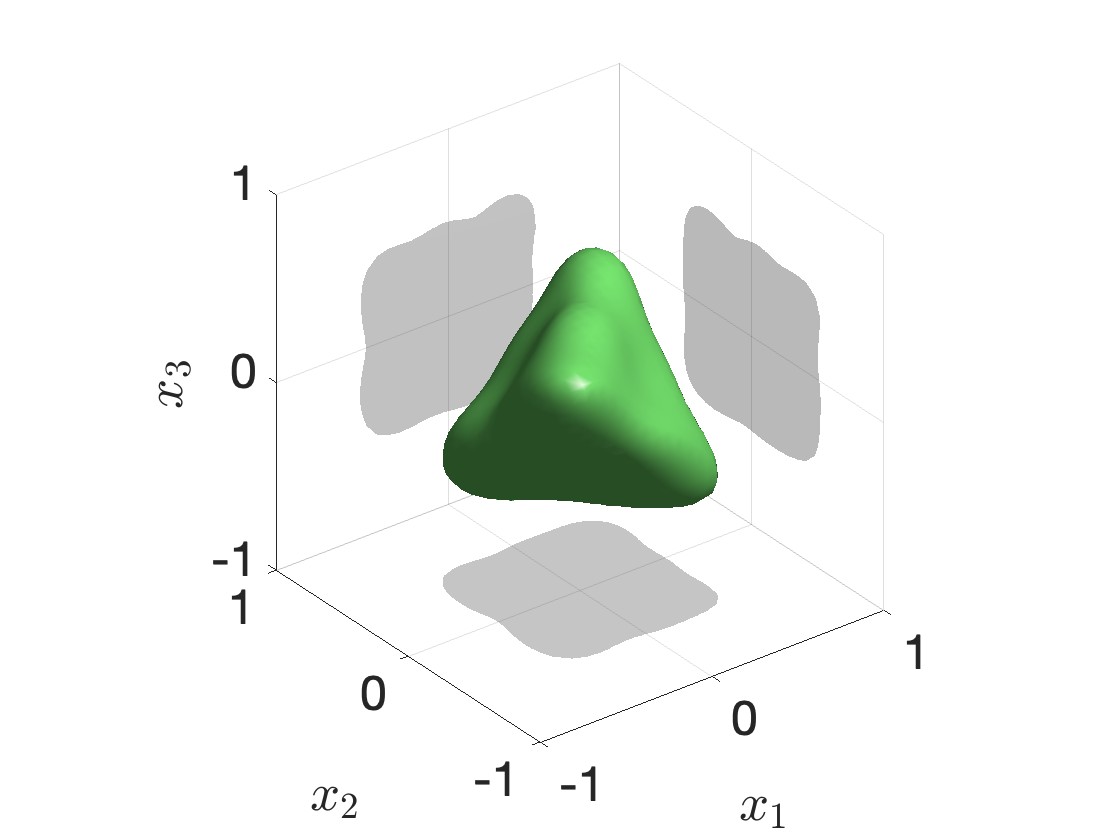}
            \caption{Iso-surface plot
            of (e)}
            \label{prediso:pyramid}
        \end{subfigure}

         \end{center}

    \caption{Reconstruction results for the pyramid-shaped source. 
Top row: exact Fourier coefficients \eqref{trueF:pyramid} and the real \eqref{predFreal:pyramid} and imaginary \eqref{predFimag:pyramid} parts predicted by the proposed method. 
Bottom rows: cross-sectional views and corresponding iso-surface representations of the true source \eqref{truecross:pyramid}, \eqref{trueiso:pyramid}, and its reconstruction \eqref{predcross:pyramid}, \eqref{prediso:pyramid}.}
    \label{fig:new pyramid-3d}
\end{figure}

\begin{table}[H]
    \centering
    \begin{tabular}{|c|c|c|c|c|}
    \hline
    Source &    $\mathcal{E}_{\text{Fourier}}\left[\widehat{f}_{\Theta}^{\ast}\right]$ & Computation time \\
    \hline
    Cube   & 29.14\%  & 29 min 59 sec \\
    \hline
    Sphere  & 24.17\%  & 27 min 18 sec \\
    \hline
    Pyramid   & 26.38\%  &  25 min 40 sec \\
    \hline
    \end{tabular}
    \caption{Relative error and computational time for $3$D reconstructions.}
    \label{table: new 3D table}
\end{table}

\section{Conclusion}\label{section6}

We have developed a novel unsupervised deep learning based algorithm to solve the two- and three-dimensional inverse source problems. The approach first employs an imaging function to obtain an initial approximation of the source geometry and location. Using this information, a Fourier-domain model equation is derived through a periodization technique and incorporated into the training of a model-informed neural network for parameter estimation. Furthermore, we established a theoretical lower bound for Fourier series truncation, enabling fast computation while maintaining high approximation accuracy. Numerical results demonstrate that the proposed method provides fast, accurate, and stable reconstructions, exhibits robustness to noise, and outperforms the conventional least-squares approach. Future work includes extending this framework to inverse acoustic and electromagnetic scattering problems. 
\vspace{0.3 cm}

\noindent  \textbf{Acknowledgment.} The research is partially supported by NSF Grant DMS-2243854.

\printbibliography[heading=bibintoc,title={References}]

\end{document}